\documentclass{aims}
\usepackage{amsmath}
\usepackage{caption}
\usepackage{makecell}
\usepackage{draftwatermark}
\SetWatermarkText{ACCEPTED}
\SetWatermarkScale{5}
  \usepackage{paralist}
  \usepackage{graphics} 
  \usepackage{epsfig} 
\usepackage{graphicx}  \usepackage{epstopdf}
 \usepackage[colorlinks=true]{hyperref}
\hypersetup{urlcolor=blue, citecolor=red}

\def\currentvolume{X}
 \def\currentissue{X}
  \def\currentyear{200X}
   \def\currentmonth{XX}
    \def\ppages{X--XX}
     \def\DOI{10.3934/xx.xx.xx.xx}

\newtheorem{theorem}{Theorem}[section]

\newtheorem{proposition}{Proposition}
\theoremstyle{definition}
\newtheorem{definition}[theorem]{Definition}
\newtheorem{assumption}[theorem]{Assumptions}
\newtheorem{algorithm}[theorem]{Algorithm}
\newtheorem{example}[theorem]{Example}
\newtheorem{remark}{Remark}

\newcommand{\ang}[1]{\langle #1 \rangle}
\newcommand{\bb}[1]{\lbrace #1 \rbrace}
\newcommand{\co}{\operatorname{co}}

\newcommand{\epi}{\operatorname{epi}}
\newcommand{\sign}{\operatorname{sign}}

\newcommand{\di}{\operatorname{d}}
\newcommand{\diam}{\operatorname{diam}}

\newcommand{\dH}{\di_H}

\newcommand{\inter}{\operatorname{int}}
\newcommand{\norm}[1]{\left\|#1\right\|}
\newcommand{\Orem}[1]{\mathcal{O}(#1)}
\newcommand{\rk}{\operatorname{rank}}

\newcommand{\RSU}{\mathcal{R}_\leq}
\newcommand{\scal}[2]{\langle {#1},{#2} \rangle}
\newcommand{\Y}{\operatorname{Y}}

\newcommand{\R}{{\mathbb R}}
\newcommand{\N}{{\mathbb N}}

\title[Construction of the minimum time function via reachable sets] 
      {Construction of the Minimum Time Function for Linear Systems
             Via Higher-Order Set-Valued Methods}

\author[Robert Baier and Thuy T. T. Le]{}


 \email{robert.baier@uni-bayreuth.de}
 \email{lethienthuy@gmail.com}

\thanks{The second author is supported by a PhD fellowship for foreign students at the Universit\`a di Padova funded by Fondazione CARIPARO. This paper was developed while the second author was visiting the Department of Mathematics of the University of Bayreuth.}

\thanks{$^*$ Corresponding author: Thuy T. T. Le}

\begin{document}
\maketitle

\centerline{\scshape Robert Baier}
\medskip
{\footnotesize
 \centerline{ Universit\"at Bayreuth, Mathematisches Institut}
   \centerline{95440 Bayreuth, Germany}
} 

\medskip

\centerline{\scshape Thuy T. T. Le$^*$}
\medskip
{\footnotesize
 \centerline{Otto von Guericke University Magdeburg, Department of Mathematics}
   \centerline{Universit\"atsplatz 2, 39106 Magdeburg, Germany}
}

\bigskip


\begin{abstract}
The paper is devoted to introducing an approach to compute the approximate minimum time function of control problems which is based on reachable set approximation and uses arithmetic operations for convex compact sets. In particular, in this paper the theoretical justification of the proposed approach is restricted to a class of linear control systems.   The error estimate of the fully discrete reachable set is provided by employing the Hausdorff distance to the continuous-time reachable set. The detailed procedure solving the corresponding discrete set-valued problem is described. Under standard assumptions, by means of convex analysis and knowledge of the regularity of the true minimum time
function, we estimate the error of its approximation. Higher-order discretization of the reachable set of the linear control problem can balance missing regularity (e.g., H\"older continuity) of the minimum time function for smoother problems. To illustrate the error estimates and to demonstrate differences to other numerical approaches  we provide a collection of numerical examples which either allow higher order of convergence with respect to time discretization or where the continuity of the minimum time function cannot be sufficiently granted, i.e.,~we study cases in which the minimum time function is H\"older continuous or even discontinuous.
\end{abstract}
\section{Introduction}

\ Reachable sets have attracted several mathematicians since longer
times both in theoretical and in numerical analysis.
The approaches for the numerical computation of reachable sets mainly split into
two classes, those for reachable sets up to a given time and the other ones 
for reachable sets at a given end time. We will give here only exemplary references,
since the literature is very rich (more references are given in an early version of this paper in \cite{BLe}). There are methods based on overestimation and 
underestimation of reachable sets based on ellipsoids \cite{KV}, zonotopes \cite{Alth,GlGM} or on approximating the reachable set with support functions resp.~supporting points \cite{BL,Alth}. Other popular and well-studied approaches involve level-set methods, semi-Lagrangian schemes and the computation of an associated Hamilton-Jacobi-Bellman equation, see~e.g.~\cite{BCD,BFZ,F} or are based on the viability concept~\cite{ABS-P} and the viability kernel algorithm~\cite{S-P}. Further methods \cite{BL,BLcham,BPhD} are set-valued generalizations of quadrature methods and Runge-Kutta methods initiated by the works~\cite{D2F,V_int,DF,W,DV}.

Here, we will focus on set-valued quadrature methods and set-valued Runge-Kutta methods
with the help of support functions or supporting points,
since they do not suffer on the wrapping effect or on an exploding number of vertices
and the error of restricting computations only for finitely many directions 
can be easily estimated. Furthermore, they belong to the most efficient and fast methods
(see~\cite[Sec.~3.1]{Alth}, \cite[Chap.~9, p.~128]{lG}) for linear control problems
to which we restrict the computation of the minimum time function $T(x)$.
We refer to~\cite{BPhD,BL,lG} (and references therein) for technical details on the numerical implementation, although we will lay out the main ideas of this approach for reader's convenience.

In optimal control theory the regularity of the minimum time functions is studied
intensively, see e.g.~in~\cite{CMW,CNN} and references therein. For the error estimates in this paper it will be essential to single out example classes for which the minimum time function is Lipschitz (no order reduction of the set-valued method) or H\"older-continuous with exponent $\frac{1}{2}$ (order reduction by the square root).
 
Minimum time functions are usually computed by solving the Hamilton-Jacobi-Bellman (HJB) equations and by the dynamic programming principle, see e.g.~\cite{BBZ,BFZ,BF1,BF2,BFS,CL,GL}. In this approach, the minimal requirement on the regularity of $T(x)$ is the continuity, see e.g.~\cite{BF1,CL,GL}. The solution of a HJB equation with suitable boundary conditions gives immediately -- after a transformation -- the minimum time function and its level sets provide a description of the reachable sets. A natural question occurring is whether it is also possible to do the other way around, i.e., to reconstruct the minimum time function $T(x)$ if knowing the reachable sets. One of the attempts was done in \cite{BBZ,BFZ}, where the approach is based on PDE solvers and on the reconstruction of the optimal control and solution via the value function.
On the other hand, our approach in this work is completely different. It is based on very efficient quadrature methods for convex reachable sets as described in Section~3.
 
In this article we present a novel approach for calculating the minimum time function.
The basic idea is to use set-valued methods for approximating reachable sets at a given
end time with computations based on support functions resp.~supporting points.
By reversing the time and start from the convex target as initial set we compute
the reachable sets for times on a (coarser) time grid. 
Due to the strictly expanding condition for reachable sets,
the corresponding end time is assigned to all boundary points of the computed 
reachable sets. Since we discretize in time and in space (by choosing a finite number
of outer normals for the computation of supporting points), the vertices of the polytopes
forming the fully discrete reachable sets are considered as data points of an irregular
triangulated domain. On this simplicial triangulation, a piecewise linear approximation
yields a fully discrete approximation of the minimum time function. 

The well-known interpolation error and the convergence results for the set-valued method can be applied to yield an easy-to-prove error estimate by taking into account the regularity of the minimum time function. It requires at least H\"older continuity and 
involves the maximal diameter of the simplices in the used triangulation. A second error estimate is proved without explicitely assuming the continuity of the minimum time function and depends only on the time interval between the computed (backward) reachable sets. The computation does not need the nonempty interior of the target set in contrary to the Hamilton-Jacobi-Bellman approach, for singletons the error estimate even improves. It is also able to compute discontinuous minimum time functions, since the underlying set-valued method can also compute lower-dimensional reachable sets. There is no explicit dependence of the algorithm and the error estimates on the smoothness of optimal solutions or controls. These results are devoted to reconstructing discrete 
optimal trajectories which reach a set of supporting points from a given target 
for a class of linear control problems and also proving the convergence of discrete 
optimal controls by the use of nonsmooth and variational analysis. 
The main tool is Attouch's theorem that allows to benefit from the convergence of the discrete reachable sets to the time-continuous one.

The plan of the article is as follows: in Section \ref{sec:preliminaries} we collect notations, definitions and basic properties of convex analysis, set operations, reachable sets and the minimum time function. The convexity of the reachable set for linear control problems and the characterization of its boundary via the level-set of the minimum time function is the basis for the algorithm formulated in the next section. In Section \ref{sec:construction}, we briefly introduce the reader to set-valued quadrature methods and Runge-Kutta methods and their implementation and discuss the convergence order for the fully discrete approximation of reachable sets at a given time both in time and in space.
In the next subsection we present the error estimate for the fully discrete minimum time function which depends on the regularity of the minimum time function and on the
convergence order of the underlying set-valued method. Another error estimate expresses
the error only on the time period between the calculated reachable sets. The last subsection discusses the construction of discrete optimal trajectories and convergence of discrete optimal controls. A series of accompaning examples can be found in Section \ref{sec:num_tests}. We compare the error of the minimum time function with respect to time and space discretization studying the influence of its regularity and of the smoothness of the support functions of corresponding set-valued integrands.
We first consider several linear examples with various target 
and control sets and study different levels of regularity of the corresponding
minimum time function. The nonlinear example in Subsection~\ref{subsec_nonlin} demonstrates that this approach is not restricted to the  class of linear control systems. 
 Although first numerical experiences are gathered there, its theoretical justification has to be
gained by a forthcoming paper. In Subsection~\ref{subsec_non_exp_prop} 
one example demonstrates the need of the strict expanding property of (union of) reachable
sets for characterizing boundary points of the reachable set via time-minimal points. The section ends with a collection of examples which either are more challenging for numerical 
calculations or partially violate our assumptions. Finally, a discussion of our approach and possible improvements can be found in Section~\ref{sec:concl}.

\section{Preliminaries}\label{sec:preliminaries}
In this  section we will recall some notations, definitions as well as basic knowledge of convex analysis and control theory for later use. Let $\mathcal{C}(\mathbb{R}^n)$ be the set of convex, compact,  nonempty  subsets of $\mathbb{R}^n$, 
 $\| \cdot \|$ be the Euclidean norm and $\scal{\cdot}{\cdot}$ the inner product in $\mathbb{R}^n$,
 $B_r(x_0)$ be  the closed (Euclidean) ball with radius $r>0$ centered at $x_0$  and $S_{n-1}$ be the unique sphere in $\mathbb{R}^n$. 
Let $A$ be a subset of $\mathbb{R}^n$, $M$ be an $n\times n$ real matrix, then $B_r(A):= \bigcup_{x\in A}B_r(x)$, $\norm{M}$ denotes the \emph{lub-norm}
 of $M$ with respect to $\|\cdot\|$, i.e., the spectral norm.
The \emph{convex hull}, the \emph{boundary}, the \emph{interior} and the \emph{diameter} of a set $A$ are signified by $\co(A),\,\partial A,\, \inter(A),\, \diam(A)$ respectively. We define the support function, the supporting points in a given direction and the set arithmetic operations as follows.
\begin{definition}
Let $A\in \mathcal{C}(\mathbb{R}^n),\, l \in  \R^n$. The \emph{support function} and the \emph{supporting face} of $A$ in the direction $l$ are defined as, respectively,
\begin{equation*}
\begin{aligned}
\delta^*(l,A):= \max_{x\in A} \,\scal{l}{x},\,\,
 \Y(l,A):=\bb{x\in A \colon  \scal{l}{x}=\delta^*(l,A)}.
\end{aligned}
\end{equation*}
An element of the supporting face is called \emph{supporting point}.
\end{definition}
Known properties of the convex hull, the support function and the supporting points when applied to the set operations introduced above can be found in ,e.g.\,~\cite[Chap.~0]{AC}, \cite[Sec.~4.6, 18.2]{ABS-P}, \cite{BPhD,lG,Alth}.
Especially, the convexity of the arithmetic set operations becomes obvious.
We also recall the definition of Hausdorff distance which is the main tool to measure the error of reachable set approximation.
\begin{definition}
Let $C,D\in   \mathcal{C}(\mathbb{R}^n),\,x\in \mathbb{R}^n$. Then the \emph{distance function} from $x$ to $D$  is $\di(x,D):=\min_{d\in D} \norm{x-d}$
and the Hausdorff distance between $C$ and $D$ is defined as
\begin{align*}
\dH(C,D)&:=\max \bb{\max_{x\in C} \di(x,D),\max_{y\in D} \di(y,C)}.
\end{align*}
\end{definition}

Now we will recall some basic notations of control theory, see e.g.,~\cite[Chap.~IV]{BCD} for more details. Consider the following linear time-variant control dynamics in $ \mathbb{R}^n$
\begin{equation}\label{LCDyn}
\begin{cases}
\begin{array}{r@{\,}l@{\quad}l}
\dot y(t) & =A(t)y(t)+B(t)u(t)
  & \text{ for a.e. } t \in [t_0, \infty),  \\
 u(t) & \in U 
  & \text{ for a.e. } t \in [t_0, \infty), \\
y(t_0) & = y_0. &
\end{array}
\end{cases}
\end{equation}
The coefficients $ A(t), B(t) $ are $n\times n$ and $n\times m$ matrices respectively,  $y_0 \in \mathbb{R}^n$ is the initial value,
 $U\in \mathcal{C}(\R^m)$ is the set of control values. Under standard assumptions, the existence and uniqueness of \eqref{LCDyn} are guaranteed for any measurable function $u(\cdot)$ and any $y_0 \in \mathbb{R}^n$. Let $\mathcal{S}\subset \mathbb{R}^n$, a nonempty compact set, be the 
\emph{target} and 
$$
\mathcal{U}:=\bb{ u \colon [t_0,\infty) \rightarrow U \text{ measurable}},
$$
the set of \emph{admissible controls}
and $y(t,y_0,u)$ be the solution of \eqref{LCDyn}. 
We define the \emph{minimum time starting from $y_0 \in \mathbb{R}^n$ to reach the target $\mathcal{S}$}  for some $u \in \mathcal{U}$
$$
t(y_0,u)=\min \,\bb{t\ge t_0:\ y(t,y_0,u)\in \mathcal{S}}\leq \infty.
$$
The \emph{minimum time function to reach $\mathcal{S}$ from $y_0$} is defined as
$
T(y_0)=\inf_{u\in \mathcal{U}} \,\bb{t(y_0,u)},
$
see e.g.,~\cite[Sec.~IV.1]{BCD}.
We also define the 
\emph{reachable sets for fixed end time} $t> t_0$, \emph{up to time $t$} 
resp.~\emph{up to a finite time} as follows:
 \begin{align*}
    &\mathcal{R}(t)  := \bb{y_0 \in \R^n: \textit{ there exists } u\in \mathcal{U},\,y(t,y_0,u)\in \mathcal{S}},\\
    &\RSU(t)   :=  \mbox{}  \bb{y_0 \in \R^n: \textit{ there exists } u\in \mathcal{U},\,y(s,y_0,u)\in \mathcal{S} \text{ for some } s\in [t_0,t]}\\& \quad\quad\quad=\bigcup_{\substack s \in [t_0,t]}\mathcal{R}(s), \\
    &\mathcal{R}  := \bb{y_0 \in \R^n: \textit{ there exists some finite time $ t \geq  t_0$ with }y_0 \in \mathcal{R}(t) } = \bigcup_{t\in [t_0,\infty)} \mathcal{R}(t).
 \end{align*}
By definition 
 \begin{equation} \label{eq:reach_leq_sub_level_set}
    \RSU(t) = \bb{y_0\in \mathbb{R}^n \colon T(y_0)\le t}
 \end{equation}
is a sublevel set
of the minimum time function, while
for a given maximal time $ t_f > t_0$ and some $t \in  I :=  [t_0, t_f]$,
$\mathcal{R}(t)$ is the set of points \emph{reachable from the target in time} $ t $ 
by the \emph{time-reversed system} 
 \begin{align} \label{eq:time_rev_cp}
    \dot{y}(t) & =\bar A(t)y(t) + \bar B(t)u(t), \\
    y(t_0) & \in \mathcal{S}, \label{InCond}
 \end{align}
 where $\bar A(t):= -A(t_0+ t_f-t),\,\bar B(t):=-B(t_0+ t_f-t)$ for shortening notations.
 In other words, $\mathcal{R}(t)$ equals  the set of starting points from 
which the system can reach the target in time $ t $. 
Sometimes $\mathcal{R}(t)$ is called the \emph{backward reachable set} which is also
considered in~\cite{BBZ} for computing the minimum time function by solving
a Hamilton-Jacobi-Bellman equation.

The following standing hypotheses are assumed to be fulfilled in the sequel.
\begin{assumption}\label{standassum}\mbox{}
\begin{enumerate}
\item[ (i) ] $ A(t),\,B(t) $ are $ n \times n $, $ n \times m $ real-valued matrices defining integrable functions on any compact interval of $[t_0,\infty) $. 
\item[ (ii) ] The control set $U\subset \mathbb{R}^m$ is  convex, compact and nonempty, i.e., $U \in \mathcal{C}(\R^m)$.
\item[ (iii) ] The target set $\mathcal{S}\in \mathbb{R}^n$ is convex, compact and nonempty, i.e., $\mathcal{S} \in \mathcal{C}(\R^n)$. \\
   Especially, the target set can be a singleton.
\item[ (iv) ] $\mathcal{R}(t)$ is \emph{strictly expanding} on the compact interval $[t_0, t_f]$, i.e., $\mathcal{R}(t_1) \subset \inter \mathcal{R}(t_2)$ for all $t_0\le t_1<t_2\le  t_f$.
\end{enumerate}
\end{assumption}
\begin{remark}
The reader can find sufficient conditions for Assumption \ref{standassum}(iv)  for  $\mathcal{S}=\bb{0}$ in \cite[Chap.~17]{HL}, \cite[Sec.~2.2--2.3]{LM}.  Under this assumption, it is obvious that
\label{rem:strict_expand}
\begin{equation*}
\begin{aligned}
 \mathcal{R}(t)= \RSU(t) .
\end{aligned}
\end{equation*}
\end{remark}
Under our standard hypotheses, the control problem~\eqref{eq:time_rev_cp} can 
equivalently be replaced by
the following linear differential inclusion
\begin{equation}\label{InLCDyn}
\dot y(t)\in \bar A(t)y(t)+\bar B(t)U \ \ \text{ for a.e. }  t \in [t_0, \infty)
\end{equation}
with absolutely continuous solutions $y(\cdot)$ (see \cite[Appendix~A.4]{Tol}). 
All the solutions of \eqref{InCond}--\eqref{InLCDyn} are represented as
\begin{equation*}
y(t)=\Phi(t,t_0)y_0+\int_{t_0}^{t} \Phi(t,s)\bar B(s)u(s)ds
\end{equation*}
for all $y_0 \in \mathcal{S},\, u\in \mathcal{U}$, and $t_0\le t<\infty$, where $\Phi(t,s)$ is the \emph{fundamental solution matrix} of the homogeneous system 
\begin{equation}\label{fundSol}
\dot y(t)=\bar A(t)y(t), 
\end{equation}
with  $\Phi(s,s)=I_n$, the $n\times n$ identity matrix. 
Using the Minkowski addition and the Aumann's integral \cite{Aum}, the reachable set can be described by means of Aumann's integral as follows
\begin{equation}\label{Rt}
\mathcal{R}(t)=\Phi(t,t_0)\mathcal{S}+\int_{t_0}^{t} \Phi(t,s)\bar B(s)Uds.
\end{equation}
For time-invariant systems, 
i.e., $ \bar A(t) = \bar A $, we have $\Phi(t,t_0)=e^{\bar A(t-t_0)}$. 

For the linear control system, under Assumptions \ref{standassum}(i)--(iii), \eqref{LCDyn} the reachable set at a fixed end time is
convex which allows to apply support functions or supporting points for its approximation.
Furthermore, the reachable sets change continuously with respect to the end time. 

The following proposition provides the connection between $\mathcal{R}(t)$ and the level set of $T(\cdot)$ at time $t$ which is essential for this approach. 
We will benefit from
the sublevel representation in~\eqref{eq:reach_leq_sub_level_set}.
 The result is related to~\cite[Theorem~2.3]{BBZ}, where the minimum time function at $x$ is 
the minimum for which $x$ lies on a zero-level set bounding the backward reachable set.
\begin{proposition}
Let Assumption \ref{standassum} be fulfilled and $t > t_0$. Then 
\label{prop:bd_descr_monotone_case_w_level_set}
\begin{equation}\label{ReachLevel}
\partial \mathcal{R}(t)=\bb{y_0  \in \mathbb{R}^n \colon T(y_0)=t}.
\end{equation}
\end{proposition}
\begin{proof}
"$\subset$":
Assume that there exists $x\in \partial \mathcal{R}(t)$ with $x\notin \bb{y_0\in \mathbb{R}^n \colon T(y_0)=t}$. 
Clearly, $x \in \RSU(t)$ and~\eqref{eq:reach_leq_sub_level_set} shows that $T(x) \leq t$.
By definition there exists
$s \in [t_0,t]$ with $x \in \mathcal{R}(s)$. Assuming $s <t$ we get the
contradiction $x \in \mathcal{R}(s) \subset \inter \mathcal{R}(t)$ from Assumption~\ref{standassum}(iv).\\
"$\supset$":
Assume that there exists $x\in \bb{y_0\in \mathbb{R}^n \colon T(y_0)=t}$ (i.e., $T(x)=t$) be such that $x\notin \partial \mathcal{R}(t)$. Since 
$x\in \mathcal{R}(t)$ by~\eqref{eq:reach_leq_sub_level_set} and we assume that $x\notin \partial \mathcal{R}(t)$, then $x\in \inter( \mathcal{R}(t))$. 

Hence, there exists $\varepsilon > 0$ with
    \begin{align*}
       x + \varepsilon B_1(0) \subset \mathcal{R}(t).
    \end{align*}
   The continuity of $\mathcal{R}(\cdot)$ ensures for $t_1 \in [t - \delta, t+\delta] 
   \cap I$ that
\begin{align*}
   \dH(\mathcal{R}(t), \mathcal{R}(t_1)) & \leq \frac{\varepsilon}{2}.
\end{align*}
   Hence,
\begin{align*}
   x + \varepsilon B_1(0) \subset \mathcal{R}(t) & \subset \mathcal{R}(t_1)
     + \frac{\varepsilon}{2} B_1(0).
\end{align*}
   The order cancellation law in~\cite[Theorem~3.2.1]{PalUrb}
   can be applied, since $\mathcal{R}(t_1)$ is convex and all sets are compact. Therefore,
    \begin{align*}
       (x + \frac{\varepsilon}{2} B_1(0)) + \frac{\varepsilon}{2} B_1(0) 
         & \subset \mathcal{R}(t_1) + \frac{\varepsilon}{2} B_1(0)
    \end{align*}
    which implies $ x + \frac{\varepsilon}{2} B_1(0) \subset \mathcal{R}(t_1) $.

   Hence, $x \in \inter( \mathcal{R}(t_1))$ with $t_1 < t$ so that $T(x) \leq t_1 < t$
   which is again a contradiction. Therefore, $\bb{y_0\in \mathbb{R}^n \colon T(y_0)=t} \subset \partial \mathcal{R}(t)$. 
   The proof is completed.
\end{proof}

In the previous characterization of the boundary of the reachable set at fixed end time
the assumption of monotonicity of the reachable sets played a crucial role. 
As stated in Remark~\ref{rem:strict_expand}, Assumption \ref{standassum}(iv) also guarantees that the union of reachable sets 
coincides with the reachable set at the largest end time and is trivially convex.
If we drop this assumption, 
we can only characterize the boundary of the \emph{union} of reachable sets up to a time
under relaxing the expanding property~(iv) while demanding convexity as can be seen in the following proposition.

\begin{proposition}
Let $t > t_0$, Assumptions \ref{standassum}(i)--(iii) and Assumption
 \begin{quote}
    (iv)' \quad $\RSU(t)$ has convex images and is strictly expanding on the compact interval $[t_0, t_f]$, i.e.,
     $
        \RSU(t_1) \subset \inter \RSU(t_2) \quad \text{for all } t_0\le t_1<t_2 \le  t_f.
     $
 \end{quote} holds. Then
\label{prop:bd_descr_w_level_set}
\begin{equation} \label{equ:bd_union_reach_sets}
  \partial \RSU(t)
  = \bb{x\in \R^n \colon T(x)=t}.
\end{equation}
\end{proposition}
 \begin{proof}
  \emph{The proof can be found in~\cite[Proposition~7.1.4]{Le}.}
\end{proof}

\begin{remark}\label{rem:inclusion}
   Assumption~(iv)' implies that the considered system is small-time controllable, see~\cite[Chap.~IV, Definition 1.1]{BCD}. Moreover, under the assumption of small-time controllability the nonemptiness of the interior of $\mathcal{R}$ and the continuity of the minimum time
function in $\mathcal{R}$ are consequences, see~\cite[Chap.~IV, Propositions~1.2,~1.6]{BCD}.
Assumption~(iv)' is essentially weaker than~(iv), since the convexity of $\RSU(t)$
and the strict expandedness of $\RSU(\cdot)$ follow by Remark~\ref{rem:strict_expand}. 
The inclusion for $\RSU(\cdot)$ in this assumption is equivalent to small-time controllability
(STC) for time-invariant systems, sufficient conditions for STC in this case via generalized
Petrov and second-order conditions are discussed in~\cite{AM}. 
Under one of these two conditions the minimal time function is either continuous or
H\"older continuous with exponent $\frac{1}{2}$. 
Extensions of the continuity property
to $\varphi$-convexity can be found in~\cite{CMW}.

In the previous proposition we can allow that $\RSU(t)$ is lower-dimensional and are 
still able to prove the inclusion "$\supset$" in~\eqref{equ:bd_union_reach_sets}, 
since the interior of $\RSU(t)$ would be empty 
and $x$ cannot lie in the interior which also creates the (wanted) contradiction.

For the other inclusion "$\subset$" the nonemptiness of the interior of $\mathcal{R}(t)$
in Proposition~\ref{prop:bd_descr_monotone_case_w_level_set}
resp.~the one of $\RSU(t)$ in Proposition~\ref{prop:bd_descr_w_level_set} is essential. Therefore, the expanding property
in~Assumptions~(iv) resp.~(iv)' cannot be relaxed by assuming only monotonicity in the sense
 \begin{align} \label{ex:relaxed_expand}
    \mathcal{R}(s) \subset \mathcal{R}(t) \quad\text{or}\quad
    \RSU(s) \subset \RSU(t)
 \end{align}
for $s < t$ as Example \ref{ex:counter_ex_1} shows.
\end{remark}


\section{Approximation of the minimum time function}\label{sec:construction}
\subsection{Set-valued discretization methods}\label{subsec:sv_discr_meth}
Consider the linear control dynamics \eqref{LCDyn}. For a given $x\in \mathbb{R}^n$, the problem of computing approximately the minimum time $T(x)$ to reach $\mathcal{S}$ by following the dynamics $\eqref{LCDyn}$ is deeply investigated in literature. It was usually obtained  by solving the associated discrete Hamilton-Jacobi-Bellman equation (HJB), see, for instance, \cite{BF1,F,CL,GL}. Neglecting the space discretization we obtain an approximation of $T(x)$. In this paper, we will introduce another approach to treat this problem based on approximation of the reachable set of the corresponding linear differential inclusion. The approximate minimum time function is not derived from the PDE solver, but from iterative set-valued methods or direct discretization of control problems.

Our aim now is to compute $\mathcal{R}(t)$ numerically up to a maximal time $t_f$ based on the representation~\eqref{Rt} by means of set-valued methods to approximate  Aumann's integral. There are many approaches to achieving this goal. We will describe three known options for discretizing the reachable set which are used in the following. 

Consider for simplicity of notations an equidistant grid over the interval $I=[t_0,t_f]$ with $N$ subintervals, step size $h = \frac{t_f - t_0}{N}$ and grid points $t_i=t_0+ih$, $i=0,\ldots,N$.
\begin{enumerate}
\item[(I)] Set-valued quadrature methods with the exact knowledge of the fundamental solution matrix of \eqref{fundSol} (see e.g.,~\cite{V_int,D2F,BL}, \cite[Sec.~2.2]{BPhD}): as in the pointwise case, we replace the integral $\int_{t_0}^{t} \Phi(t,s)\bar B(s)Uds$ by some quadrature scheme of order $p$ with non-negative weights.
   Therefore, \eqref{Rt} is approximated by
\begin{equation} \label{eq:sv_quad_meth_global}
\mathcal{R}_h(t_N)=\Phi(t_N,t_0)\mathcal{S}+h \sum_{i=0}^{N}c_i \Phi(t_N,t_i)\bar B(t_i)U
\end{equation}
with weights $c_i\ge 0,\,i=0,\ldots,N$. Moreover, the following error estimate holds:
\begin{equation*}
\dH(\int_{t_0}^{t_N} \Phi(t_N,s)\bar B(s)Uds,h \sum_{i=0}^{N}c_i \Phi(t_N,t_i)\bar B(t_i)U)\le Ch^p.
\end{equation*}
\item[(II)] Set-valued combination methods (see e.g.,~\cite{BL}, \cite[Sec.~2.3]{BPhD}): we replace $\Phi(t_N,t_i)$ in method $ (I) $ by its approximation (e.g.,~via ODE solvers of the corresponding matrix equation) such that 
\begin{enumerate}
\item[a)]$\Phi_h(t_{m+n},t_0)=\Phi_h(t_{m+n},t_m)\Phi_h(t_m,t_0)$ for all $m \in \{0,\ldots,N\}$, $n \in \{0,\ldots,N-m\}$.   
\item[b)]$\sup_{0\le i\le N} \norm{\Phi(t_N,t_i)-\Phi_h(t_N,t_i)}\le Ch^p.$
\end{enumerate}
Then, the discrete reachable sets is globally resp.~locally recursively represented as
\begin{align}
\mathcal{R}_h(t_N) & =\Phi_h(t_N,t_0)\mathcal{S}+h\sum_{i=0}^{N}c_i \Phi_h(t_N,t_i)\bar B(t_i)U, \label{eq:sv_comb_meth_global}\\
\mathcal{R}_h(t_0) & = \mathcal{S}, \label{eq:sv_quad_meth_local_start} 
\\ 
\mathcal{R}_h(t_{i+1}) & =\Phi_h(t_{i+1},t_{i})\mathcal{R}_h(t_{i})+ h \sum_{j=0}^{1} \widetilde{c}_{ij} \Phi_h(t_{i+1},t_{i+j})\bar B(t_{i + j})U.
\label{semigroupcombmethRh}
\end{align}
\item[(III)] Set-valued Runge-Kutta methods (see e.g.,~\cite{DF,W,V,B}): \\
We can approximate \eqref{InLCDyn} by set-valued analogues of Runge-Kutta schemes.
The discrete reachable set is computed recursively with the starting condition \eqref{eq:sv_quad_meth_local_start} for the set-valued Euler scheme (see e.g.,~\cite{DF}) as
\begin{equation} \label{eq:sv_euler_rec}
\mathcal{R}_h(t_{i+1})=  \Phi_h(t_{i+1},t_i) \mathcal{R}_h(t_i)+hB(t_i)U,
\end{equation}
 for the set-valued Heun's scheme  with piecewise constant selections  
 (see e.g., \cite{V}) as   
\begin{equation} \label{eq:sv_heun_rec}
\mathcal{R}_h(t_{i+1})= \Phi_h(t_{i+1},t_i) \mathcal{R}_h(t_i)
+\frac{h}{2}\Big( (I + h A(t_{i+1}))B(t_i)+B(t_{i+1})\Big)U.
\end{equation}
\end{enumerate}
An example of $R_h$ with different choices of numerical methods is as follows.
\begin{equation*}
R_h(t_{j+1}) = \begin{cases}e^{hA} R_h(t_j) + h e^{hA} \overline{B} U & \mbox{set-valued Riemann sum}, \\
(I+hA) R_h(t_j) + h (I+hA) \overline{B} U & \mbox{Riemann sum combined with Euler}, \\
(I+hA)R_h(t_j) + h \overline{B} U &  \mbox{set-valued Euler}.
\end{cases}
\end{equation*}
The purpose of this paper is not to focus on the set-valued numerical schemes themselves, but on the approximative construction of $T(\cdot)$. Thus, without loss of generality, we mainly utilize the scheme described in (II) to present our idea from now on. In practice, there are several strategies in control problems to discretize the set of controls $\mathcal{U}$, see e.g.,~\cite{BBCG}. Here  we choose a  \emph{piecewise constant} approximation $\mathcal{U}_h$ for the sake of simplicity which corresponds to use only one selection on the subinterval $[t_i,t_{i+1}]$ in the corresponding set-valued quadrature method. 
Depending on the choice of the method, we can find a subset $\mathcal{U}_h$ of $ U $, usually the piecewise constant controls so that in the case (II), for instance, we have
\begin{align*}
\mathcal{R}_h(t_N)=\bb{y\in \mathbb{R}^n \colon \text{ there exists a piecewise constant control } u_h \in \mathcal{U}_h \text{ and } y_0 \in \mathcal{S}\\
\text{ such that } y=\Phi_h(t_N,t_0)y_0+h\sum_{i=0}^{N}c_i \Phi_h(t_N,t_i)\bar B(t_i)u_h(t_i)},
\end{align*}
or equivalently
$
\mathcal{R}_h(t_N)=\Phi_h(t_N,t_0)\mathcal{S}+h\sum_{i=0}^{N}c_i \Phi_h(t_N,t_i)\bar B(t_i)U.
$
We set
\begin{equation*}
t_h(y_0,y,u_h) = \min \bb{ t_n \colon n\in \mathbb{N},\,\, y=\Phi_h(t_n,t_0)y_0+h\sum_{i=0}^{n}c_i \Phi_h(t_n,t_i)\bar B(t_i) u_h(t_i) }
\end{equation*}
for some  $y\in \mathbb{R}^n,\,y_0\in \mathcal{S}$ and a piecewise constant grid function $u_h$ with $u_h(t_i) = u_i \in U$, $i=0,\ldots,n$. If there does not exist such a grid control $u_h$ which reaches $y$ from $y_0$ by the corresponding discrete trajectory, $t_h(y_0,y,u_h)=\infty$. Then the discrete minimum time function $T_h(\cdot)$ is defined as
\begin{equation*}
T_h(y)=\min\limits_{\substack{ u_h\in \mathcal{U}_h\\[0.3ex] y_0\in \mathcal{S}} }\, t_h(y_0,y,u_h).
\end{equation*}
Notice that the definitions of $\mathcal{R}_h$ and $t_h$  for the remaining cases (I) and (III) can be derived in a similar way by using the corresponding expressions of $y$.
\begin{proposition}
In all of the constructions (I)--(III) described above, $\mathcal{R}_h(t_N)$ is a convex, compact and nonempty set.
\end{proposition}
\begin{proof}
The key idea of the  proof of this proposition is to employ the linearity of \eqref{InLCDyn}, in conjunction with the convexity of $\mathcal{S},\,U$ and the arithmetic operations for convex sets. In particular, it follows analogously to the proof of \cite[Proposition~3.3]{BBCG}.
\end{proof}
\begin{theorem}\label{dHRRh}
Consider the linear control problem \eqref{InCond}--\eqref{InLCDyn}. Assume that the set-valued quadrature method and the ODE solver have the same order $p$. Furthermore, assume that $\bar A(\cdot)$ and $\delta^*(l,\Phi(t_f,\cdot)\bar B(\cdot)U)$ have absolutely continuous $(p-2)$-nd derivative, the $(p-1)$-st derivative is of bounded variation uniformly with respect to all $l\in S_{n-1}$ and $\sum_{i=0}^{N}c_i\norm{B(t_i)U}$ is uniformly bounded for $N\in \mathbb{N}$. Then
\begin{equation}
\dH(\mathcal{R}(t_N),\mathcal{R}_h(t_N))\le Ch^p,
\end{equation}
where $C$ is a non-negative constant.
\end{theorem}
\begin{proof}
See \cite[Theorem 3.2]{BL}.
\end{proof}
\begin{remark}
For $p=2$ the requirements of Theorem \ref{dHRRh} are fulfilled if $A(\cdot),\,B(\cdot)$ are absolutely continuous  and $A'(\cdot),\,B'(\cdot)$ are bounded variation (see~\cite{DV}, \cite[Secs.~1.6, 2.3]{BPhD}). 
\end{remark}
The next subsection is devoted to the full discretization of the reachable set, i.e., we consider the space discretization as well. Since we will work with supporting points, we do this implicitly by discretizing the set $S_{n-1}$ of normed directions.
 This error will be adapted to the error  of the  set-valued  numerical scheme caused by the time discretization  to preserve its order of convergence with respect to time step size as stated in Theorem~\ref{dHRRh}. Then we will describe in detail the procedure to construct the graph of the minimum time function based on the approximation of the reachable sets.  We will also provide the corresponding overall   error estimate.

\subsection{Implementation and error estimate of the reachable set approximation}\label{subsec:Algorithm}
 For a particular problem, according to its smoothness in an appropriate sense we are first able to choose a difference method with a suitable order, say $O(h^p)$ for some $p>0$, to solve \eqref{fundSol} numerically effectively, for instance Euler scheme, Heun's scheme or Runge-Kutta scheme. Then we approximate Aumann's integral in~\eqref{Rt} by a quadrature formula with the same order, for instance Riemann sum, trapezoid rule, or Simpson's rule to obtain the discrete scheme of the global order $O(h^p)$. 

We implement the set arithmetic operations in~\eqref{semigroupcombmethRh} 
only approximately as indicated in \cite[Proposition~3.4]{BBCG}
and work with finitely many normed directions 
\begin{equation} \label{eq:discr_unit_spheres}
    \begin{array}{r@{\,}l@{\,}r@{\,}l@{\,}r@{\,}l}
       S_{\mathcal{R}}^{\Delta} & := \{ & l^k & \,:\, k=1,\ldots,N_{\mathcal{R}} & \} & \subset S_{n-1}, \\
       S_{U}^{\Delta} & := \{ & \eta^r & \,:\, r=1,\ldots,N_U & \} & \subset S_{m-1}
    \end{array}
 \end{equation}
satisfying
$
    \dH(S_{n-1},S_{\mathcal{R}}^{\Delta})  \le C h^p,\,
    \dH(S_{m-1},S_{U}^{\Delta}) \le C h^p
$
to  preserve  the order of the considered scheme  approximating the reachable set.

It is well-known that convex sets can be described via support functions or points in every directions. With this approximation we generate a finite set of supporting points of 
$\mathcal{R}_h(\cdot)$ and with its convex hull the fully discrete reachable set
$\mathcal{R}_{h\Delta}(\cdot)$.  
To reach this target, we also discretize the target set $\mathcal{S}$ and the control set $U$
appearing in~\eqref{eq:sv_quad_meth_local_start} and~\eqref{semigroupcombmethRh}, 
e.g.,~along the line of \cite[Proposition 3.4]{BBCG}:
\begin{equation}\label{SUdelta}
\begin{array}{r@{\,}l@{\,}r@{\,}l@{\,}r@{\,}l}
\widetilde{\mathcal{S}}_{\Delta} & :=\bigcup _  {l^k \in S_{\mathcal{R}}^{\Delta}} & \bb{y(l^k,\mathcal{S})}, & \ \,  \mathcal{S}_{\Delta} := \co(\widetilde{\mathcal{S}}_{\Delta}), \\
\widetilde{U}_{\Delta}& :=\bigcup _  {\eta^r \in S_{U}^{\Delta}} & \bb{y(\eta^r,U)} , & \ \, U_{\Delta} := \co(\widetilde{U}_{\Delta}).
\end{array}
\end{equation}
Hence, $\mathcal{S}_{\Delta},\,U_{\Delta}$ are polytopes  approximating 
$\mathcal{S}$ resp.~$U$ in the Hausdorff distance with error term $\Orem{h^p}$.

  Let $T_{h\Delta}(\cdot)$ be the fully discrete version of $T(\cdot)$ (it will be 
   defined  later in details). Our aim is to  construct the graph of $T_{h\Delta}(\cdot)$ up to a given time $t_{f}$ based on the knowledge of the reachable set approximation. We divide $[t_0, t_f]$ into $K$ subintervals each of length $\Delta t$: 
  $$\Delta t=\frac{t_f-t_0}{K},\,h=\frac{\Delta t}{N},$$
  where we have $t_f - t_0 = K N h$ and  compute subsequently the sets of supporting points $Y_{h\Delta}(\Delta t)$,\ldots, $Y_{h\Delta}(t_f)$ by the algorithm described below yielding fully discrete reachable sets
  $\mathcal{R}_{h\Delta}(i \Delta t)$, $i=1,\ldots,K$. Here $K$ decides how many sublevel sets of the graph of $T_{h\Delta}(\cdot)$ we would like to have and $h$ is the step size of the numerical scheme computing $Y_{h\Delta}(i\Delta t)$ starting from $Y_{h\Delta}((i-1)\Delta t)$. 

Due to \eqref{Rt} and \eqref{ReachLevel}, the description of each sublevel set of $T(\cdot)$ can be formulated only with its boundary points, i.e., the supporting points of the reachable sets at the corresponding time. For the discrete setting, at each step, we will determine the value of $T_{h\Delta}(x)$ for $x \in Y_{h\Delta}(\cdot)$. Therefore, we only store this information for constructing the graph of $T_{h\Delta}(\cdot)$ on the subset $[t_0,t_f]$  of its range.
\begin{algorithm}
\label{algorithm}~
\begin{enumerate}
\item[ step  1:] Set $Y_{h\Delta}(t_0)=\widetilde{\mathcal{S}}_{\Delta}$, $\mathcal{R}_{h\Delta}(t_0):= \mathcal{S}_{\Delta}$ as in \eqref{SUdelta}, $i=0$.
\item[ step  2:] Compute $\widetilde{Y}_{h\Delta}(t_{i+1})$ as follows
\begin{align*}
\widetilde{Y}_{h\Delta}(t_{i+1}) & =\Phi_h\big(t_{i+1},t_{i}\big)Y_{h\Delta}\big(t_{i}\big)+h\sum_{j=0}^{N}c_j\Phi_h(t_{i+1},t_{ij})\bar B(t_{ij}) \widetilde U_{\Delta}, \\
  \widetilde{\mathcal{R}}_{h\Delta}(t_{i+1})   & =   \co\big( \widetilde{Y}_{h\Delta}(t_{i+1}) \big), 
\end{align*}
where 
 \begin{align} \label{eq:time_steps}
    t_i & =t_0+i\Delta t,\ t_{ij}=t_i+jh \quad(j=0,1,\ldots,N).
 \end{align}
\item[ step   3:] Compute the set of the supporting points $ \bigcup_{l^k\in S_{\mathcal{R}}^{\Delta}} \bb{y(l^k,\widetilde{ \mathcal{R}}_{h\Delta}(t_{i+1}))}$  and set 
 \begin{align} \label{eq:comp_supp_pts_in_fixed_dir}
    Y_{h\Delta}(t_{i+1}) & =  \bigcup\limits_{l^k\in S_{\mathcal{R}}^{\Delta}} \big\{ y\big(l^k,\widetilde{ \mathcal{R}}_{h\Delta}(t_{i+1})\big)\big)  \big\}, 
 \end{align}
 where $ y(l^k,\widetilde{ \mathcal{R}}_{h\Delta}(t_{i+1}))$ is an arbitrary element of $\Y(l^k,\widetilde{ \mathcal{R}}_{h\Delta}(t_{i+1}))$ and set $$\mathcal{R}_{h\Delta}(t_{i+1}):=\co(Y_{h\Delta}(t_{i+1})).$$
\item[step  4:]  If $i<K-1$, set $i=i+1$ and go back to step 2. Otherwise, go to step 5.
\item[step  5:] Construct the graph of $T_{h\Delta}(\cdot)$ by 
the (piecewise) linear interpolation based on the values $t_i$ at the points $Y_{h\Delta}(t_i)$, $i=0,\ldots,K$.
\end{enumerate}
\end{algorithm}
The algorithm computes the set of vertices $Y_{h\Delta}(t_i)$ of the polygon 
$\mathcal{R}_{h\Delta}(t_i)$ which are supporting points in the directions $l^k \in
S_{\mathcal{R}}^{\Delta}$. 
The 
following proposition is the error estimate between the fully discrete reachable set $\mathcal{R}_{h\Delta}(\cdot)$ and   $\mathcal{R}(\cdot)$.  
\begin{proposition}\label{dHR_deltahl}
 Let Assumptions \ref{standassum}(i)--(iii), together with 
\begin{equation}\label{semiR}
\dH\Big(\mathcal{R}_{h}(t_i),\mathcal{R}(t_i )\Big)\le C_{s} h^p
\end{equation}
  for the set-valued combination method~\eqref{eq:sv_comb_meth_global} in (II), be valid. 
  Furthermore, finitely many directions $S_U^{\Delta},\,S_{\mathcal{R}}^{\Delta} \subset S_{n-1}$ are chosen
  with 
  $$\max(\dH(S_{n-1},S_U^{\Delta}),\dH(S_{n-1},S_{\mathcal{R}}^{\Delta}))\le C_{\Delta} h^p.$$
Then, for $h$ small enough, 
 \begin{equation}\label{dHRRdeltahlglobal}
\begin{aligned}
& \dH\Big(\mathcal{R}_{h\Delta}(t_i),\mathcal{R}_h(t_i )\Big)\le C_{f} h^p,\\
& \dH\Big(\mathcal{R}_{h\Delta}(t_i),\mathcal{R}(t_i)\Big)\le C_{f} h^p,
\end{aligned}
\end{equation}
where $C_s,\,C_{\Delta},\, C_f$ are some positive constants and  $t_i=t_0+i\Delta t,\,i=0,\ldots,K$.
\end{proposition}
\begin{proof}
  \emph{The proof can be found in~\cite[Proposition~7.2.5]{Le}.}
\end{proof}
\begin{remark}
If $\mathcal S$ is a singleton, we do not need to discretize the target set. 
The overall error estimate in~\eqref{dHRRdeltahlglobal} even improves in this case, since 
$\dH\big(\widetilde{\mathcal{R}}_{h\Delta}(t_0),\mathcal{R}_h(t_0)\big)=0$. 
\end{remark}
As we can see in this subsection the convexity of the reachable set plays a vital role. Therefore, this approach  can only be extended to special nonlinear  control systems 
with convex reachable sets.

In the following subsection, we provide the error estimation of $T_{h\Delta}(\cdot)$ obtained by the indicated approach under  Assumptions~\ref{standassum}, the regularity of $T(\cdot)$ and the properties of the numerical approximation.
\subsection{Error estimate of the minimum time function}
 After computing the fully discrete reachable sets  in Subsection~\ref{subsec:Algorithm},  we obtain the values  of $T_{h\Delta}(x)$ 
for all $x\in  \bigcup_{i=0,\ldots,K}  Y_{h\Delta}(t_i)$, $t_i= t_0 +  i\Delta t$. 
For all boundary points $x \in \partial \mathcal{R}_{h\Delta}(t_i)$ and some $i=1,\ldots,K$, we define
 \begin{align}
    T_{h\Delta}(x) & = t_i \text{ for } 
                                       x\in \partial  \mathcal{R}_{h\Delta}(t_i) ,
                                  \label{eq:discr_min_time_bd_pt}
    \intertext{together with the initial condition}
    T_{h\Delta}(x) & = t_0 \ \, \text{ for }  x \in \mathcal{S}_{\Delta}. \nonumber
 \end{align}
 The task is now to define a suitable value of $T_{h\Delta}(x)$ 
 in the computational domain
 $$ 
      \Omega := \bigcup_{i=0,\ldots,K} \mathcal{R}_{h\Delta}(t_i),
 $$ 
 if $x$ is neither a boundary point of reachable sets nor lies inside the target set.
 First we construct a simplicial triangulation 
 $\bb{\Gamma_j}_{j=1,\ldots,M}$ 
 over the set 
 $\Omega \setminus \inter(\mathcal{S})$ of points
 with grid nodes in $ \bigcup_{i=0,\ldots,K} Y_{h\Delta}(t_i) $.
 Hence, 
  \begin{itemize}
     \item $\Gamma_j \subset \R^n$ is a simplex for $j=1,\ldots,M$,
     \item $\Omega \setminus \inter(\mathcal{S}) = \bigcup_{j=1,\ldots,M}\Gamma_j$,
     \item the intersection of two different simplices is either empty
           or a common face,
     \item all supporting points in the sets $\bb{Y_{h\Delta}(t_i)}_{i=0,\ldots,K}$
           are vertices of some simplex,
     \item all the vertices of each simplex have to belong either to 
           the fully discrete reachable set
           $\mathcal{R}_{h\Delta}(t_i)$ or to $\mathcal{R}_{h\Delta}(t_{i+1})$ for some $i=0,1,\ldots,K-1$.
  \end{itemize}
 For the triangulation as in Figure~\ref{fig:part_triang},
 we introduce the maximal diameter of simplices as
 $$\Delta_{\Gamma}:= \max_{j=1,\ldots,M} \diam(\Gamma_j) .$$
\begin{figure}[htp]
\begin{center}
\includegraphics[scale=0.25]{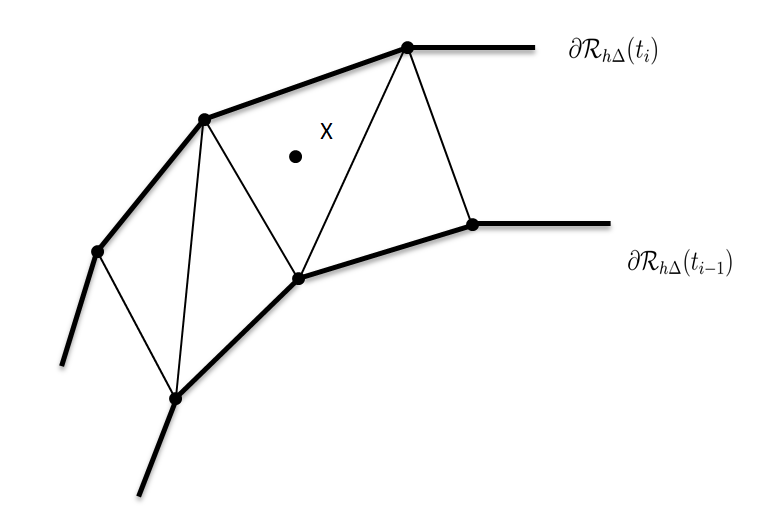}
\end{center}
\caption{Part of the triangulation}
\label{fig:part_triang}
\end{figure}\\
 Assume that $x$ is neither a boundary point of one of the computed discrete
     reachable sets $\bb{\mathcal{R}_{h\Delta}(t_i)}_{i=0,\ldots,K}$ nor an element of
     the target set $\mathcal{S}$ and
 let $ \Gamma_j $ be the simplex containing $x$.
 Then 
  \begin{align} \label{eq:cpa_def}
     T_{h\Delta}(x) & =\sum_{ \nu =1}^{ n+1 }\lambda_{ \nu } T_{h\Delta}(x_{ \nu} ),
  \end{align}
 where $x=\sum_{ \nu =1}^{ n+1 }\lambda_{ \nu } x_{ \nu },\,\sum_{ \nu =1}^{\ n+1 }\lambda_{ \nu }=1$  with  $\lambda_{ \nu }\ge 0$
 and $\bb{x_{ \nu }}_{ \nu =1, \ldots,n+1 }$  being the vertices  of $ \Gamma_j $.
  
 If $x$ lies in the interior of $\Gamma_j$, the index $j$ of this simplex is unique.
 Otherwise, $x$ lies on the common face of two or more simplices due to our
 assumptions on the simplicial triangulation and~\eqref{eq:cpa_def} is well-defined. Let $i$ be the index such that $\Gamma_j \in   \mathcal{R}_{h\Delta}(t_i) \setminus  \inter(\mathcal{R}_{h\Delta}(t_{i-1}))$.

 Since $T_{h\Delta}(x_\nu)$ is either $t_i$ or $t_{i-1}$ due to~\eqref{eq:discr_min_time_bd_pt}, we have
  \begin{align*}
     T_{h\Delta}(x) & = \sum_{\nu=1}^{n+1}\lambda_\nu T_{h\Delta}(x_\nu) \le t_i, \\
     \partial \mathcal{R}_{h\Delta}(t_i) & = \bb{y \in \R^n: T_{h\Delta}(y) = t_i}.
  \end{align*}
 The latter holds, since the convex combination is bounded by $t_i$ and equality to $t_i$
 only holds, if all vertices with positive coefficient $\lambda_\nu$ lie on the
 boundary of the reachable set $\mathcal{R}_{h\Delta}(t_i)$.

The following theorem is about the error estimate of the minimum time function obtained by this approach.
\begin{theorem}\label{errT}
Assume that $T(\cdot)$ is continuous with a non-decreasing modulus $\omega(\cdot)$ in $\mathcal{R}$, i.e.,
\begin{equation}
  | T(x)-T(y)|\le \omega(\norm{x-y})\,\,\text{ for all } x,y \in \mathcal{R}.
\end{equation}
Let Assumptions~\ref{standassum} be fulfilled, furthermore assume that
\begin{equation}\label{dHRRhlll}
\dH(\mathcal{R}_{h\Delta}(t_i),\mathcal{R}(t_i))\le Ch^p
  \quad \text{for $i=1,\ldots,K$}
\end{equation}
holds.  Then
\begin{equation}\label{fullestT}
\norm{T-T_{h\Delta}}_{\infty,\, \Omega  }\le \omega( \Delta_{\Gamma} )+ \omega(Ch^p).
\end{equation}
where $\norm{\cdot}_{\infty,\, \Omega  }$ is the supremum norm taken over $ \Omega  $.
\end{theorem}
\begin{proof}
We divide the proof into two cases.
\begin{enumerate}
\item[case 1:] $x\in \partial \mathcal{R}_{h\Delta}(t_i)$ for some $i=1,\ldots,K$. \\
Let us choose a best approximation 
$\bar{x} \in \partial\mathcal{R}(t_i)$ of $x$ so that
$$\norm{x-\bar x} = \di(x,\partial \mathcal{R}(t_i)) \leq \dH(\partial \mathcal{R}_{h\Delta}(t_i),\partial \mathcal{R}(t_i))
  =  \dH(\mathcal{R}_{h\Delta}(t_i),\mathcal{R}(t_i)),$$
where we used \cite{Wil} in the latter equality. Clearly, \eqref{ReachLevel}, 
\eqref{eq:cpa_def} show that
$$
  T_{h\Delta}(x)=T(\bar{x})=t_i.
 $$
 Then
\begin{align}
|T(x)-T_{h\Delta}(x)|&\le |T(x)-T(\bar{x})|+|T(\bar{x})-T_{h\Delta}(x)| \nonumber \\
& \le \omega(\norm{x-\bar{x}}) \le \omega\big(\dH(\mathcal{R}_{h\Delta}(t_i),\mathcal{R}(t_i))\big)\le \omega(Ch^p) \label{eq:estim_T} 
\end{align} 
due to \eqref{dHRRhlll}.

\item[ case 2:]
$x\in \inter\big(\mathcal{R}_{h\Delta}(t_i)\big)\setminus \mathcal{R}_{h\Delta}(t_{i-1})$
 for some $i=1,\ldots,K$. \\
Let $ \Gamma_j $ be a simplex containing $x$ with 
 the set of  vertices  $\bb{x_j}_{j=1,\ldots, n+1 }$. Then $$T_{h\Delta}(x)=\sum_{j=1}^{ n+1 }\lambda_j T_{h\Delta}(x_j),$$ where $x=\sum_{j=1}^{ n+1 }\lambda_j x_j,\,\sum_{j=1}^{ n+1 } \lambda_j=1, \lambda_j\ge 0$. 
 We obtain
\begin{align*}
&|T(x)-T_{h\Delta}(x)|= |T(x)-\sum_{j=1}^{ n+1 }\lambda_j T_{h\Delta}(x_j)|\\
&\le |T(x)-\sum_{j=1}^{ n+1 }\lambda_j T( x_j)|+|\sum_{j=1}^{ n+1 }\lambda_j T(x_j)-\sum_{j=1}^{ n+1 }\lambda_j T_{h\Delta}(x_j)| \\
&  \le \sum_{j=1}^{\ n+1 }\lambda_j \bigg( |T(x)-T( x_j)|+ |T(x_j) - T_{h\Delta}(x_j)| \bigg)   \le \omega(\Delta_\Gamma)+  \omega(Ch^p),
\end{align*}
where we applied the continuity of $T(\cdot)$ for the first term and the error estimate~\eqref{eq:estim_T} of case~1 for the other.
\end{enumerate}
Combining two cases and noticing that $T(x)=T_{h\Delta}(x)=t_0$ if $x\in \mathcal{S}_{\Delta}$, we get
\begin{equation}
\norm{T-T_{h\Delta}}_{\infty,\,\Omega}
:= \max_{x \in \Omega} |T(x) -T_{h\Delta}(x)|
\le \omega(\Delta_{\Gamma})+ \omega(Ch^p).
\end{equation}
The proof is completed.
\end{proof}
\begin{remark}\label{Rem_errT}
Theorem  \ref{dHRRh}  provides sufficient conditions for set-valued combination methods such that \eqref{dHRRhlll} holds. 
 See also  e.g.,~\cite{DF} for set-valued Euler's method resp. ~\cite{V} for Heun's method.
If the minimum time function is H\"older continuous  on $\Omega $, \eqref{fullestT} becomes 
\begin{equation}\label{errorHolder}
\norm{T-T_{h\Delta}}_{\infty,\, \Omega  }\le   C\Big((\Delta_{\Gamma} )^{\frac{1}{k}}+  h^{\frac{p}{k}}\Big)
\end{equation}
for some positive constant $C$. The inequality \eqref{errorHolder} shows that the error 
estimate is improved in comparison with the one obtained in~\cite{CL} 
and does not assume explicitly the regularity of optimal solutions as in~\cite{BF2}.
One possibility to define the modulus of continuity satisfying the required  property of non-decrease in Theorem~\ref{errT} is as follows:
\begin{equation*}
\omega(\delta) = \sup \{ \lvert T(x) - T(y)\rvert : \norm{x-y} \le \delta\}.
\end{equation*}
An advantage of the methods of Volterra type studied in~\cite{CL} which benefit from
non-standard selection strategies is that the discrete reachable sets converge with higher 
order than 2. The order 2 is an order barrier for set-valued Runge-Kutta methods 
with piecewise constant controls or independent choices of controls, 
since many linear control problems with intervals or boxes for the control values
are not regular enough for higher order approximations (see~\cite{V}).
Moreover, notice that there are many different triangulations based on the same data. Among them, we can always choose the one with a smaller diameter close to the Hausdorff distance of the two sets by applying standard grid generators.
\end{remark}
\begin{proposition}\label{errTt2}
Let the conditions of Theorem~\ref{errT} be fulfilled.
Furthermore assume that the step size $h$ is so small such that $C h^p$ in~\eqref{dHRRhlll} 
is smaller than $\frac{\varepsilon}{3}$, where
 \begin{align}\label{inclusionR}
    \mathcal{R}(t_i) + \varepsilon B_1(0)
      & \subset \inter \mathcal{R}(t_{i+1}) \quad\mbox{for all $i=0,\ldots,K-1$}.
 \end{align}
Then
 \begin{align}\label{inclusionRh}
     \mathcal{R}_{h\Delta}(t_i) + \frac{\varepsilon}{3} B_1(0)
       & \subset \inter \mathcal{R}_{h\Delta}(t_{i+1})
 \end{align}
and
\begin{equation}\label{fullestTt2}
\norm{T-T_{h\Delta}}_{\infty,\,  \Omega  }\le 2 \Delta t,
\end{equation}
where $\norm{\cdot}_{\infty,\, \Omega  }$ is the supremum norm taken over $ \Omega $.
\end{proposition}

\begin{proof}

For some $i = 0,\ldots,K-1$ we choose a constant $M_{i+1} > 0$ such that 
$\mathcal{R}(t_{i+1}) \subset M_{i+1} B_1(0)$.
Since $\mathcal{R}(t_i)$ does not intersect the complement of
$ \inter \mathcal{R}(t_{i+1}) $ bounded with $M_{i+1} B_1(0)$ and both are compact sets, 
there exists $\varepsilon > 0$ such that
 \begin{align}
    \mathcal{R}(t_i) + \varepsilon B_1(0) 
      & \subset \inter \mathcal{R}(t_{i+1}) \subset M_{i+1} B_1(0). \label{eq:eps_incl}
 \end{align}
We will show that a similar inclusion as~\eqref{eq:eps_incl} holds for the 
discrete reachable sets for small step sizes. If the step size $h$ is so small
that $C h^p$ in~\eqref{dHRRhlll} is smaller than $\frac{\varepsilon}{3}$, then we have
the following inclusions:
 \begin{align}
     \inter \mathcal{R}(t_{i+1}) 
       & \subset \inter \big( \mathcal{R}_{h\Delta}(t_{i+1}) + C h^p B_1(0) \big)
         = \inter \mathcal{R}_{h\Delta}(t_{i+1}) + C h^p \inter B_1(0),\nonumber \\
     \mathcal{R}(t_i) + \varepsilon B_1(0)
       & \subset \inter \mathcal{R}(t_{i+1}) 
         \subset \inter \mathcal{R}_{h\Delta}(t_{i+1}) + \frac{\varepsilon}{3} B_1(0). \nonumber
     \intertext{By the order cancellation law of convex compact sets in~\cite[Theorem~3.2.1]{PalUrb}}
     \mathcal{R}(t_i) + \frac{2}{3}\varepsilon B_1(0)
       & \subset \inter \mathcal{R}_{h\Delta}(t_{i+1}), \nonumber \\
     \mathcal{R}_{h\Delta}(t_i) + \frac{\varepsilon}{3} B_1(0)
       & \subset \big( \mathcal{R}(t_i) + \frac{\varepsilon}{3} B_1(0) \big)
         + \frac{\varepsilon}{3} B_1(0)
         \subset \inter \mathcal{R}_{h\Delta}(t_{i+1}) \label{est2levfull}.
 \end{align}
We have
\begin{equation}\label{estcase2b}
|T(x)-T_{h\Delta}(x)|= \sum_{j=1}^{ n+1 }\lambda_j|T(x)- T_{h\Delta}(x_j)|.
\end{equation}
In order to obtain the estimate, we observe that
\begin{enumerate}
\item[1)]   $x_j \in \partial \mathcal{R}_{h\Delta}(t_i)$, then 
$t_\nu \le T(x_j)\le t_{i+1}$ with $\nu=\max\{0, i-1\}$.
\item[2)]    $x\in \inter(\mathcal{R}_{h\Delta}(t_i))\setminus \mathcal{R}_{h\Delta}(t_{i-1})$,
then $t_\nu < T(x)\le t_{i+1}$ with $\nu=\max\{0, i-2\}$.
\end{enumerate}
To prove 1) the inequality
$T(x_j) >= t_0$ is clear. Assume that $T(x_j) < t_{i-1}$ for some $i > 1$. Then $x_j
\in \mathcal{R}(t_{i-1})$. By the estimates~\eqref{dHRRhlll},~\eqref{est2levfull} and $C h^p<\frac{\varepsilon}{3}$, it follows that
 \begin{align*}
    x_j & \in \mathcal{R}_{h\Delta}(t_{i-1}) + C h^p B_1(0) \subset \inter \mathcal{R}_{h\Delta}(t_i)
 \end{align*}
which is a contradiction to the assumption $x_j \in \partial \mathcal{R}_{h\Delta}(t_i)$. Hence, $T(x_j) \geq t_{i-1}$. Assume that $T(x_j) > t_{i+1}$. Then, $x_j \notin \mathcal{R}(t_{i+1})$. Furthermore, $x_j$ cannot be an element of $\mathcal{R}_{h\Delta}(t_i)$, since otherwise
a contradiction to $x_j \notin \mathcal{R}(t_{i+1})$ follows:
 \begin{align*}
    x_j & \in \mathcal{R}_{h\Delta}(t_i) \subset \mathcal{R}(t_i) + C h^p B_1(0)
      \subset \inter \mathcal{R}(t_{i+1}).
 \end{align*}
Therefore, $x_j \notin \mathcal{R}_{h\Delta}(t_i)$  which
contradicts $x_j \in \partial \mathcal{R}_{h\Delta}(t_i)$. Hence, the starting assumption $T(x_j) > t_{i+1}$ must be wrong which proves 
$T(x_j) \leq t_{i+1}$.  \\
To prove 2) if we assume $T(x) \leq t_{i-2}$ for some $i \geq 2$, then $x \in \mathcal{R}(t_{i-2})$
and 
 \begin{align*}
    x & \in \mathcal{R}_{h\Delta}(t_{i-2}) + C h^p B_1(0) \subset \inter \mathcal{R}_{h\Delta}(t_{i-1})
 \end{align*}
by estimate~\eqref{dHRRhlll}.
But this contradicts $x \notin \mathcal{R}_{h\Delta}(t_{i-1})$. Therefore, $T(x) > t_{i-2}$.
 
Assuming $T(x) > t_{i+1}$ for some $i < K-1$, then $x \notin \mathcal{R}(t_{i+1})$.
Furthermore, if $x$ is an element of $\mathcal{R}_{h\Delta}(t_i)$,
 \begin{align*}
    x & \in \mathcal{R}_{h\Delta}(t_i) \subset \mathcal{R}(t_i) + C h^p B_1(0)
      \subset \inter \mathcal{R}(t_{i+1})
 \end{align*}
which is a contradiction to $x \notin \mathcal{R}(t_{i+1})$. \\
Therefore, $x \notin \mathcal{R}_{h\Delta}(t_i)$  which
contradicts $x\in \inter(\mathcal{R}_{h\Delta}(t_i))\setminus \mathcal{R}_{h\Delta}(t_{i-1})$.
Hence, the starting assumption $T(x) > t_{i+1}$ must be wrong which proves 
$T(x) \leq t_{i+1}$. Consequently, 1) and 2) are proved. Notice that 
\begin{enumerate}
\item[a)]  the case 1) means
       \begin{alignat*}{2}
          T(x_j) & \in  [t_{i-1}, t_{i+1}] & \quad & (i \geq 1), \\
          T(x_j) &  =t_0 & \quad & (i = 0)
        \end{alignat*}
and $| T(x_j) - T_{h\Delta}(x_j) |  \leq \Delta t $ due to $ T_{h\Delta}(x_j)=t_i,\,i=0,\ldots,K$.
\item[b)] from the case 2), we obtain 
 \begin{align*}
           T(x ) & \in  (t_{i-2}, t_{i+1}] \quad (i \geq 2), \\
           T_{h\Delta}(x_j)& - T(x)   < t_i - t_{i-2} = 2 \Delta t, \\
    T_{h\Delta}(x_j) & - T(x)   > t_{i-1} - t_{i+1} = -2 \Delta t. 
 \end{align*}
Therefore, $    | T(x)  - T_{h\Delta}(x_j) |  \leq 2 \Delta t$
for $i \geq 2$ (similarly with estimates  for $i=0,1$).
\end{enumerate}
Altogether, \eqref{fullestTt2} is proved.
\end{proof}
\section{Convergence and reconstruction of discrete optimal controls}%
\label{sec:converg}
In this sub\-section we first prove the convergence of the normal cones of $\mathcal{R}_{h\Delta}(\cdot)$ to the ones of the continuous-time reachable set $\mathcal{R}(\cdot)$ in an appropriate sense. Using this result we will be able to reconstruct discrete optimal trajectories to reach the target from a set of given points and also derive the proof of $L^1$-convergence of discrete optimal controls.
 In the following only convergence under weaker assumptions and no convergence order 1 as in~\cite{ABGL-13} 
are proved (see more references 
therein for the classical field of direct discretization methods). 
We also restrict to linear minimum time problems.\\
Some basic notions of nonsmooth and variational analysis which are needed in constructing and proving the convergence of controls can be found in \cite{CLSW,Rockaf}. 
	Let $A$ be a subset in $\R^n$ and $f: A \rightarrow \R \cup \{\infty\}$ be a function. The \emph{indicator function} of $A$ and the \emph{epigraph} of $f$ be defined as 
	\begin{equation}
	\begin{aligned}
	I_A(x)=
	\begin{cases}
	0 & \quad \text{ if } x\in A\\
	+\infty & \quad \text{ otherwise}
	\end{cases}, \quad \epi f =\bb{(x,r) \in \R^n \times \R \colon x\in A, \ r\ge f(x)}.
	\end{aligned}
	\end{equation}
The definitions of normal cone and subdifferential in convex case are taken from \mbox{\cite[Sec.~8.C]{Rockaf}}.
With reference to \mbox{\cite[Definition 7.1]{Rockaf}} for epi-convergence and  \cite[Definition 5.32]{Rockaf} for graphical convergence, let us recall Attouch's theorem in a reduced version which plays an important role for convergence results of discrete optimal controls and solutions.
\begin{theorem}[see~\mbox{\cite[Theorem~12.35]{Rockaf}}] \label{theo:attouch}
   Let $(f^i)_i$ and $f$ be lower semicontinuous, convex, proper functions from $\R^n$ 
   to $\R \cup \{\infty\}$. \\
   Then the epi-convergence of $(f^i)_{i \in \N}$ to $f$ is equivalent to the graphical convergence
   of the subdifferential maps $(\partial f^i)_{i \in \N}$ to $\partial f$.
\end{theorem}
The following theorem plays an important role in this reconstruction and will deal with the convergence of the normal cones. If the normal vectors of $\mathcal{R}_{h\Delta}(\cdot)$ converge to the corresponding ones of  $\mathcal{R}(\cdot)$, the discrete optimal controls can be computed with the discrete Pontryagin Maximum Principle under suitable assumptions.\\
For the remaining part of this subsection let us consider a 
fixed index $i\in \bb{1,2\ldots,K}$.
We choose a space discretization $\Delta = \Delta(h)$ with $\Orem{\Delta} = \Orem{h^p}$ 
(compare with~\cite[Sec.~3.1]{BPhD}) and often suppress the index $\Delta$ for the approximate solutions and controls.
\begin{theorem}\label{theo:normaconver}
Consider a discrete approximation of reachable sets of type (I)--(III) with
 \begin{align} \label{eq:conv_reach_sets}
    \lim_{h \downarrow 0} \dH(\mathcal{R}_{h\Delta}(t_i),\mathcal{R}(t_i)) & = 0.
 \end{align}
Under Assumptions~\ref{standassum}, the set-valued maps $x \mapsto N_{\mathcal{R}_{h\Delta}(t_i)}(x)$ converge graphically to the set-valued map $x \mapsto N_{\mathcal{R}(t_i)}(x)$
for $i=1,\ldots,K$.
\end{theorem}
\begin{proof}
    Let us recall that, under Assumptions \ref{standassum} and by the construction in Subsec.~\ref{subsec:sv_discr_meth}, $\mathcal{R}_{h\Delta}(t_i)$, $\mathcal{R}(t_i)$ are 
    convex, compact and nonempty sets.
    Moreover, we also have that the indicator functions 
     $I_{\mathcal{R}_{h\Delta}(t_i)}(\cdot),I_{\mathcal{R}(t_i)}(\cdot)$ are lower semicontinuous convex functions (see \cite[Exercise~2.1]{CLSW}). 
By \cite[Example~4.13]{Rockaf} the convergence in~\eqref{eq:conv_reach_sets} with respect to the
Hausdorff set also implies the set convergence in the sense of Painlev\'{e}-Kuratowski (see \cite[Sec.~4.A--4.B]{Rockaf}). 
Hence, \cite[Proposition~7.4(f)]{Rockaf} applies and shows that the corresponding
indicator functions converge epi-graphically. Since the subdifferential of the (convex)
indicator functions coincides with the normal cone by~\cite[Exercise~8.14]{Rockaf},
Attouch's Theorem \ref{theo:attouch} yields the graphical convergence of the corresponding
normal cones.
\end{proof}
The remainder deals with the reconstruction of discrete optimal trajectories and 
the proof of convergence of optimal controls in the \emph{$L^1$-norm}, 
i.e., $\int_{0}^{t_i}\|\hat{u}(t)-\hat{u}_h(t) \|_1dt\rightarrow 0$ as $h\downarrow 0$
for $\hat{u}(\cdot),\,\hat{u}_h(\cdot)$ being defined later,  where 
the \emph{$\ell_1$-norm} is defined for $x\in\R^n$ as $ \|x\|_1=\sum_{i=1}^{n}|x_i|$. 
To illustrate the 
idea, we confine to a special form of the target and control set, i.e., $\mathcal S=\bb{0},\,U=[-1,1]^m,\,t\in [0,t_i]$ 
and the time invariant time-reversed linear system
\begin{align}
\label{eq:invardyn}
\begin{cases}
\dot{y}(t)&=\bar{A}y(t)+\bar Bu(t),\ u(t)\in [-1,1]^m,\\
y(0)&=0.
\end{cases}
\end{align}
Algorithm \ref{algorithm} can be interpreted pointwisely in this context as follows. For any $y_{(i-1)N}\in Y_{h\Delta}(t_i)$ there exists a sequence of controls $\bb{u_{kj}}^{ k=1,\ldots,i-1}_{j=0,\ldots,N}$ such that
\begin{equation}\label{eq:numpointw}
\begin{cases}
y_{(k-1)N}&=\Phi_h\big(t_k,t_{k-1}\big)y_{(k-1)0}+h\sum_{j=0}^{N}c_{kj}\Phi_h(t_k,t_{(k-1)j})\bar Bu_{(k-1)j},\\
y_{00}&=0,
\end{cases}
\end{equation}
for $k=1,\ldots,i$. Thus $y_{(i-1)N}=h\sum_{k=1}^{i}\sum_{j=0}^{N}c_{kj}\Phi_h(t_i,t_{(k-1)j})\bar Bu_{(k-1)j}.$
The continuous-time adjoint equation of~\eqref{eq:invardyn} written for $n$-row vectors reads as
\begin{equation}\label{eq:adjoint}
\begin{cases}
\dot{\eta}(t)&=- \eta(t)\bar{A},\\
\eta(t_i)&=\zeta
\end{cases}
\end{equation}
and its discrete version, approximated by the same method (see~\cite[Chap.~5]{Ge}) as the one used 
to discretize \eqref{eq:invardyn}, i.e., \eqref{eq:numpointw}, can be written as follows. For $k=i-1,i-2,\ldots,0$ and $j=N,N-1,\ldots,1$,
\begin{equation}\label{eq:disadjoint}
\begin{cases}
\eta_{k(j-1)}&= \eta_{kj}  \Phi _h(t_{kj},t_{k(j-1)})  \\
\eta_{(i-1)N}&=\zeta_{h},
\end{cases}
\end{equation}
where $\zeta  ,\,\zeta_h  $ will be clarified later.
By the definition of $t_{kj}$ (see Algorithm~\ref{algorithm}) the index $k0$ can be replaced by $(k-1)N$, the solution of \eqref{eq:disadjoint} in backward time is therefore possible.  Here, the end condition will be chosen subject to certain transversality conditions, see the latter reference for more details. 

 Due to well-known arguments (see e.g.,~\cite[Sec.~2.2]{LM}) 
the end point of the time-optimal solution lies on the boundary of the reachable set and the adjoint solution
$\eta(\cdot)$ is an outer normal at this end point.
Similarly, this also holds in the discrete case. The following proposition formulates this fact by a discrete version of \cite[ Sec.~2.2, Theorem~2]{LM}. The proof is just a translation of the one of the cited theorem in \cite{LM} to the discrete language. For the sake of clarity, we will formulate and prove it in detail.
\begin{proposition}
Consider the system \eqref{eq:invardyn} in $\R^n$  with its adjoint problem~\eqref{eq:adjoint} as well as their discrete pendants \eqref{eq:numpointw}, \eqref{eq:disadjoint} respectively. Let $ \bb{u_{kj}} $  be a sequence of controls, $ \bb{y_{kj}} $ be its corresponding discrete solution. Then under Assumptions \ref{standassum}, for $h$ small enough,  
$y_{(i-1)N} \in Y_{h\Delta}(t_i)$ if and only if there exists nontrivial solution $\bb{\eta_{kj}}$ of \eqref{eq:disadjoint} such that 
$$\eta_{kj}\bar B u_{kj}=\max_{u\in U} \bb{\eta_{kj} \bar Bu}$$
 for $k=0,...,i-1,\,\,j=0,...,N$, where $Y_{h\Delta}(t_i)$ is defined as in Algorithm \ref{algorithm}.
\end{proposition}
\begin{proof}
Assume that $\bb{u_{jk}}$ is such that $y_{(i-1)N}$ by the response 
$$y_{(i-1)N}=h\sum_{k=1}^{i}\sum_{j=0}^{N}c_{kj}\Phi_h(t_i,t_{(k-1)j})\bar Bu_{(k-1)j}.$$
Since $\mathcal R_{h\Delta}(t_i)$ is a compact and convex set by construction, there exists a supporting hyperplane $\gamma$ to $\mathcal R_{h\Delta}(t_i)$ at $y_{(i-1)N}$. Let $\zeta_h$ be the outer normal vector of $\mathcal R_{h\Delta}(t_i)$ at $y_{(i-1)N}$. Define the nontrivial discrete adjoint response \eqref{eq:disadjoint}, i.e.,
\begin{equation*} 
\begin{cases}
\eta_{k(j-1)}&= \eta_{kj}  \Phi _h(t_{kj},t_{k(j-1)}),  \\
\eta_{(i-1)N}&=\zeta_{h},
\end{cases}
\end{equation*}
Then $\eta_0=\eta_{(i-1)N} \Phi _h(t_i, 0)=\zeta_h\, \Phi _h(t_i, 0)$. Noticing that $\Phi _h(t_{kj},t_{k(j-1)} )$ is a perturbation of the identity matrix $I_n$, there exists $\bar{h}$ such that $\Phi _h(t_{kj},t_{k(j-1)} )$ is invertible for $h\in [0,\bar{h}]$ and so is $\Phi _h(t_i, 0)$. Therefore, $\eta_{(i-1)N}=\eta_0  \Phi _h^{-1}(t_i, 0)$. Now we compute the inner product of $\eta_{(i-1)N},\,y_{(i-1)N}$:
\begin{equation*}
\begin{aligned}
\eta_{(i-1)N}&\,y_{(i-1)N}= \displaystyle \eta_0  \Phi _h^{-1}(t_i, 0) \Big(h\sum_{k=1}^{i}\sum_{j=0}^{N}c_{kj}\Phi_h(t_i,t_{(k-1)j})\bar Bu_{(k-1)j}\Big)\\
&=h\sum_{k=1}^{i}\sum_{j=0}^{N}c_{kj} \eta_0  \Phi _h^{-1}(t_i, 0) \Phi_h(t_i,t_{(k-1)j})\bar Bu_{(k-1)j}\\
&=h\sum_{k=1}^{i}\sum_{j=0}^{N}c_{kj} \eta_0 \Phi _h^{-1}(t_{(k-1)j}, 0) \Phi _h^{-1}(t_i, t_{(k-1)j})  \Phi_h(t_i,t_{(k-1)j})\bar Bu_{(k-1)j}\\
&=h\sum_{k=1}^{i}\sum_{j=0}^{N}c_{kj} \eta_0 \Phi _h^{-1}(t_{(k-1)j}, 0) \bar Bu_{(k-1)j}=h\sum_{k=1}^{i}\sum_{j=0}^{N}c_{kj} \eta_{(k-1)j} \bar Bu_{(k-1)j}.\\
\end{aligned}
\end{equation*}
Now assume that $\eta_{kj}\bar B u_{kj}< \max_{u\in U} \bb{\eta_{kj} \bar Bu}$ for some indices $k,\,j$. Then define another sequence of controls as follows
\begin{equation*}
\tilde{u}_{kj}=
\begin{cases}
u_{kj} &\text{ if } \eta_{kj}\bar B u_{kj}=\max_{u\in U} \bb{\eta_{kj} \bar Bu},\\
\max_{u\in U} \bb{\eta_{kj} \bar Bu} &\text{ otherwise}.
\end{cases}
\end{equation*}
Let $\tilde{y}_{(i-1)N}$ be the end point of the discrete trajectory following $\bb{\tilde{u}_{kj}}$. We have
\begin{equation*}
 \eta_{(i-1)N}\,\tilde{y}_{(i-1)N}=h\sum_{k=1}^{i}\sum_{j=0}^{N}c_{kj} \eta_{(k-1)j} \bar B \tilde u_{(k-1)j}
 \end{equation*}
which implies 
 $\eta_{(i-1)N}\, y_{(i-1)N}<\eta_{(i-1)N}\,\tilde{y}_{(i-1)N}$ or $\eta_{(i-1)N}( \tilde{y}_{(i-1)N}-y_{(i-1)N})>0$ which contradicts the construction of $\eta_{(i-1)N}=\zeta_h$, an outer normal vector of $\mathcal R_{h\Delta}(t_i)$ at $y_{(i-1)N}$. Therefore, $\eta_{kj}\bar B u_{kj}= \max_{u\in U} \bb{\eta_{kj} \bar Bu}$.\\
 Conversely, assume that for some nontrivial discrete adjoint response $${\eta_{(i-1)N}=\eta_0  \Phi _h^{-1}(t_i, 0)},$$ the controls satisfies
 \begin{equation}\label{eq:asscontrol}
 \eta_{kj}\bar B u_{kj}= \max_{u\in U} \bb{\eta_{kj}\bar B u}
 \end{equation}
  for every indices $k=0,...,i-1,\,j=0,...,N$. We will show that the end point $y_{(i-1)N}$ of the corresponding trajectory $\bb{y_{kj}}$ will lie at the boundary of $\mathcal R_{h\Delta}(t_i)$, not at any point belonging to its interior. Suppose, by contradiction, $y_{(i-1)N}$ lies in the interior of 
 $\mathcal R_{h\Delta}(t_i)$. Let $\tilde{y}_{(i-1)N}$ be a point reached by a sequence of controls $\bb{\tilde{u}_{kj}}$ in $\mathcal R_{h\Delta}(t_i)$ in such that 
 \begin{equation}\label{eq:ineqcontrass}
 \eta_{(i-1)N}y_{(i-1)N} < \eta_{(i-1)N}\tilde{y}_{(i-1)N}.
 \end{equation} 
Our assumption \eqref{eq:asscontrol} implies that 
 \begin{equation}\label{eq:ineqcontr}
 \eta_{kj}\bar B \tilde{u}_{kj}\le  \eta_{kj} \bar Bu_{kj}
 \end{equation} for all $k,j$. As above, due to \eqref{eq:ineqcontr}, we show that
 $$\eta_{(i-1)N}\tilde{y}_{(i-1)N}\le \eta_{(i-1)N}y_{(i-1)N}$$ which is a contradiction to \eqref{eq:ineqcontrass}. Consequently, $y_{(i-1)N} \in \partial \mathcal R_{h\Delta}(t_i)=Y_{h\Delta}(t_i)$.
\end{proof}

Motivated by the outer normality of the adjoints in continuous resp.~discrete time and the
maximum conditions, we 
define the optimal controls $\hat{u}(t),\,\hat{u}_h(t)$ as follows
\begin{equation}\label{def:contr}
\left\{
\begin{aligned}
\hat{u}(t) & =\sign (\eta(t)\bar B )^\top & & \text{for } (t \in [0,t_i]), \\
\hat{u}_h(t)&=\hat{u}_{kj}  & & \text{if } t\in [t_{kj},t_{k(j+1)}),\,k=0,...,i-1,\, \\
& & & j=0,...,N-1,\\
\hat{u}_h(t_{(i-1)N})&=\hat{u}_{(i-1)(N-1)} & & \text{for } t=t_{(i-1)N},
\end{aligned}
\right.
\end{equation}
where $\hat{u}_{kj}=\sign (\eta_{kj}\bar B)^\top,\,k=0,...,i-1,\,j=0,...,N$ 
and 
\begin{equation*}
w := \sign(v) \text{ with } w_\mu =
\begin{cases}
1 &\text{ if } v_\mu>0,\\
0 &\text{ if } v_\mu=0,\\
-1 &\text{ if } v_\mu<0
\end{cases}
\end{equation*}
is the \emph{signum function} and $v,w \in \R^m$, $\mu=1,\ldots,m$.

Owing to Theorem~\ref{theo:normaconver}, we have that the set-valued maps $(N_{\mathcal{R}_{h\Delta}(t_i)}(\cdot))_h$ converge graphically to $N_{\mathcal{R}(t_i)}(\cdot)$ 
which implies
that for every 
sequence $(y_{(i-1)N},\eta_{(i-1)N})_N$ 
in the graphs there exists an element $(y(t_i),\eta(t_i))$
of the graph such that 
\begin{equation}\label{eq:convernorvec}
(y_{(i-1)N},\eta_{(i-1)N}) \rightarrow (y(t_i),\eta(t_i)) \text{ as } h \downarrow 0,
\end{equation}
where $\eta_{(i-1)N} \in N_{\mathcal{R}_{h\Delta}(t_i)}(y_{(i-1)N}),\,\eta(t_i)\in N_{\mathcal{R}(t_i)}(y(t_i))$. Thus $\zeta,\,\zeta_h$ are chosen such that \eqref{eq:convernorvec} is realized.
Then it is obvious that $\eta_{kj} \rightarrow \eta(t_{kj})$ as $h \downarrow 0$ with $k=0,...,i-1$ uniformly in $j=0,...,N$. 

For a function $g \colon I \rightarrow \R^m$, we denote the total variation $V(g,I):=\sum_{1}^{m}V(g_i,I)$, where $V(g_i,I)$ is the usual total variation of the $i$-th components of $g$ over a bounded interval $I\in \R$.  
Now if we assume that the system \eqref{eq:invardyn}  
is normal, $\hat{u}_h(t)$ converges to $\hat{u}(t)$ in 
the $L^1$-norm.
\begin{proposition}
Consider that the minimum time problem with the dynamics~\eqref{eq:invardyn} in $\R^n$. Assume that the normality condition holds, i.e.,
\begin{equation}\label{eq:rank}
\rk \bb{B\omega,AB\omega,\ldots,A^{n-1}B\omega}=n 
\end{equation}
for each (nonzero) vector $\omega$
 along an edge of $U=[-1,1]^m$ or along the two end points of the interval $U=[-1,1]$ if $m=1$. Then, under Assumptions \ref{standassum}, $\int_{0}^{t_{i}} \|\hat{u}(t)-\hat{u}_h(t)\|_1dt \rightarrow 0$ as $h\rightarrow 0$ for any $i\in \bb{1,\ldots,K}$.
\end{proposition}
\begin{proof}
Due to \eqref{eq:rank} $\hat{u}(t)$ defined as in \eqref{def:contr} on $t_0\le t\le t_i$ is the optimal control to reach the state $\hat{y}(t_i)$ of the corresponding optimal solution from the origin. Moreover, it has a finite number of switchings see \cite[Sec.~2.5, Corollary~2]{LM}. Therefore, the total variation, $V(\hat{u}(t),[t_0,t_i])$, is bounded. Let $\displaystyle I_{kj}=[t_{kj},t_{k(j+1)})$, for $k=0,\ldots,i-1,\,j=0,\ldots,N-1$, and except for  ${I_{(i-1)(N-1)}=[t_{(i-1)(N-1)},t_{(i-1)N}]}$. Then 
\begin{equation}
\begin{aligned}
&\int_{I_{kj}}\|\hat{u}(t)-\hat{u}_h(t)\|_1 dt\le \int_{I_{kj}}(\|\hat{u}(t)-\hat{u}(t_{kj})\|_1+ \|\hat{u}(t_{kj})-\hat{u}_h(t_{kj})\|_1)dt\\
&\le hV(\hat{u}(t),I_{kj})+h\|\sign (\eta(t_{kj})\bar B)^\top-\sign (\eta_{kj}\bar B))^\top \|_1.
\end{aligned}
\end{equation}
Taking a sum over $k=0,\ldots,i-1,\, j=0,\ldots,N-1$ we obtain
\begin{align*}
& \int_{t_0}^{t_i}\|\hat{u}(t)-\hat{u}_h(t)\|_1dt \\
\le \ & hV(\hat{u}(t),[t_0,t_i])+h\sum_{k=0}^{i-1}\sum_{j=0}^{N-1} \|\sign ( \eta(t_{kj})\bar B))^\top-\sign (\eta_{kj}\bar B))^\top \|_1.\\
\end{align*}
Since $\hat{u}(t)$ has a finite number of switchings and 
$\eta_{kj}, \eta(t_{kj})$ are non-trivial with the convergence 
$\eta_{kj} \rightarrow \eta(t_{kj})$ as $h\rightarrow 0$ 
for $k=0,\ldots,i,\, j=0,\ldots,N$, the variation $V(\hat{u}(t),[t_0,t_i])$ and $\sum_{k=0}^{i}\sum_{j=0}^{N-1} \|\sign (\eta(t_{kj})\bar B))^\top-\sign (\eta_{kj}\bar B))^\top \|_1$ are bounded. Therefore, 
\begin{equation*}
\int_{t_0}^{t_i} \|\hat{u}(t)-\hat{u}_h(t) \|_1dt \rightarrow 0 \text{ as } h\rightarrow 0.
\end{equation*}
The proof is completed.
\end{proof}
\section{Numerical tests}\label{sec:num_tests}
The following examples should serve as a collection of academic test examples 
for calculating
the minimum time function for several, mainly linear control problems
which were previously discussed in the literature.
The examples also illustrate the performance of the error behavior of our proposed approach. 

The space discretization follows the presented approach in 
Subsection \ref{subsec:Algorithm}
and uses supporting points in directions
\begin{align*}
l^k & := \bigg( \cos\bigg(2 \pi \frac{k-1}{N_{\mathcal{R}}-1}\bigg), \ \sin\bigg(2 \pi \frac{k-1}{N_{\mathcal{R}}-1}\bigg) \bigg)^\top, \ k=1,\ldots,N_{\mathcal{R}}, \\
\eta^r & := \begin{cases}
-1 + 2(r-1) & \quad \text{if $U=[-1,1]$},\ r=1,\ldots,N_U, \\
l^r  & \quad \text{if $U \subset \R^2$},\  r=1,\ldots,N_U \\
\end{cases}
\end{align*}
and normally choose either $N_U = 2$ for one-dimensional control sets or $N_U = N_{\mathcal{R}}$ for $U \subset \R^2$  in the discretizations of the unit sphere \eqref{eq:discr_unit_spheres}.

The comparison of the two applied methods is done by computing the error with respect to the  
$L^{\infty}$-norm  of the difference between the approximate and the true minimum time 
function evaluated at test points. 
The true minimum time function is delivered 
analytically by tools  from control theory.  
The  test grid points are distributed uniformly over the domain 
$\mathcal{G}=[-1,1]^2$ with step size $\Delta x= 0.02$.  
\subsection{Linear examples}
In the linear, two-dimensional, time-invariant Examples~\ref{ex:1}--\ref{ex:3b} we can
check Assumption \ref{standassum}(iv) 
\begin{quote}
	$\mathcal{R}(t)$ is \emph{strictly expanding} on the compact interval $[t_0,t_f]$, i.e., ${\mathcal{R}(t_1) \subset \inter \mathcal{R}(t_2)}$ for all $t_0\le t_1<t_2\le t_f$.
\end{quote}
in several ways. From the numerical calculations
we can observe this property in the shown figures for the fully discrete reachable sets.
Secondly, we can use the available analytical formula for the minimum time function
resp.~the reachable sets or check 
the Kalman rank condition
$
\rk\Big[ B, A B \Big] = 2
$
for time-invariant systems if the target is the origin (see~\cite[Theorems~17.2 and~17.3]{HL}).

The control sets in the linear examples are either one- or two-dimensional polytopes (a segment or a square)
or balls and are varied to study different regularity allowing high or low order of convergence
for the underlying set-valued quadrature method.
In all linear examples, we apply a set-valued combination method of order 1 and 2 (the set-valued
Riemann sum combined with Euler's method resp.~the set-valued trapezoidal rule with Heun's method).

We start with an example having a Lipschitz continuous minimum time function and verify the
error estimate in Theorem \ref{errT}. Observe that the numerical error here is only
contributed by the spatial discretization of the target set or control set.
\begin{example} 
	\label{ex:1}
	Consider the control dynamics , see \cite{BFS,GL},
	\begin{equation}\label{example1}
	\dot{x}_1=u_1,\,\,\dot{x}_2=u_2,\,\, (u_1,u_2)^\top \in U \text{ with $U:= B_1(0)$ or $U := [-1,1]^2$ }.
	\end{equation}
	We consider either the small ball $B_{0.25}(0)$ or the origin as target set $\mathcal{S}$. This is a simple time-invariant example with $\bar A=\begin{bmatrix}
	0&  0  \\[0.3em]
	0 & 0 
	\end{bmatrix}$, $\bar B=\begin{bmatrix}
	-1&  0  \\[0.3em]
	0 & -1 
	\end{bmatrix}$.
	Its fundamental solution matrix is the identity matrix, therefore
	\begin{equation*}
	\mathcal{R}(t)=\Phi(t,t_0)S+\int_{t_0}^{t}\Phi(t,s)\bar B(s)U=S+(t_0-t)U,
	\end{equation*}
	and any method from (I)--(III) gives the exact solution, i.e.,
	$$
	\mathcal{R}_h(t)=\mathcal{R}(t) 
	=S+(t-t_0)U
	$$ 
	due to the symmetry of $U$. For instance, the set-valued Euler scheme with ${h=\frac{t_{j+1}-t_j}{N}}$ yields 
	\begin{equation*}
	\begin{cases}
	\mathcal{R}_h(t_{j+1})=\mathcal{R}_h(t_j)+h(\bar A \mathcal{R}_h(t_j)+\bar B U)=\mathcal{R}_h(t_j)-hU,\\
	\mathcal{R}_h(t_0)=S,
	\end{cases}
	\end{equation*}
	therefore, $\mathcal{R}_{h}(t_N)=S-NhU=S+( t_N -t_0)U$ and the error is only due to the space discretizations $\mathcal{S}_\Delta \approx \mathcal{S}$, $U_\Delta \approx U$ and does not depend on $h$ 
	(see Table~\ref{tab:1}). The error would be the same for finer step size $h$ and $\Delta t$
	in time or if a higher-order method is applied. Note that the error for the origin as target
	set (no space discretization error) is in the magnitude of the rounding errors of floating
	point numbers.
	We choose $t_f = 1,\,K=10$ and $N=2$ 
        and the set-valued Riemann sum combined with Euler's method
        for the computations. 
	It is easy to check that the minimum time function is Lipschitz continuous, since 
	one of the equivalent Petrov conditions in~\cite{P}, \cite[Chap.~IV, Theorem~1.12]{BCD} with $U=B_1(0)$ or $[-1,1]^2$ hold:
	\begin{align*}
	0 & > \min_{(u_1,u_2)^{\top}\in U} \ang{\nabla d(x,\mathcal{S}),(u_1,u_2)^\top}, \\
	0 & \in \inter\bigg( \bigcup_{u \in U} f(0,u) \bigg) \quad\text{with $f(x,u) = A x + B u$.}
	\end{align*}
	Moreover, the support function with respect to the time-reversed dynamics \eqref{example1} 
	\begin{align*}
	\delta^* (l,\Phi(t,\tau)\bar B(\tau)U) & = \begin{cases}
	\|l\| & \quad\text{if $U = B_1(0)$}, \\
	|l_1|+|l_2| & \quad\text{if $U = [-1,1]^2$}
	\end{cases}
	\end{align*}
	is constant  with respect to the time $t$, so it is trivially arbitrarily continuously differentiable with respect to $t$ with bounded derivatives  uniformly for all $l\in S_{n-1}$. 
	\begin{table}[h]
	\small
		\begin{tabular}{|c|c|c|c|}
			\hline
			$  N_{\mathcal{R}} = N_U $ & $U=B_1(0)$,
			& $U=[-1,1]^2$,
			& $U=[-1,1]^2$, \rule{0ex}{3ex} \\
			&   $\mathcal{S}=B_{0.25}(0)$ 
			& $\mathcal{S}=B_{0.25}(0)$  
			& $\mathcal{S}=\bb{0}$ \rule[-1.5ex]{0ex}{1.5ex} \\
			\hline
			$100$& $ 6.14\times 10^{-4}$  & $ 4.9 \times 10^{-4} $  
			& $ 8.9 \times 10^{-16} $ \rule{0ex}{3ex} \\
			\hline
			$50$ & $ 24\times 10^{-4}$  & $ 19 \times 10^{-4} $  
			& $ 8.9 \times 10^{-16} $ \rule{0ex}{3ex} \\
			\hline
			$25$ &  $ 0.0258 $  & $ 0.0073  $  
			& $ 8.9 \times 10^{-16} $ \rule{0ex}{3ex} \\
			\hline
		\end{tabular}\\[2ex]
		\caption{error estimates for Example~\ref{ex:1} with different control and target sets} 
		\label{tab:1}
	\end{table}
	In Fig.~\ref{fig:1_ball} the minimum time functions are plotted 
	for Example~\ref{ex:1} for two different control sets $U = B_1(0)$ (left) and $U = [-1,1]^2$ (right) with the same two-dimensional target set $\mathcal{S} = B_{0.25}(0)$. 
	The minimum time function is in general not differentiable everywhere. Since it is 
	zero in the interior of the target, one has at most Lipschitz continuity at
	the boundary of $\mathcal{S}$.
	In Fig.~\ref{fig:1_square_target_origin} the minimum time function is plotted 
	for the same control set as in Fig.~\ref{fig:1_ball} (right), but this time the target set
	is the origin and not a small ball.
    
	\begin{figure}[htp]
		\begin{center}
			\includegraphics[scale=0.34]{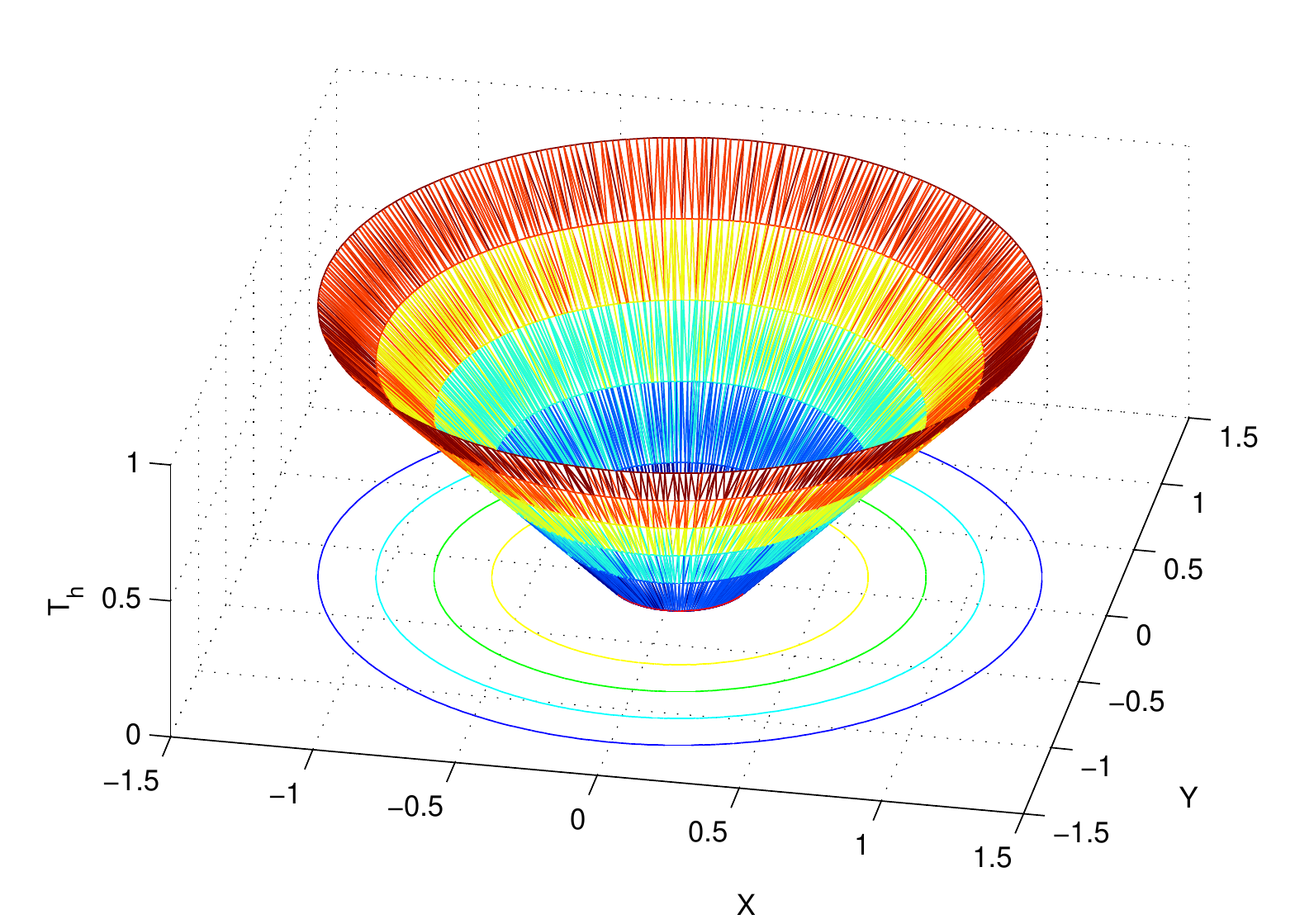}
			\includegraphics[scale=0.34]{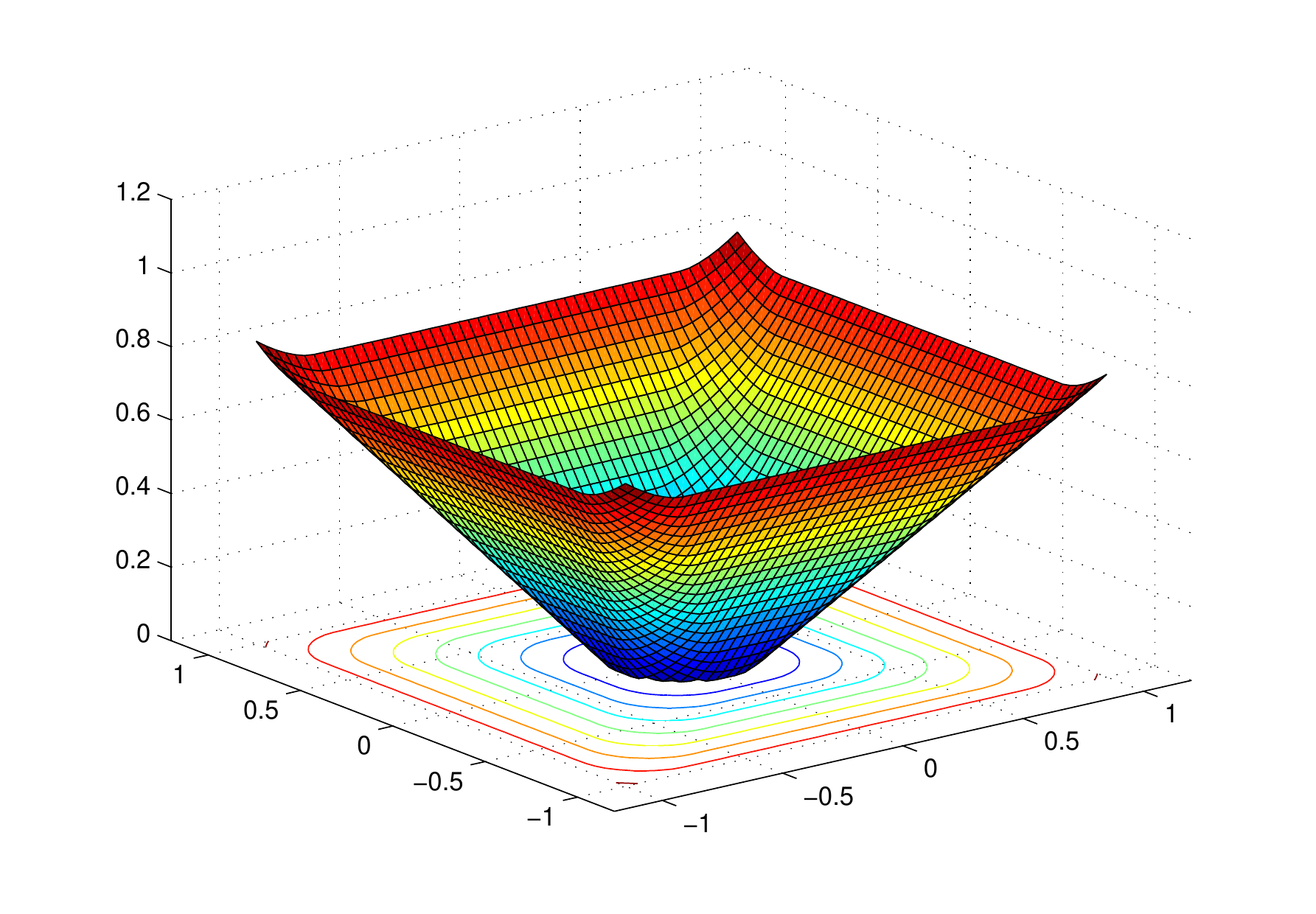}
   			\caption{Minimum time functions for Example~\ref{ex:1} with different control sets}
			\label{fig:1_ball}
		\end{center}
	\end{figure}

	\begin{figure}[htp]
		\begin{center}
			\includegraphics[scale=0.475]{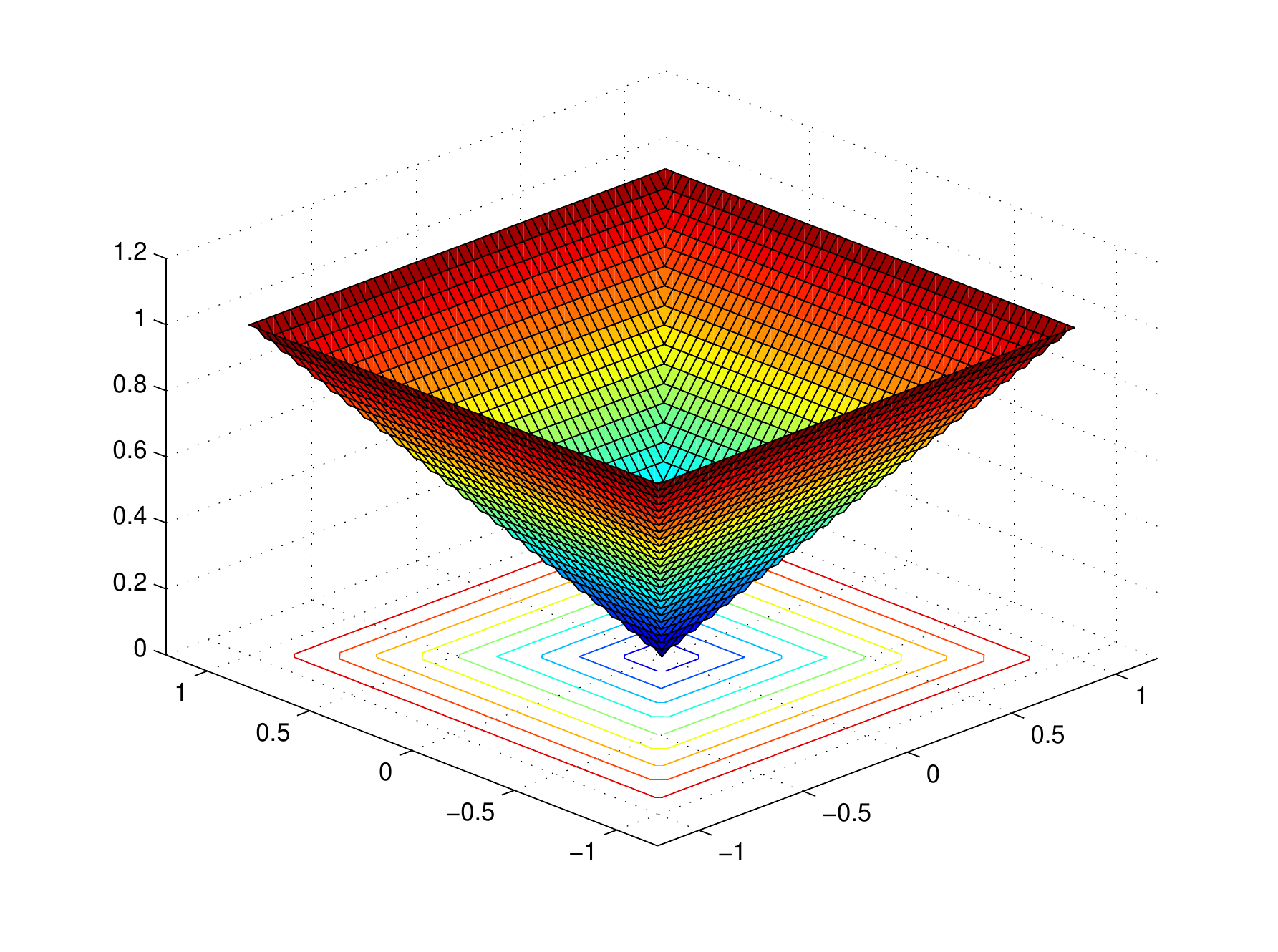}
			\caption{Minimum time function for Example~\ref{ex:1} with ${U = [-1,1]^2}$, $\mathcal{S}=\{0\}$}
			\label{fig:1_square_target_origin}%
		\end{center}
	\end{figure}
\end{example} 

We now study well-known dynamics as the double integrator and the harmonic oscillator
in which the control set is one-dimensional. The classical rocket car example with
H\"older-continuous minimum time function was already computed
by the Hamilton-Jacobi-Bellman approach in~\cite[Test~1]{F} and \cite{CL,GL}, where numerical calculations
are carried out by enlarging the target (the origin) by a small ball.
\begin{example} 
	\label{ex:2}
	a) The following dynamics is the \emph{double integrator}, see e.g.,~\cite{CL}.
	\begin{equation}\label{example3}
	\dot{x}_1=x_2,\,\dot{x}_2=u,\,\,u\in U := [-1,1].
	\end{equation}
	We consider either the small ball $B_{0.05}(0)$ or the origin as target set $\mathcal{S}$.
	Then the minimum time function is  $\frac{1}{2}$--H\"older continuous 
	for the first choice of $\mathcal{S}$ see \cite{AM,CL} and the support function for the time-reversed dynamics \eqref{example3} 
	$$\delta^* (l,\Phi(t,\tau)\bar B(\tau)[-1,1])=\delta \Bigg(l,\begin{bmatrix}
	1&  -(t-\tau)  \\[0.3em]
	0 &  1 
	\end{bmatrix} \begin{bmatrix}
	0  \\[0.3em]
	-1 
	\end{bmatrix}[-1,1]\Bigg)=\big|(t-\tau,-1) \cdot l \big|$$
	is only absolutely continuous with respect to $\tau$ for some directions $l \in S_1$
	with $l_1\neq 0$. Hence, we can expect that the convergence order 
	for the set-valued quadrature method is at most $2$. 
	We fix $t_f = 1$ as maximal computed value for the minimum time function
	and $N = 5$.
	
	In Table~\ref{tab:2} the error estimates for two set-valued combination methods
	are compared (order 1 versus order 2). Since the minimum time function is only
	$\frac{1}{2}$--H\"older continuous we expect as overall convergence order $\frac{1}{2}$
	resp.~$1$. A least squares approximation of the function $C h^{p}$ for the error term
	reveals $C = 1.37606$, $p = 0.4940$ for Euler scheme combined with set-valued Riemann sum
	resp.~$C = 22.18877$, $p = 1.4633$ (if $p=1$ is fixed, then $C= 2.62796$) for Heun's method combined with set-valued trapezoidal
	rule. Hence, the approximated error term is close to the expected one 
	by Theorem \ref{errT} and Remark \ref{Rem_errT}. 
	Very similar results are obtained with the Runge-Kutta methods of order 1 and 2
	in Table~\ref{tab:22} in which the set-valued Euler method is slightly better than the
	combination method of order 1 in Table~\ref{tab:2}, and the set-valued Heun's method
	coincides with the combination method of order 2, since both methods use the same approximations of the given dymanics. 
	Here we have chosen to double the number of directions $N_{\mathcal{R}}$ each time the step size
	is halfened which is suitable for a first order method. For a second order method
	we should have multiplied $N_{\mathcal{R}}$ by 4 instead. From this point it is not surprising
	that there is no improvement of the error in the fifth row for step size $h = 0.0025$. 

	\begin{table}[h]
	
		\begin{tabular}{|l|c|c|c|}
			\hline
			\mbox{ }\ \,$h$ & $N_{\mathcal{R}}$ &  \makecell{\textbf{Euler scheme} \\ \& \textbf{Riemann sum}} 
			& \makecell{\textbf{Heun's scheme} \\ \& \textbf{trapezoid rule}} \\
			\hline
			$0.04$ & $50$ & $0.2951$ & $0.2265$ \\
			\hline
			$0.02$ & $100$ & $0.1862$ & $0.1180$ \\
			\hline
			$0.01$ & $200$ & $0.1332$ & $0.0122$ \\
			\hline
			$0.005$ & $400$ & $0.1132$ & $0.0062$ \\
			\hline
			$0.0025$ & $800$ & $0.0683$ & $0.0062$ \\
			\hline
		\end{tabular}\\[2ex]
		\caption{Error estimates for Ex.~\ref{ex:2} a) for combination methods of order 1 and 2} 
		\label{tab:2}
	\end{table}
     \begin{table}[h]
	\small
		\begin{tabular}{|l|c|c|c|}
			\hline
			\mbox{ }\ \,$h$ & $N_{\mathcal{R}}$ & \textbf{set-valued Euler method} 
			& \textbf{set-valued Heun method} \\
			\hline
			$0.04$ & $50$ & $0.2330$ & $0.2265$ \\
			\hline
			$0.02$ & $100$ & $0.1681$  &  $0.1180$ \\
			\hline
			$0.01$ & $200$ & $0.1149$   & $0.0122$  \\
			\hline
			$0.005$ & $400$ & $0.0753$ &  $0.0062$ \\
			\hline
			$0.0025$ & $800$ &$0.0318$  &  $0.0062$ \\
			\hline
		\end{tabular}\\[2ex]
		\caption{Error estimates for Ex.~\ref{ex:2} a) 
			for Runge-Kutta meth.\ of order 1 and 2} 
		\label{tab:22}
	\end{table}
        
	\begin{figure}[htp]
		\begin{center}
			\includegraphics[scale=0.375]{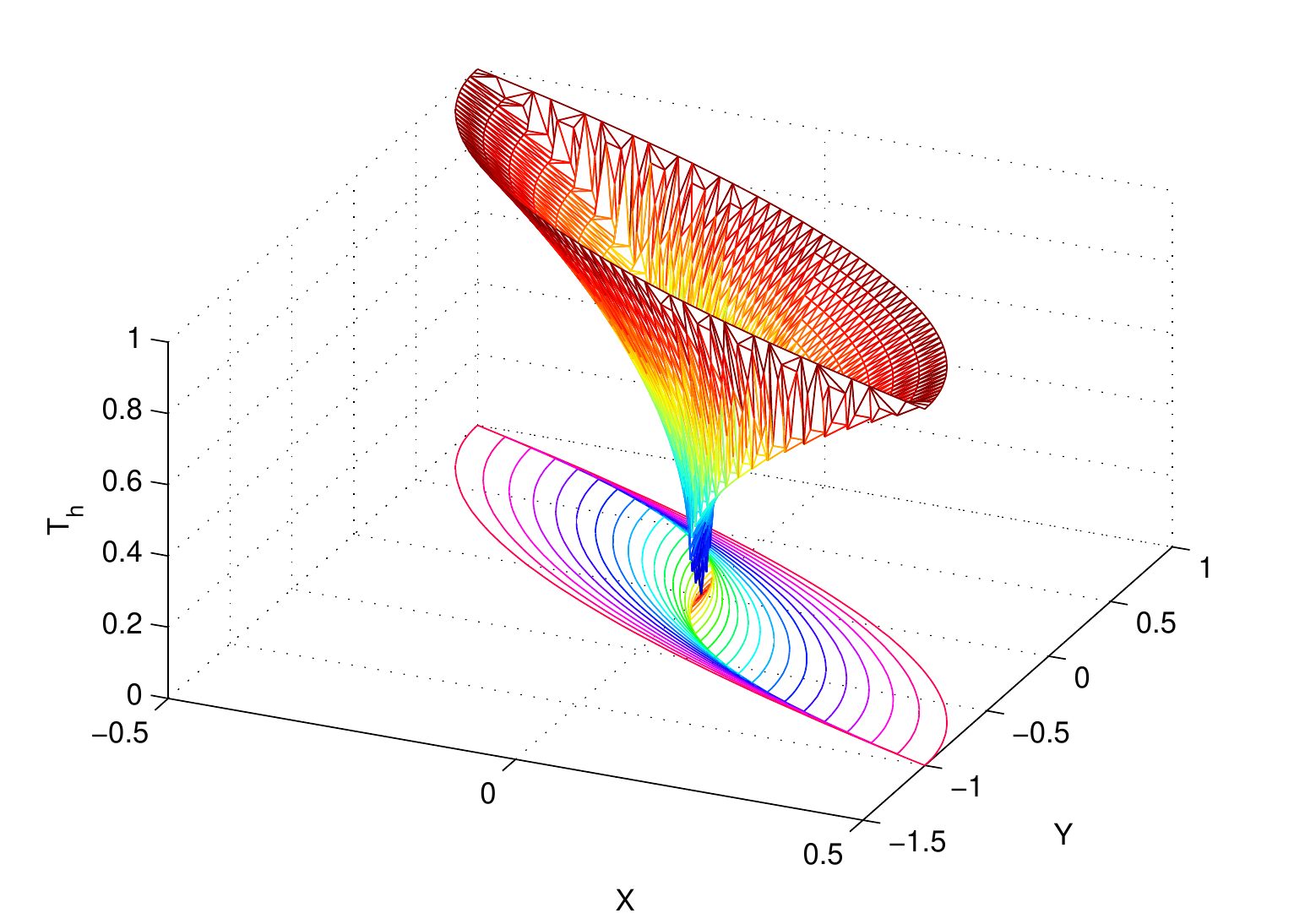}
                        \hspace*{-4ex}
			\includegraphics[scale=0.375]{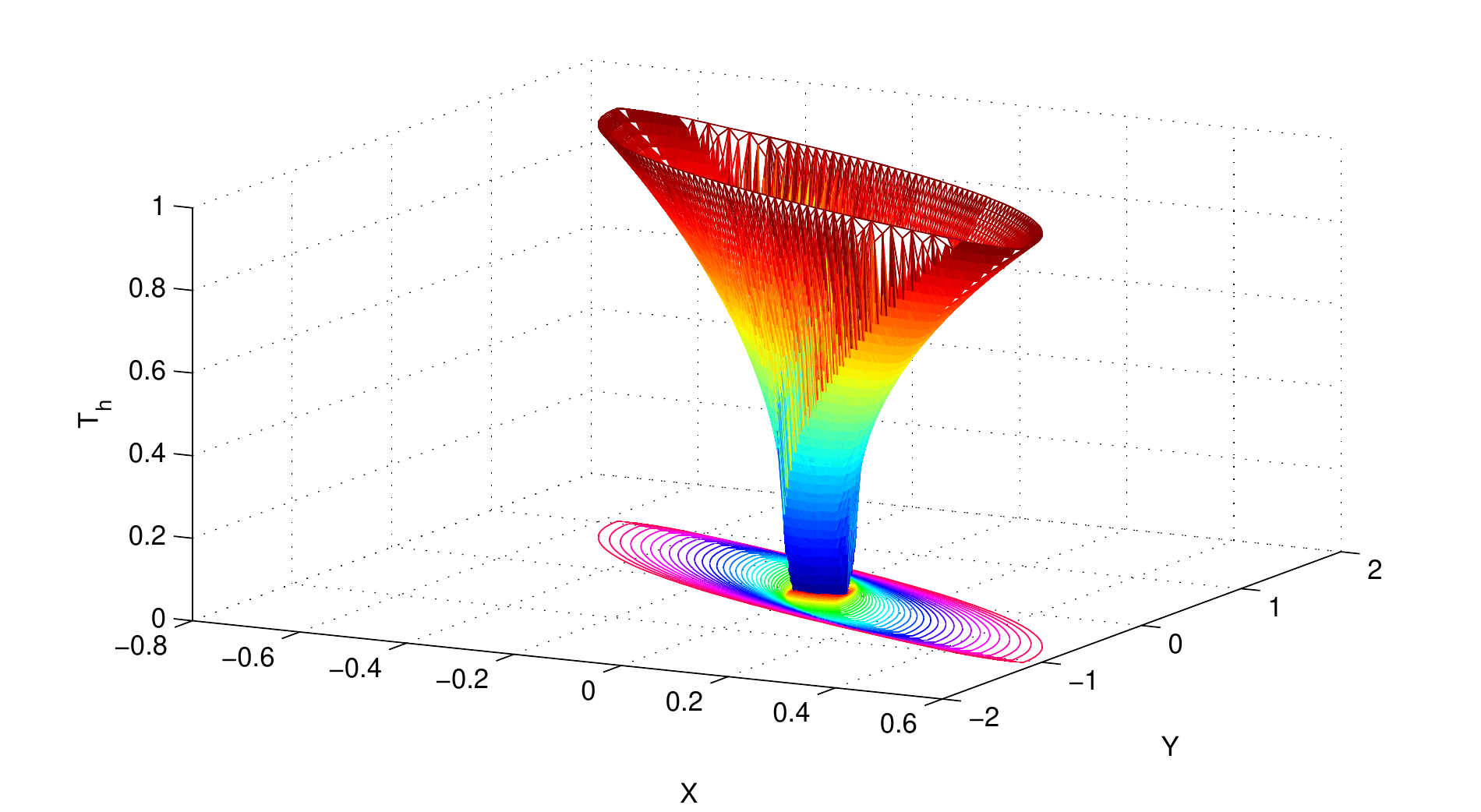}
   			\caption{Minimum time function  for Example~\ref{ex:2}a)  with target set $\{0\}$ resp.\ $B_{0.05}(0)$}\label{exam_31}
			\label{fig:2a_small_ball}
			\label{fig:2a_origin}
		\end{center}
	\end{figure}

	As in Example~\ref{ex:1}  we can consider the dynamics \eqref{example3} with the origin as a target (see the minimum time function in Fig.~\ref{fig:2a_origin}~(left). In this case, the numerical computation by PDE approaches, i.e., the solution of the associated Hamilton-Jacobi-Bellman equation (see e.g.,~\cite{F})  requires the replacement of the target point $0$ by a small ball  $B_\varepsilon(0)$ for  suitable  $\varepsilon>0$. This replacement surely increases the error of the  calculation   (compare the minimum time function in Fig.~\ref{fig:2a_small_ball} for $\varepsilon = 0.05$). However, our proposed approach works perfectly regardless of the fact whether  $\mathcal{S}$  is a  two-dimensional set or a singleton.
	\\[1ex]
	b) harmonic oscillator dynamics (see~\cite[Chap.~1, Section~1.1, Example 3]{LM})
	\begin{equation}\label{example4}
	\dot{x}_1=x_2,\,\dot{x}_2=-x_1+u,\,\,u\in U :=[-1,1].
	\end{equation} 
	Since the Kalman rank condition
	$
	\rk\Big[ B, A B \Big] = 2,
	$
	the minimum time function $T(\cdot)$ is also continuous.
	The plot for $T(x)$ for the harmonic oscillator with the origin as
	target, $ t_f=6,\,N_{\mathcal{R}} = 100,\,
	N=5 $ and $K=40$ is shown in Fig.~\ref{fig:2b_origin}.
	
	According to Section \ref{sec:converg} we construct open-loop time-optimal controls for the discrete 
	problem with target set $\mathcal{S} = \{0\}$ by Euler's method. In Fig.~\ref{fig:exam_32} the corresponding discrete open-loop time-optimal
	trajectories for Examples~\ref{ex:2}a)~(left) and b)~(right) are depicted.
        
	\begin{figure}[htp]
		\begin{center}
			\includegraphics[scale=0.27]{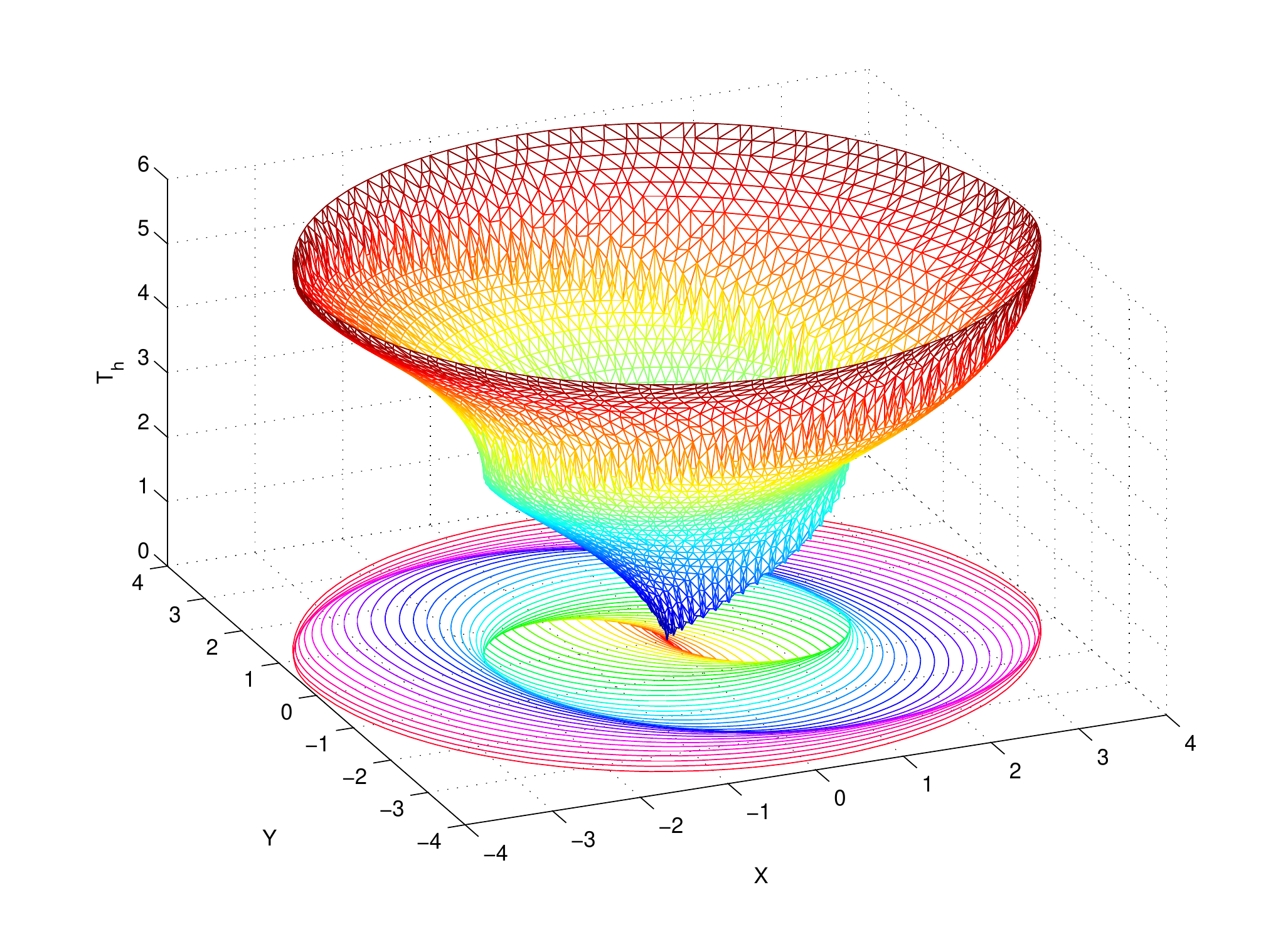}
   			\caption{Minimum time functions for Example~\ref{ex:2}b)}
			\label{fig:2b_origin}
		\end{center}
	\end{figure}
	\begin{figure}[htp]
		\begin{center}
			\includegraphics[width=2.4in,height=2.1in]{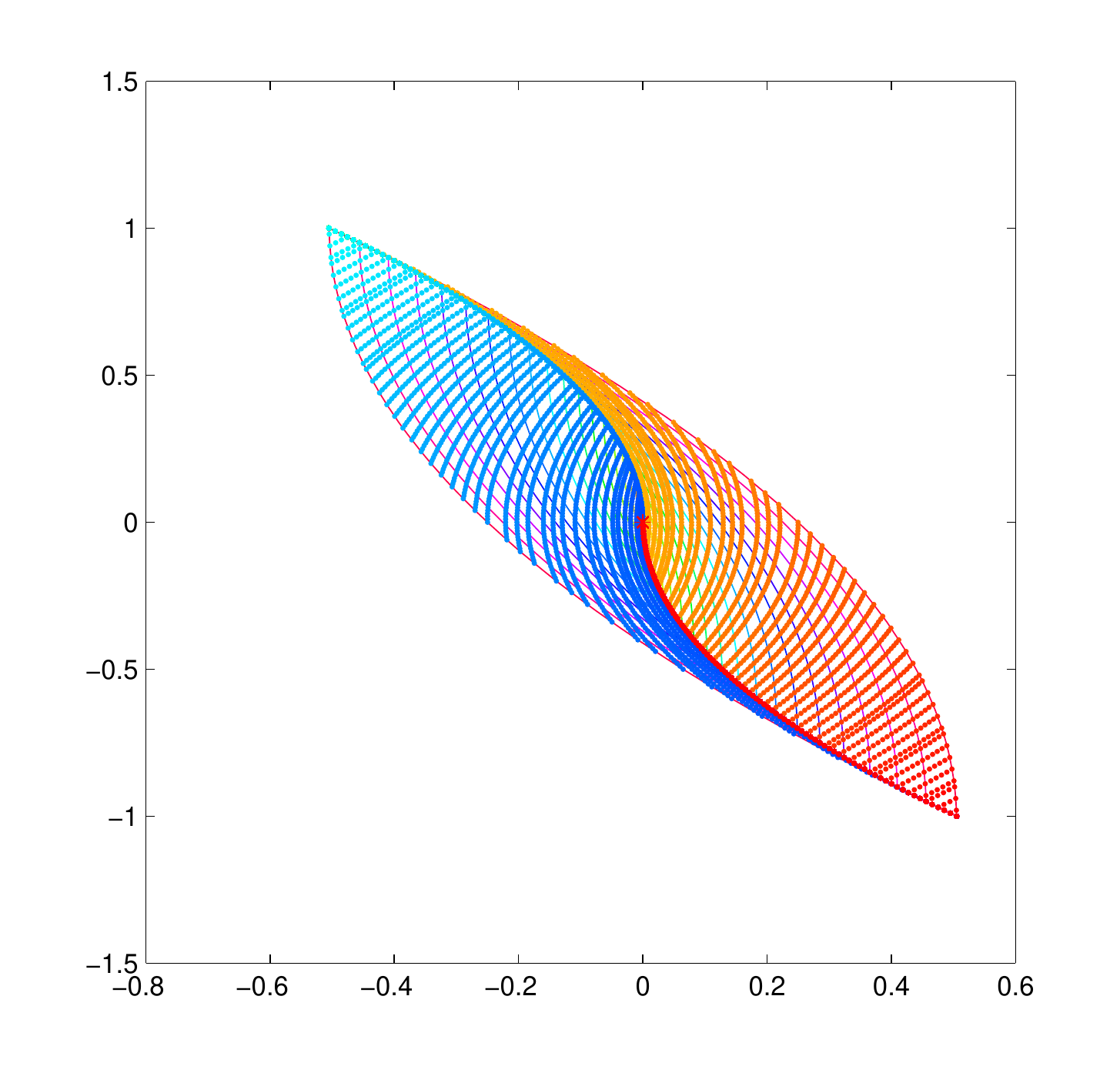}
			\includegraphics[width=2.5in,height=2.1in]{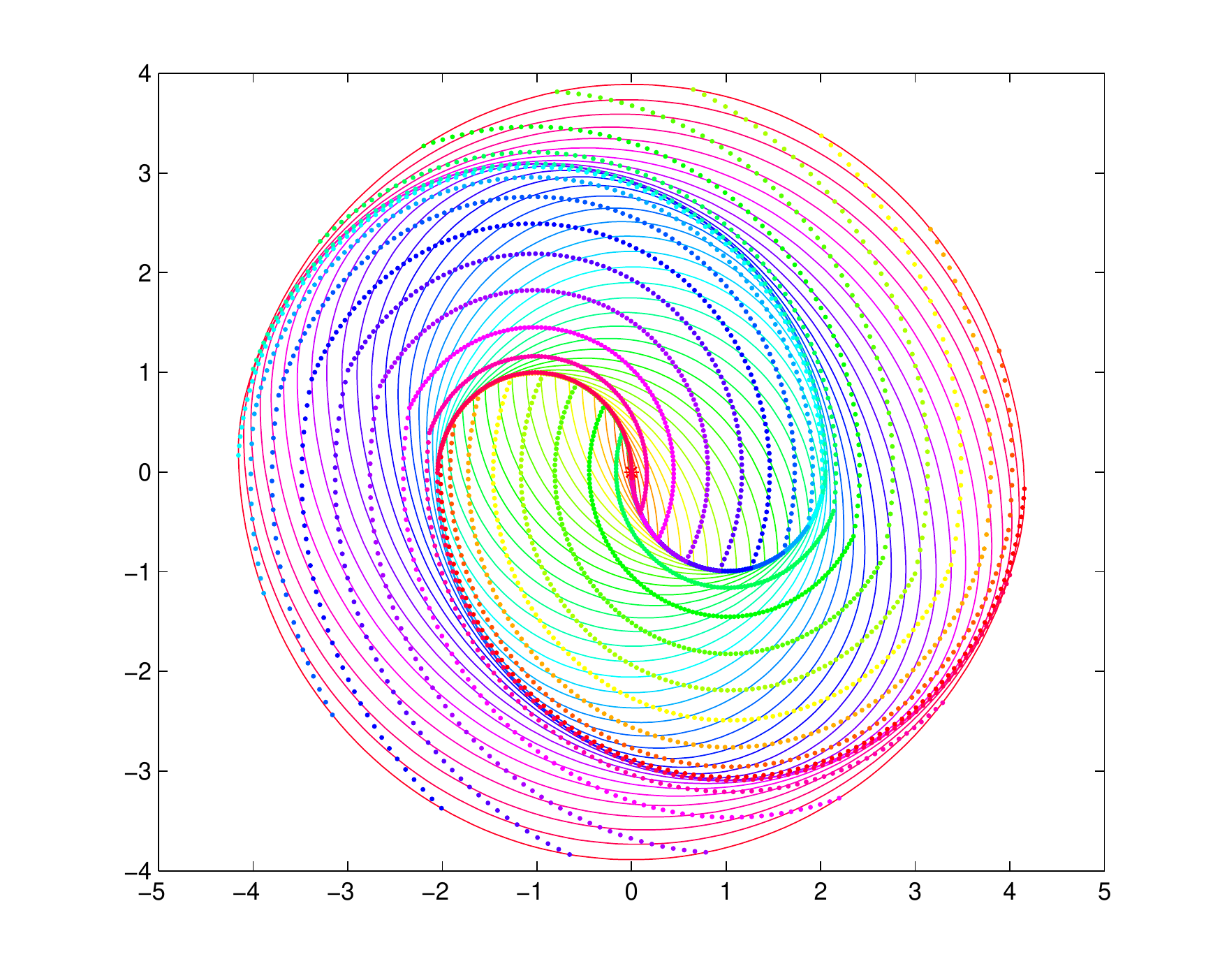}
   			\caption{Approximate optimal trajectories for Example~\ref{ex:2}a) resp.~b)}\label{fig:exam_32}
		\end{center}
	\end{figure}
\end{example} 
The following two examples exhibit smoothness of the support functions and would even allow 
for methods with order higher than two with respect to time discretization.
The first 
one has a special linear dynamics and is smooth, although the control set is a unit square.
\begin{example} 
	\label{ex:3a} 
	In the third linear two-dimensional example the reachable set for various end times $t$ 
	is always a polytope with four vertices and coinciding outer normals at its faces. 
	Therefore, it is a smooth example which would even justify the use of methods with higher order than 2 to
	compute the reachable sets (see \cite{BLcham,BL}).
	It is similar to Example~\ref{ex:counter_ex_1}, but has an additional column in matrix $B$ and is a variant of \cite[Example~2]{BLcham}.
	
	Again, we fix $t_f = 1$ as maximal time value and compute the result with $N=2$.
	We choose $N_{\mathcal{R}} = 50$ normed directions, since the reachable set has only four different
	vertices.
	\begin{equation}\label{example15}
	\begin{bmatrix}
	\dot x_1 \\[0.3em]
	\dot x_2 
	\end{bmatrix}=\begin{bmatrix}
	0 & -1 \\[0.3em]
	2 &  3  
	\end{bmatrix}\begin{bmatrix}
	x_1 \\[0.3em]
	x_2  
	\end{bmatrix}+\begin{bmatrix}
	1 & -1 \\[0.3em]
	-1 & 2  
	\end{bmatrix}\begin{bmatrix}
	u_1 \\[0.3em]
	u_2  
	\end{bmatrix},
	\end{equation}
	where $(u_1,\,u_2)^\top \in [-1,1]^2$. Let the origin be the target set  $\mathcal{S}$. 
	The fundamental solution matrix of the time-reversed dynamics of \eqref{example15} is given by
	\begin{align*}
	\Phi(t,\tau) & = \begin{bmatrix}
	2 e^{-(t-\tau)} - e^{-2(t-\tau)} 
	& e^{-(t-\tau)} - e^{-2(t-\tau)} \\[0.3em]
	-2 e^{-(t-\tau)} + 2 e^{-2(t-\tau)} 
	& -e^{-(t-\tau)} + 2 e^{-2(t-\tau)}
	\end{bmatrix}.
	\end{align*}
	This is a smooth example in the sense that the support function  for  the time-reversed
	set-valued dynamics of \eqref{example15}, 
	\begin{align*}
	\delta ^*(l,\Phi(t,\tau)\bar B(\tau)[-1,1]^2)=e^{-(t-\tau)}\vert l_1-l_2 \vert + e^{-2(t - \tau)}\vert l_1-2l_2 \vert,
	\end{align*}
	is smooth with respect to $\tau$  uniformly in $l \in S_1$.
	
	The analytical formula for the (time-continuous) minimum time function is as follows:
	\begin{equation*}
	\begin{aligned}
	T((x_1,x_2)^\top)=\max \bb{& t \colon \ t \geq 0 \text{ is the solution of one of the equations }\\
		& x_2=-2x_1\pm (e^{-t}-1),\,x_2=-x_1\pm 1/2(1-e^{-2t})}.
	\end{aligned}
	\end{equation*}
	A least squares approximation of the function $C h^{p}$ for the error term
	reveals ${C = 2.14475}$, $p = 0.8395$ for the set-valued combination method of order~1
	and $C = 23.9210$, $p= 1.7335$ (if $p=2$ is fixed, then $C=70.1265$)
	for the one of order~2. The values are similar to the expected ones from Remark \ref{Rem_errT}, 
	since the minimum time function (see Fig.~\ref{fig:3a}~(left)) is Lipschitz (see~\cite[Sec.~IV.1, Theorem~1.9]{BCD}).
	
	Similarly, another variant of this example with a one-dimensional control can be constructed by deleting
	the second column in matrix $B$. The resulting (discrete and continuous-time) reachable sets
	would be line segments. Thus, the algorithm would compute the fully discrete minimum time
	function on this one-dimensional subspace. The absence of interior points in the reachable
	sets is not problematic for this approach in contrary to common approaches based on the
	Hamilton-Jacobi-Bellman equation as shown in Example~\ref{ex:counter_ex_1}.
	\begin{table}[h]
	\small
		\begin{tabular}{|l|c|c|}
			\hline
			\mbox{ }\ \ \,$h$ & \textbf{Euler scheme \& Riemann sum} 
			& \textbf{Heun's scheme \& trapezoid rule} \\
			\hline
			$0.05$ & $0.170\phantom{0}$ & $0.1153\phantom{000}$ \\
			\hline
			$0.025$ & $0.095\phantom{0}$ & $0.0470\phantom{000}$ \\
			\hline
			$0.0125$ & $0.0599$ & $0.0133\phantom{000}$ \\
			\hline
			$0.00625$ & $0.0285$ & $0.0032\phantom{000}$ \\
			\hline
		\end{tabular}\\[2ex]
		\caption{Error  estimates for Example~\ref{ex:3a} 
			for methods of order 1 and 2} 
		\label{tab:3}
	\end{table}

	\begin{figure}[htp]
		\begin{center}
			\includegraphics[scale=0.35]{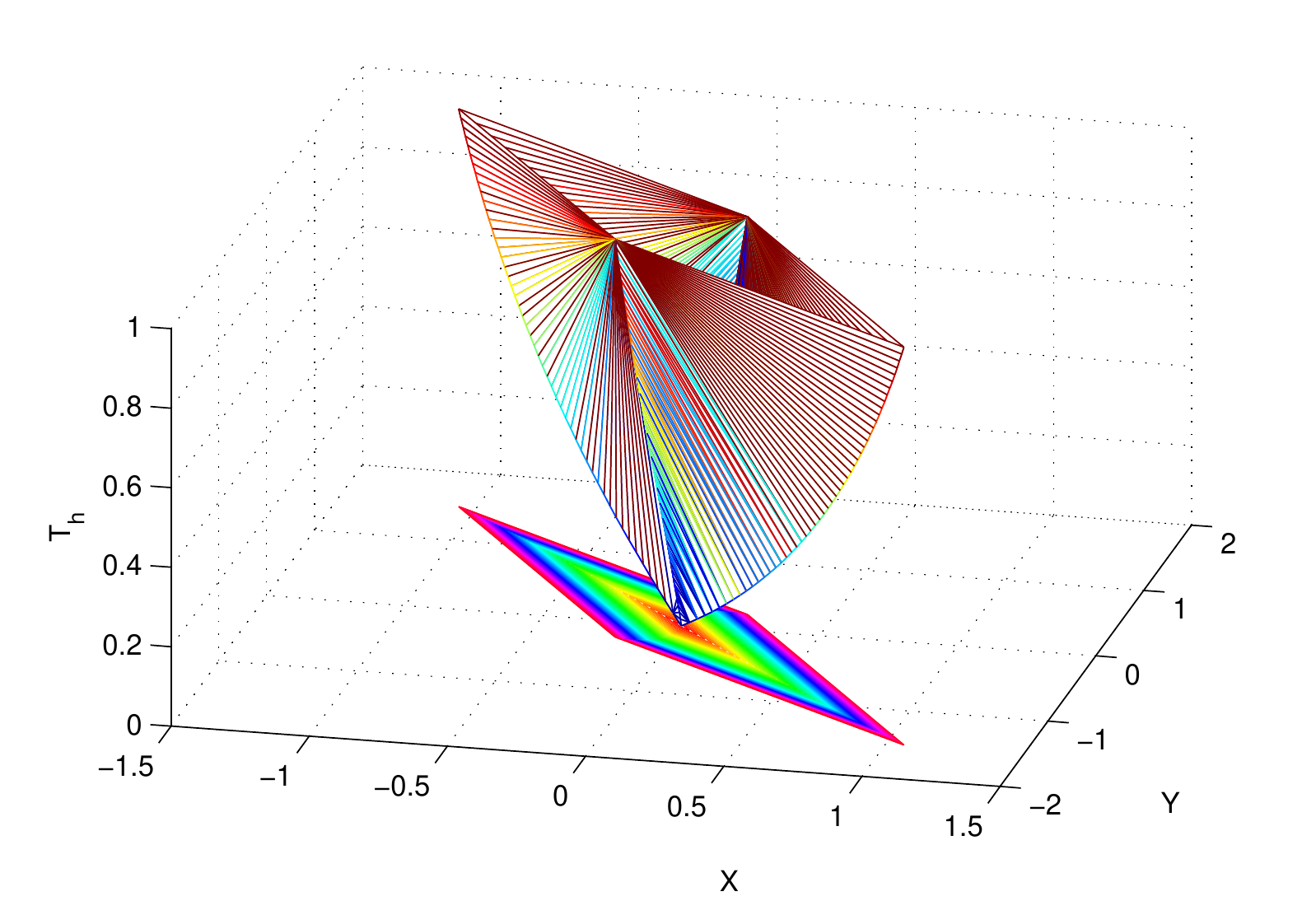}
			\ \,
			\includegraphics[scale=0.4]{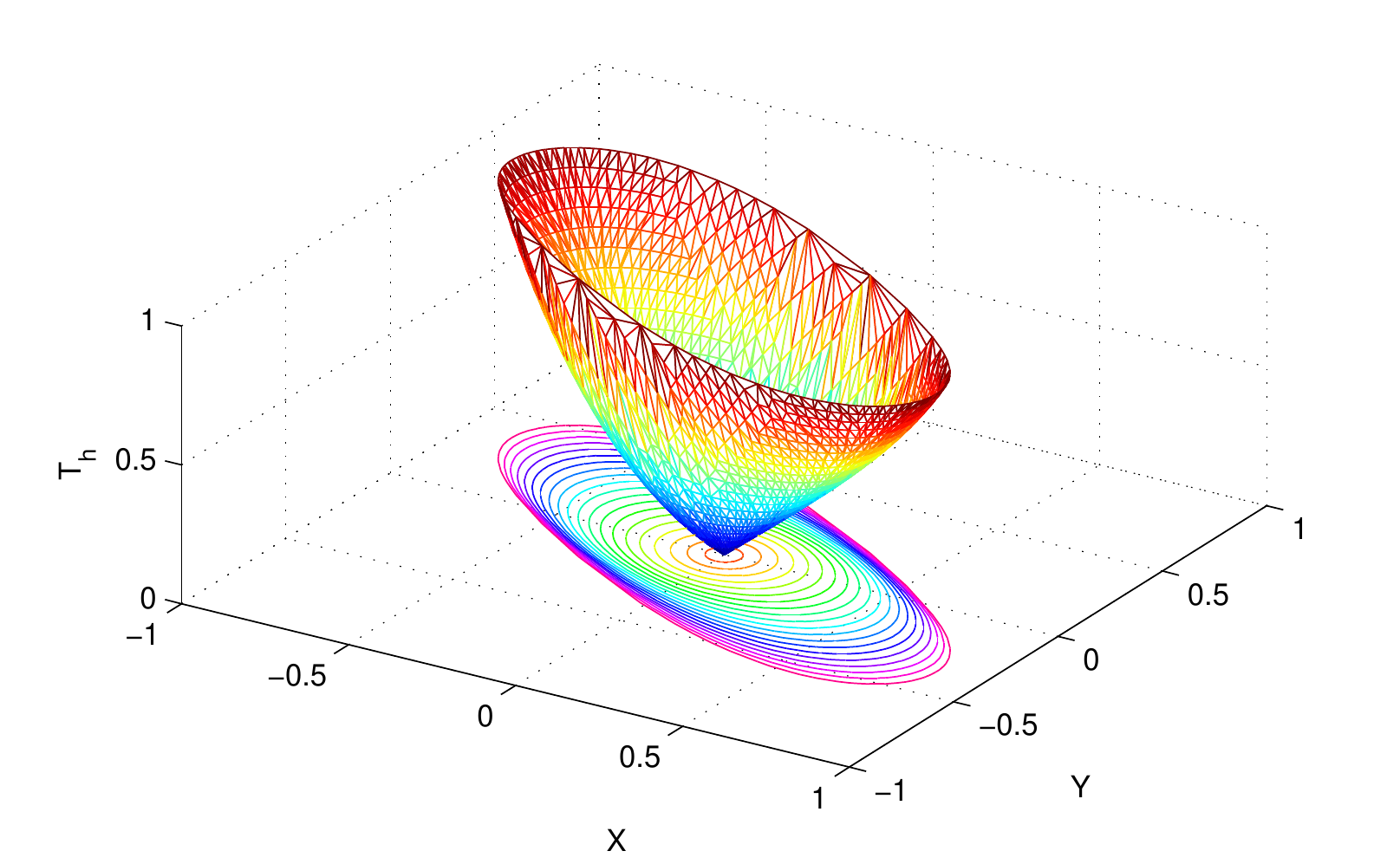}
   			\caption{ Minimum time functions for Examples~\ref{ex:3a} and~\ref{ex:3b}} 
			\label{fig:3a}
		\end{center}
	\end{figure}
\end{example} 
 
The next example involves a ball as control set and leads naturally to a smooth problem.
\begin{example} 
	\label{ex:3b}
	The following smooth example is very similar to the previous example.
	It  is given in \cite[Example~4.2]{BBCG}, \cite[Example~4.4]{BL} 
	\begin{equation}\label{example16}
	\begin{bmatrix}
	\dot{x}_1 \\[0.3em]
	\dot{x}_2  
	\end{bmatrix}=\begin{bmatrix}
	0 & -1 \\[0.3em]
	2 & 3  
	\end{bmatrix}\begin{bmatrix}
	x_1 \\[0.3em]
	x_2  
	\end{bmatrix}+B_1(0)
	\end{equation}
	and uses a ball as control set. This is a less academic example than Example~\ref{ex:3a} (in which the matrix $B(t)$ was carefully chosen), since a ball as control
	set often allows the use of higher order methods for the computation of reachable sets
	(see~\cite{BL,B}). Here, no analytic formula for the minimum time function is available
	so that we can study only numerically the minimum time function (see Fig.~\ref{fig:3a}~(right)).
	Obviously, the support function is again smooth with respect to $\tau$ uniformly in all
	normed directions $l$, since
	\begin{align*}
	\delta ^*(l,\Phi(t,\tau) B_1(0) & = \| \Phi(t,\tau)^\top l \|.
	\end{align*}
\end{example} 
\subsection{A nonlinear example}
\label{subsec_nonlin}

The following special bilinear example with convex reachable sets may provide the hope to extend 
our approach to some class of nonlinear dynamics.
We approximate the time-reversed dynamics of Example~\ref{ex:4} 
by Euler's and Heun's method.

\begin{example} 
	\label{ex:4}
	The nonlinear dynamics is one of the examples in \cite{GL}.
	\begin{equation}\label{example5}
	\dot{x}_1=-x_2+x_1 u,\,\,\dot{x}_2=x_1+x_2 u,\,\,u\in [-1,1].
	\end{equation}
	With this dynamics, after computing the true minimum time function we observe that $T(\cdot)$ is Lipschitz continuous and its sublevel set, which is exactly the reachable set at the corresponding time, satisfies the required properties. The target set  $\mathcal{S}$  is $B_{0.25}(0)$.  \\
	
	We fix $t_f = 1$ as maximal computed value for the minimum time function
	and $N=2$.
	Estimating the error term $C h^{p}$ in Table~\ref{tab:4} by least squares approximation yields the values
	$C = 0.3293133$, $p= 1.8091$ for the set-valued Euler method
	and $C = 0.5815318$, $p= 1.9117$ for the Heun method. 
	
	The unexpected good behavior of Euler's method stems from the specific behavior of trajectories.
	Although the distance of the end point of the Euler iterates for halfened step size to the
	true end point is halfened, but the distance of the Euler iterates to the boundary of the 
	true reachable set is almost shrinking by the factor 4 due to the specific tangential
	approximation. 
	In Fig.~\ref{exam_5} the Euler iterates are marked with \textasteriskcentered\ in red color,
	while Heun's iterates are shown with $\circ$ marks in blue color. The symbol $ \bullet $ marks the end point of the corresponding true solution.
	
	Observe that the dynamics originates from the following system in polar coordinates
	\begin{equation*} 
	\dot{r}=r u,\,\,\dot{\varphi}=1 ,\,\,u\in [-1,1].
	\end{equation*}
	Hence, the reachable set will grow exponentially with increasing time.
	\begin{table}[h]
	\small
		\captionsetup{width=.95\linewidth}%
		\begin{tabular}{|l|c|c|c|}
			\hline
			\mbox{ }\ \,$h$ & $N_{\mathcal{R}}$  & \textbf{set-valued Euler scheme} 
			& \textbf{set-valued Heun's scheme} \\
			\hline
			$0.5$ & $ 50 $&$0.0848\phantom{00}$ & $0.1461\phantom{00}$ \\
			\hline
			$0.1$ & $100 $&$0.0060\phantom{00}$ & $0.0076\phantom{00}$ \\
			\hline
			$0.05$ &$200$  & $0.0015\phantom{00}$ & $0.0020\phantom{00}$ \\
			\hline
			$0.025$& $400$ & $0.00042\phantom{0}$ & $0.000502$ \\
			\hline
			$0.0125$ &$800$ & $0.000108$ & $0.000126$ \\
			\hline
		\end{tabular}\\[2ex]
		\caption{Error estimates for Example~\ref{ex:4} with set-valued methods
			of order 1 and 2}
		\label{tab:4}
	\end{table}

	\noindent
        The minimum time function for this example is shown in Fig.~\ref{exam_5}.
 	
	\begin{figure}[htp]
		\captionsetup{width=.95\linewidth}%
		\begin{center}
			\includegraphics[scale=0.3]{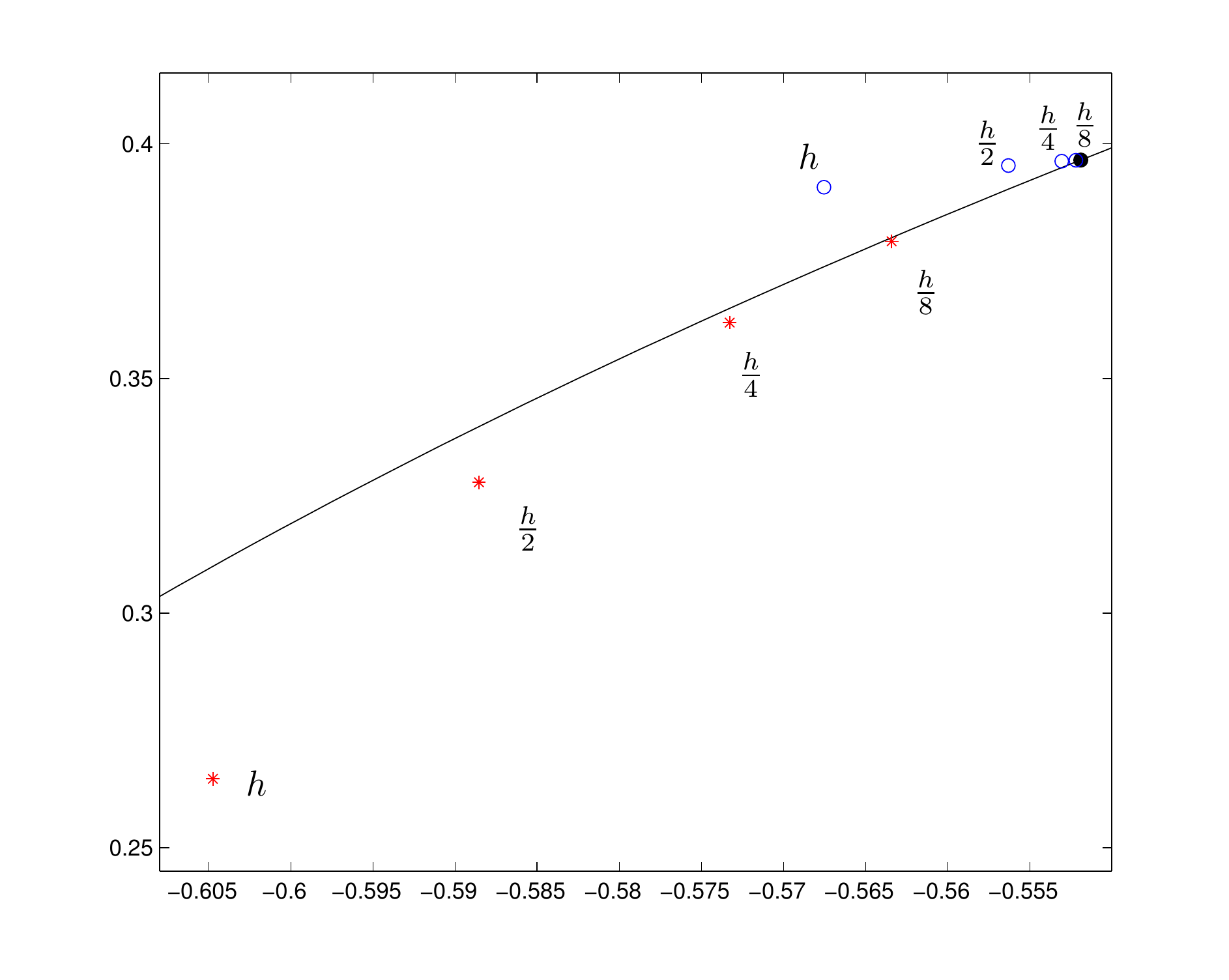}
			\includegraphics[width=2.53in,height=2.4in]{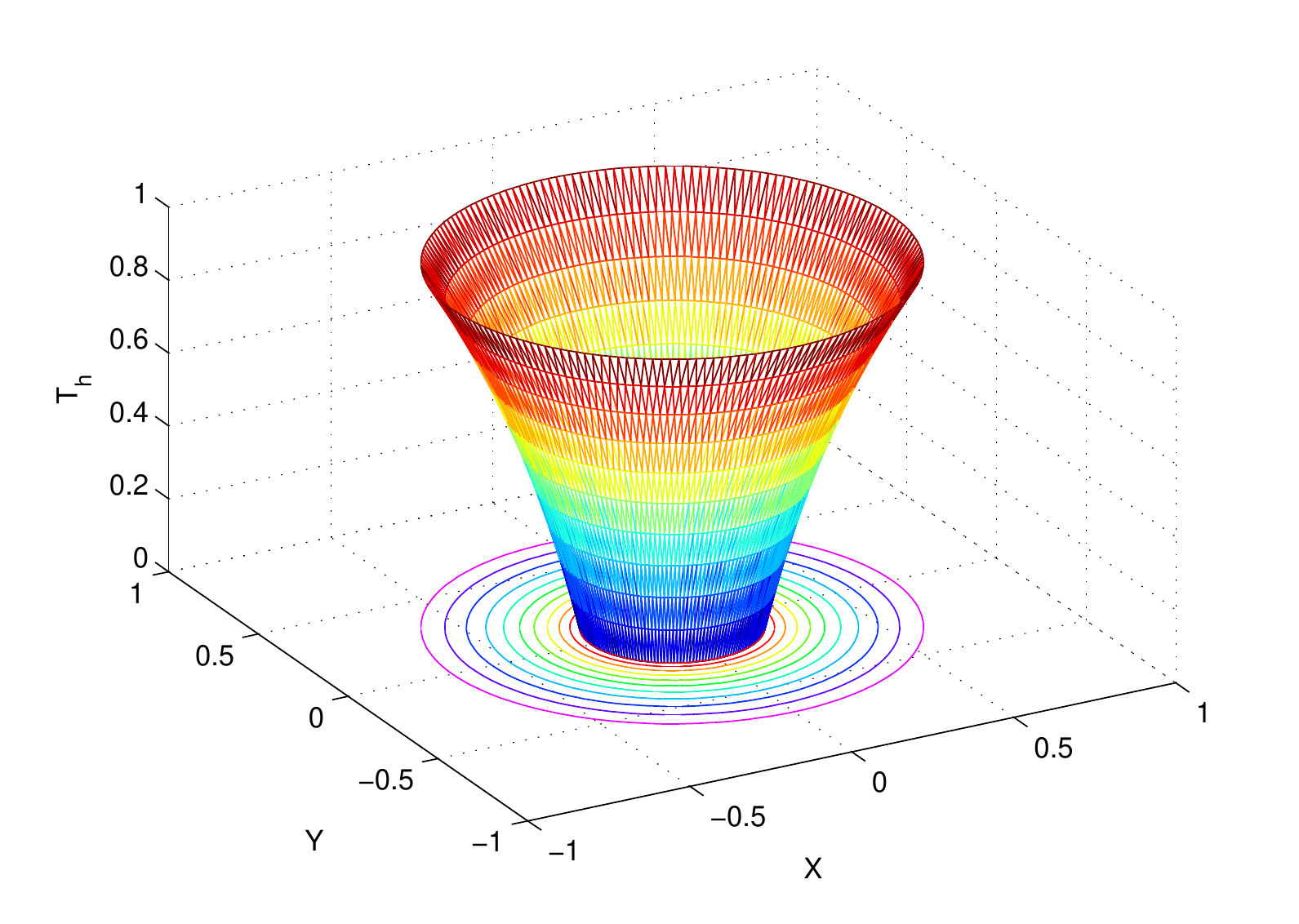}        
			\caption{Euler and Heun's iterates,  minimum time function for Example~\ref{ex:4} resp.} 
			\label{exam_5}
		\end{center}
	\end{figure}
\end{example} 

\subsection{Non-strict expanding property of reachable sets}
\label{subsec_non_exp_prop}%

The next example violates the continuity of the minimum time function (the dynamics is not normal).
Nevertheless, the proposed Algorithm \ref{algorithm} is able to provide a good
approximation of the discontinuous minimum time function.

This example also shows that boundary points of the reachable set
can no longer be characterized via time-minimal points
(compare Propositions~\ref{prop:bd_descr_monotone_case_w_level_set} 
and~\ref{prop:bd_descr_w_level_set}), if the strict expanding property 
of (the union of) reachable sets is not satisfied.

\begin{example} 
	Consider the dynamics
	\label{ex:counter_ex_1}
	\begin{equation}
	\begin{bmatrix}
	\dot{x}_1 \\[0.3em]
	\dot{x}_2  
	\end{bmatrix}=\begin{bmatrix}
	0 & -1 \\[0.3em]
	2 & 3  
	\end{bmatrix}\begin{bmatrix}
	x_1 \\[0.3em]
	x_2  
	\end{bmatrix}+ u_1 \begin{bmatrix}
	1  \\[0.3em]
	-1   
	\end{bmatrix}
	\end{equation}
	with $u_1 \in U = [-1,1]$, $\mathcal{S} = \{0\}$ and $t \in I=[0, t_f]$.
	
	The Kalman rank condition yields
	$
	\rk\Big[ B, A B \Big] = 1 < 2,	
	$
	so that the normality of the system is not fulfilled. The fundamental system $\Phi(t,\tau)$ (for the time-reversed system)
	is the same as in Example~\ref{ex:3a} so that 
	\begin{align*}
	& \delta ^*(l,\Phi(t,\tau)\bar B(\tau)[-1,1])=e^{\tau-t}\vert l_1-l_2 \vert 
	= e^{\tau-t} \delta ^*(l, V), \\
	\intertext{with the line segment $V = \co(\begin{bmatrix}
		-1  \\[0.3em]
		1   
		\end{bmatrix}, \begin{bmatrix}
		1  \\[0.3em]
		-1   
		\end{bmatrix})$. 
		Since}
	& \int_0^t \delta^*(l, \Phi(t,\tau)\bar B(\tau)[-1,1]) d\tau 
	= e^{\tau-t} \bigg|_{\tau=0}^t \cdot \delta ^*(l, V) \\
	= & (1 - e^{-t}) \cdot \delta ^*(l, V) = \delta ^*(l,  (1 - e^{-t}) V), \\
	\mathcal{R}(t) & = \int_0^t \Phi(t,\tau)\bar B(\tau)[-1,1] d\tau =  (1 - e^{-t}) V.
	\end{align*}
	Hence, the reachable set is an increasing line segment (and always part of the same line
	in $\R^2$, i.e., it is one-dimensional so that the interior is empty). Clearly,
	both inclusions 
	\begin{align}
	\mathcal{R}(s) \subset \mathcal{R}(t) \quad\text{or}\quad
	\RSU(s) \subset \RSU(t)
	\end{align}
	i.e., \eqref{ex:relaxed_expand}, 
	hold, but not the strictly expanding property of $\overline{\mathcal{R}}(\cdot)$ 
	on $[t_0, t_f]$ in Assumptions \ref{standassum}(iv) and~(iv)', i.e.,
	\begin{align}
	\overline{\mathcal{R}}(t_1) & \subset \inter \overline{\mathcal{R}}(t_2) \text{\ for all $t_0\le t_1<t_2\le t_f$, where \rule{0ex}{4ex}} \label{eq:strict_exp_prop} \\
	\overline{\mathcal{R}}(t) & = \begin{cases}
	\mathcal{R}(t) & \text{for Assumption (iv)}, \\
	\RSU(t) & \text{for Assumption (iv)'}.
	\end{cases}\nonumber
	\rule[-5.5ex]{0ex}{5.5ex}
	\end{align}
	The strict inclusion only holds in the relative interior. By \cite[Sec.~IV.1, Proposition~1.2]{BCD} the minimum time function is discontinuous (it has infinite values outside the line segment).
	
	The plots of the two continuous-time reachable sets $\mathcal{R}(t)$ for $t=1,2$ together with the true minimum time function (in red)
	and its discrete analogue (in green) obtained by the Euler scheme with $h = 0.025$ are shown
	in Fig.~\ref{fig:counter_exam}:
	\begin{figure}[htp]
		\begin{center}
			\includegraphics[scale=0.34]{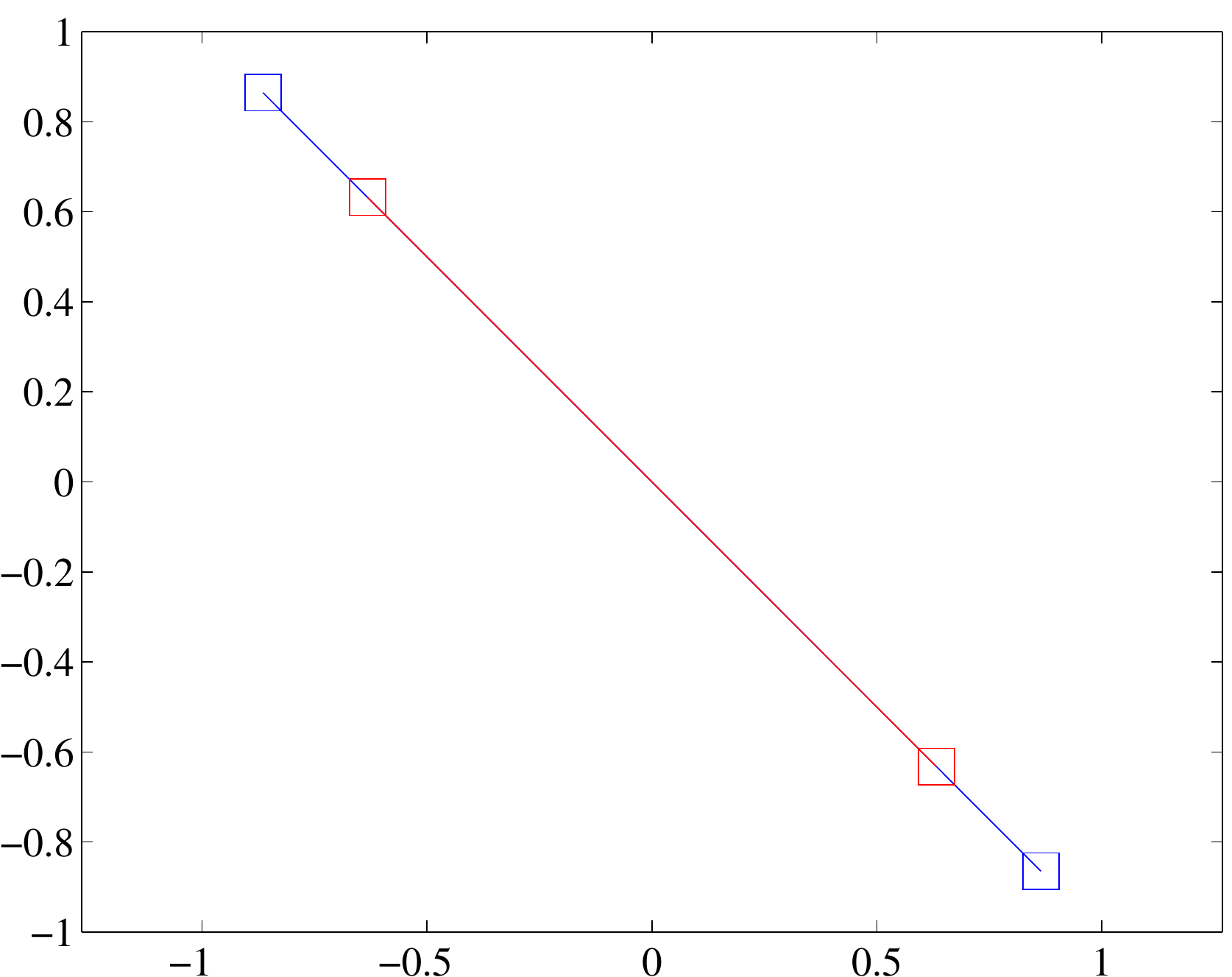}
			\quad 
			\includegraphics[scale=0.35]{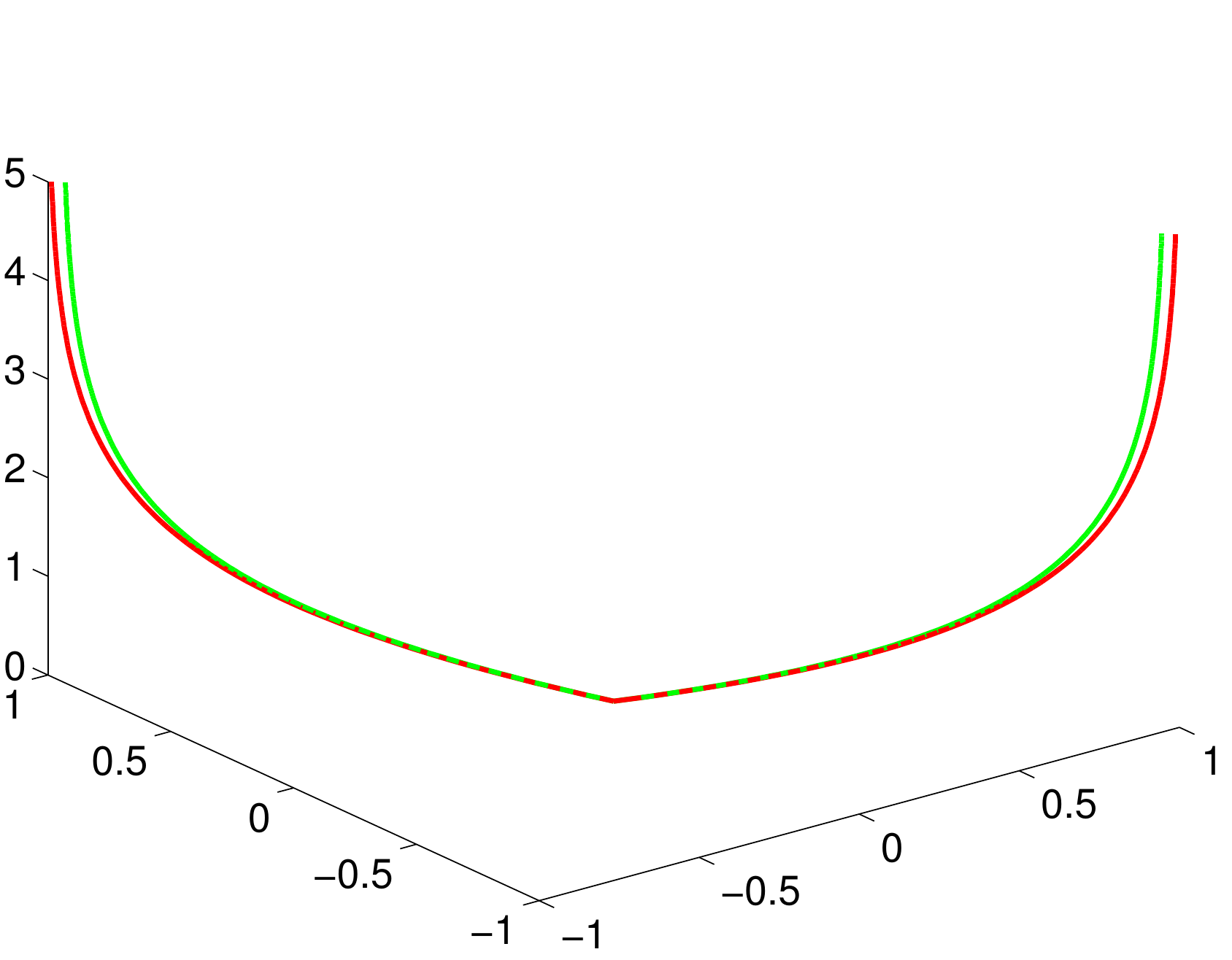}
		\end{center}
   		\caption{Reachable sets and minimum time functions for Example~\ref{ex:counter_ex_1}}
		\label{fig:counter_exam}
	\end{figure}
       
	The two red points are the end points of the line segment for a smaller time $t_1 = 1$,
	the two blue points are the end points of the line segment for a larger time $t_2 = 2 > t_1$.
	The blue line segment is the reachable set for time $t_2$ (also the reachable set up to
	time $t_2$). \\
	All four points are on the boundary of the blue set $\mathcal{R}(t_2)$, but the minimum time to reach
	the two blue points is $t_2$, while the minimum time to reach
	the two red points is $t_1 < t_2$ which is a contradiction to 
	Proposition \ref{prop:bd_descr_w_level_set}.
	 
\end{example} 

\subsection{Problematic examples}

The first two examples show linear systems with hidden stability properties so that the discrete
reachable sets converge to a bounded convex set if the time goes to infinity (or is large enough
in numerical experiments). For larger end times the numerical calculation gets more demanding,
since the step size must be chosen small enough according to Proposition \ref{errTt2}.
The remaining part of the subsection 
contain examples that violate Assumption \ref{standassum}(iv) and (iv')
from Proposition~\ref{prop:bd_descr_monotone_case_w_level_set}.
The examples demonstrate that a target or a control set not containing the origin 
(as a relative interior point) might lead to non-monotone behavior
of the (union of) reachable sets. In all of these examples the union of reachable sets is no longer convex.

\begin{example}\label{ex:n1}
	We consider the following time-dependent linear dynamics:
	\begin{equation}\label{examplen1}  
	\dot{x}_1=-x_2,\,\,\dot{x}_2=x_1 - \frac{1}{t^2} u,\,\,u\in [-1,1].
	\end{equation}
	The reachable sets converge towards a final, bounded, convex set due to
	the scaling factor $\frac{1}{t^2}$ in the matrix $B(t)$, see~Fig.~\ref{fig:exam_n1}~(left).
	From a formal point of view the strict expanding condition \eqref{inclusionR} in Proposition \ref{errTt2}
	is satisfied, but the positive number $\varepsilon$ tends to zero for increasing end time. On the other hand
	we would stop the calculations if the Hausdorff distance of two consequent discrete reachable sets 
	is below a certain threshold.
	\begin{figure}[htp]
		\begin{center}
			\includegraphics[scale=0.34]{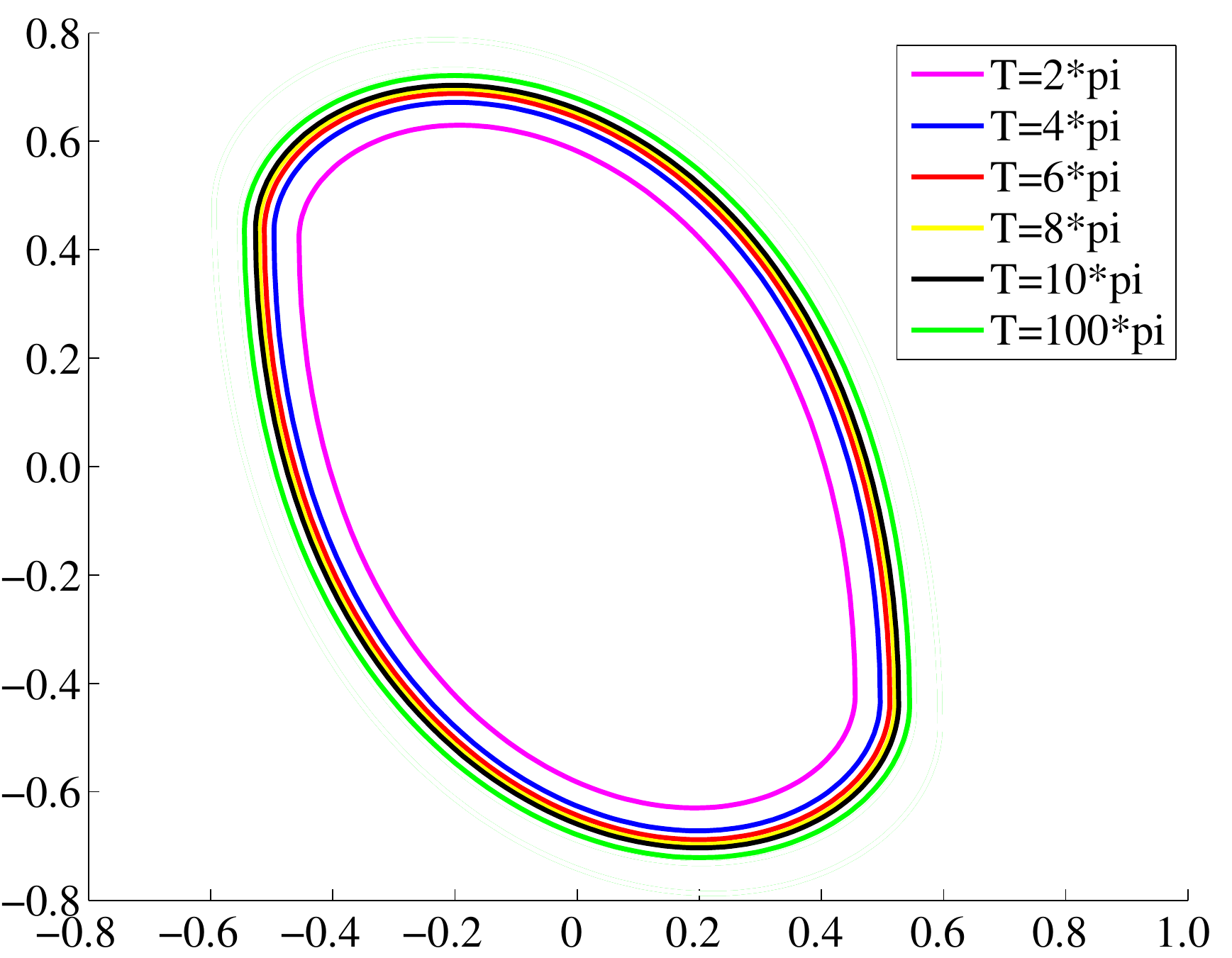}
			\ \,
			\includegraphics[scale=0.34]{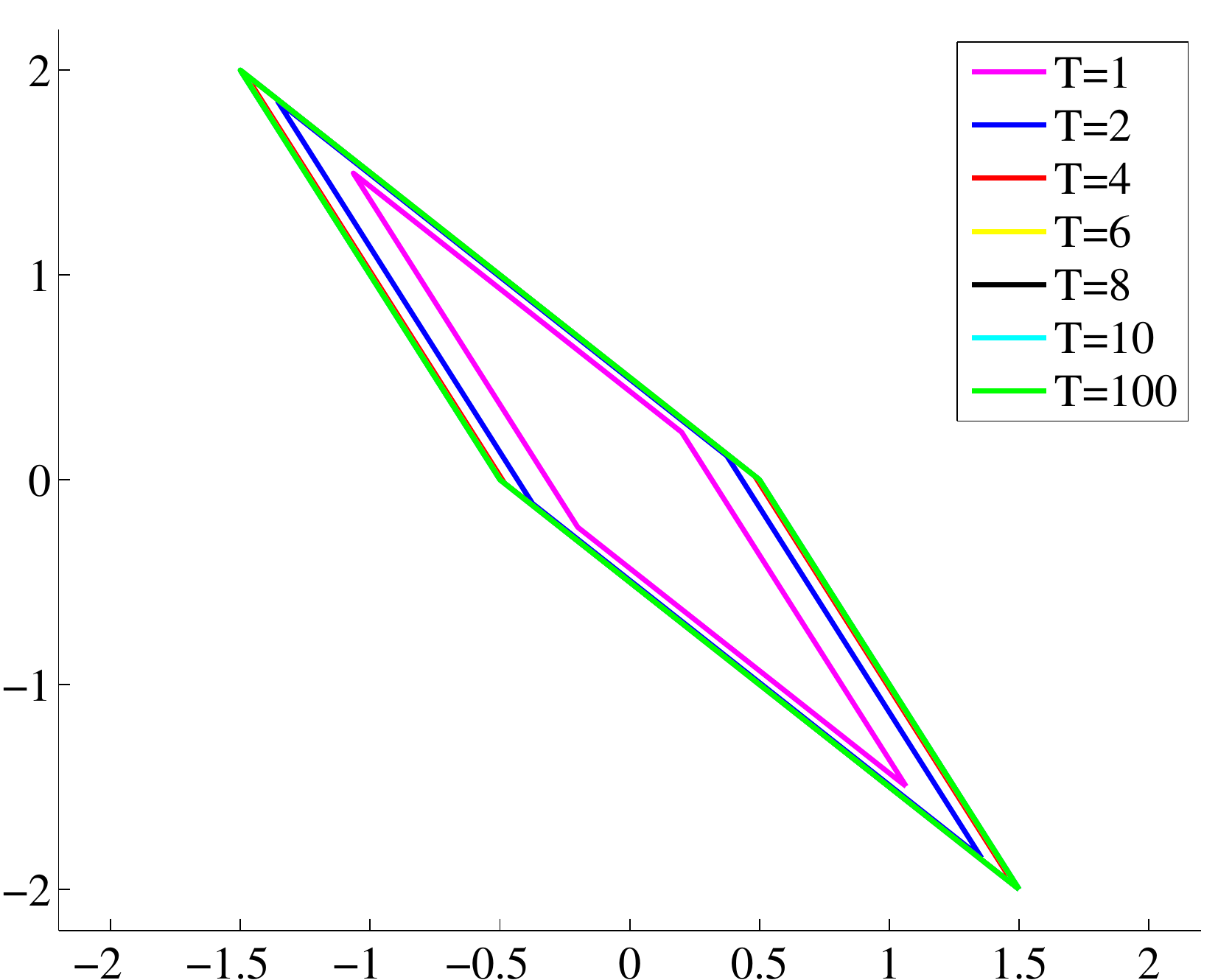}
			\caption{Reachable sets with various end times $t_f$ for Examples~\ref{ex:n1} and \ref{ex:n2}} 
			\label{fig:exam_n1}
		\end{center}
	\end{figure}
\end{example} 

\begin{example}\label{ex:n2} 
	We reconsider Example~\ref{ex:3a} on the larger time interval $[t_0,t_f]=[0,100]$. $\bar{A}$ has negative eigenvalues -1 and -2. 
	Hence, the reachable sets converge towards a final, bounded, convex set, 
	see~Fig.~\ref{fig:exam_n1}~(right). We experience the same numerical problems as 
	in Example~\ref{ex:n1}.
\end{example} 
         
\begin{example} 
	\label{ex:n4}
	Let the dynamics be given by
	\begin{equation}\label{examplen4}  
	\dot{x}_1=x_2 + u_1,\,\,\dot{x}_2=-x_1 + u_2,\,\,u\in B_1(0).
	\end{equation}
	In case a) the reachable sets for a given end time are always balls around the origin
	(see~Fig.~\ref{fig:exam_n4} (left)),
	if the target set is chosen as the origin.
	In case b) the point $(2,2)^\top$ is considered as target set. Fig.~\ref{fig:exam_n4}~(right) shows
	that the union of reachable sets is no longer convex.
	\begin{figure}[htp]
		\begin{center}
			\includegraphics[scale=0.37]{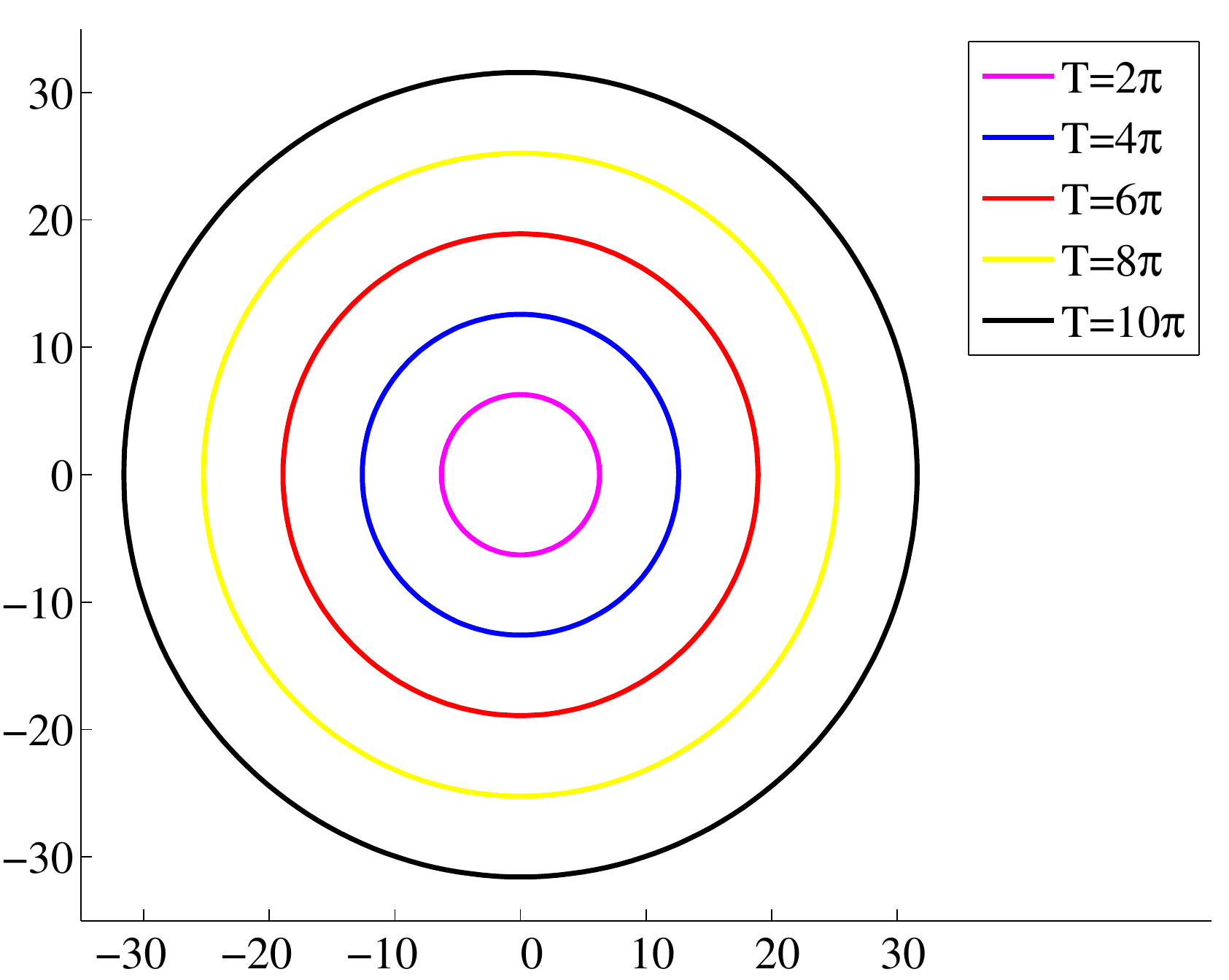}
			\quad
			\includegraphics[scale=0.37]{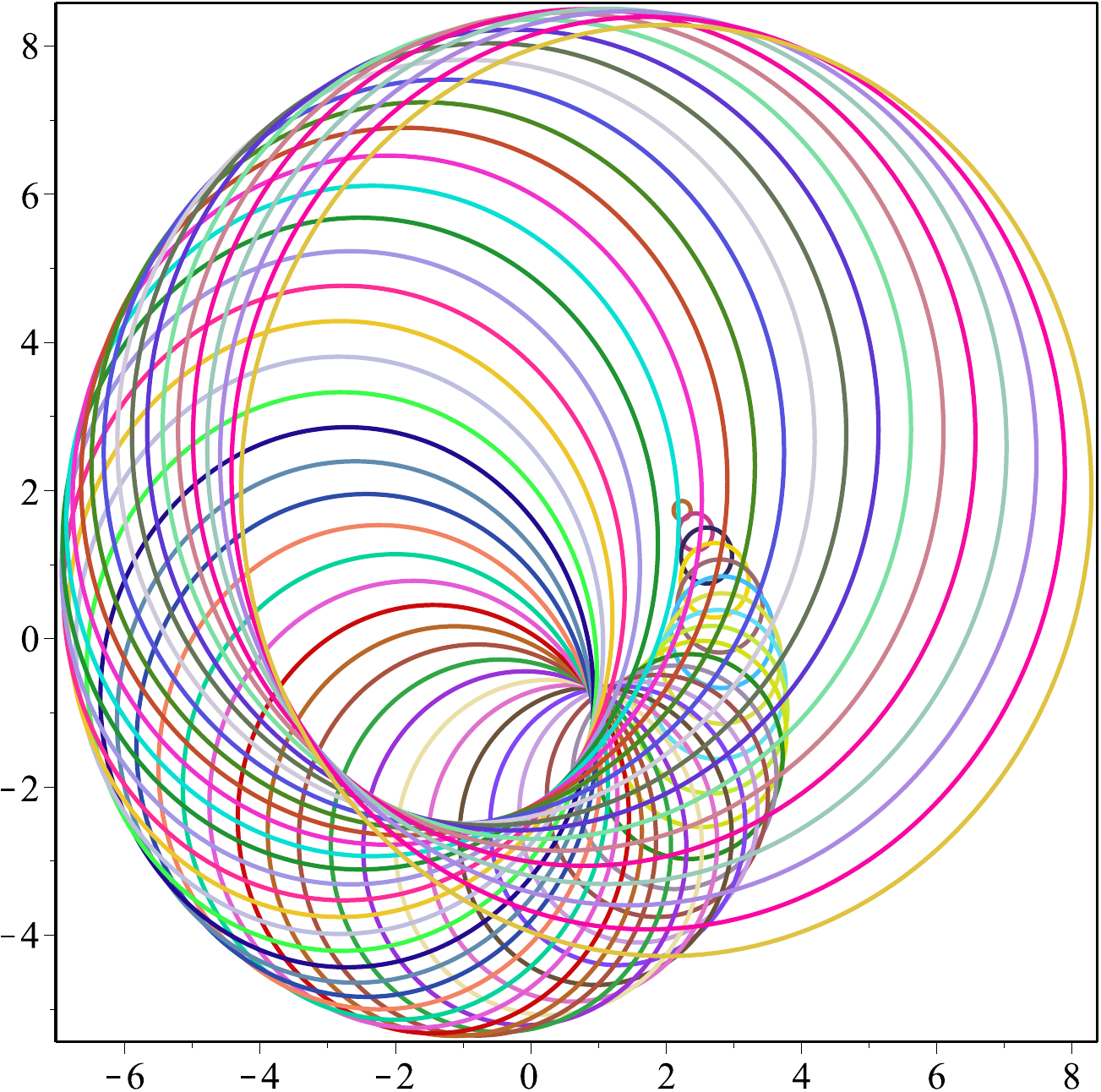}
			\caption{Reachable sets with various end times and different target sets for Example~\ref{ex:n4}} 
			\label{fig:exam_n4}
		\end{center}
	\end{figure}
\end{example} 
 
\begin{example}\label{ex:n5}
	Let us reconsider the dynamics \eqref{example3} of Example \ref{ex:2}, i.e.,
	\begin{equation*}
	\dot{x}_1=x_2,\,\dot{x}_2=u,\,\,u\in U.
	\end{equation*}
	In the first case, let $\mathcal{S}=\bb{0},\,\,U=[0,1]$. From the numerical calculations, we observe that $\mathcal{R}(t)$, $\mathcal{R}_{\le}(t)$ are still convex and satisfy \eqref{ex:relaxed_expand} in Remark \ref{rem:inclusion}, but violate the strictly expanding property \eqref{eq:strict_exp_prop} as shown in Fig.~\ref{fig:exam_n5}~(left). In the other case, $U=[1,2]$ is chosen. The convex reachable set $\mathcal{R}(t)$ is not only enlarging, but also moving which results in the nonconvexity of $\mathcal{R}_{\le}(t)$. Moreover, both \eqref{ex:relaxed_expand} in Remark \ref{rem:inclusion} and \eqref{eq:strict_exp_prop} are not fulfilled in this example as depicted in Fig.~\ref{fig:exam_n5}~(right).
	\begin{figure}[ht!]
		\begin{center}
			\includegraphics[scale=0.34]{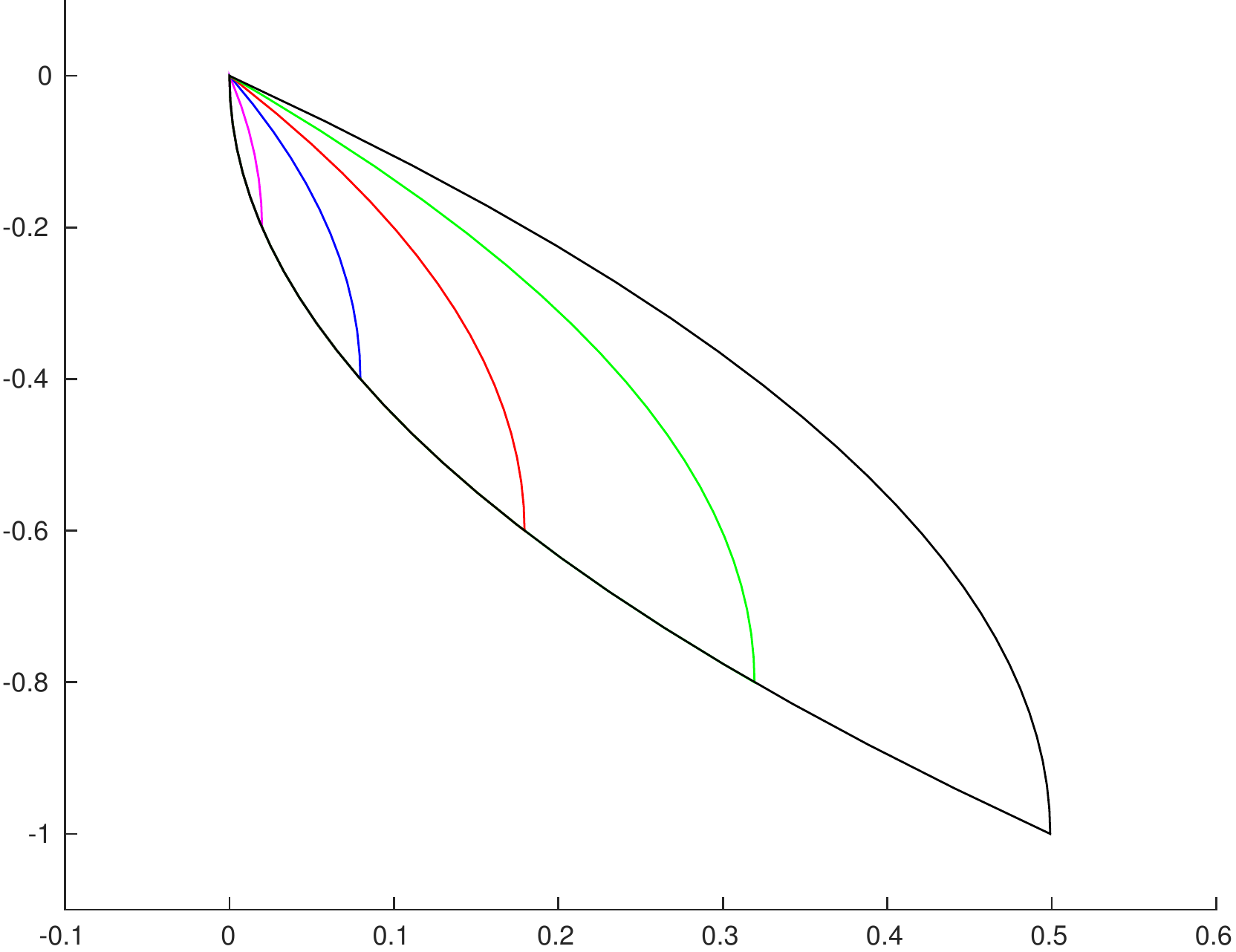}
			\quad
			\includegraphics[scale=0.34]{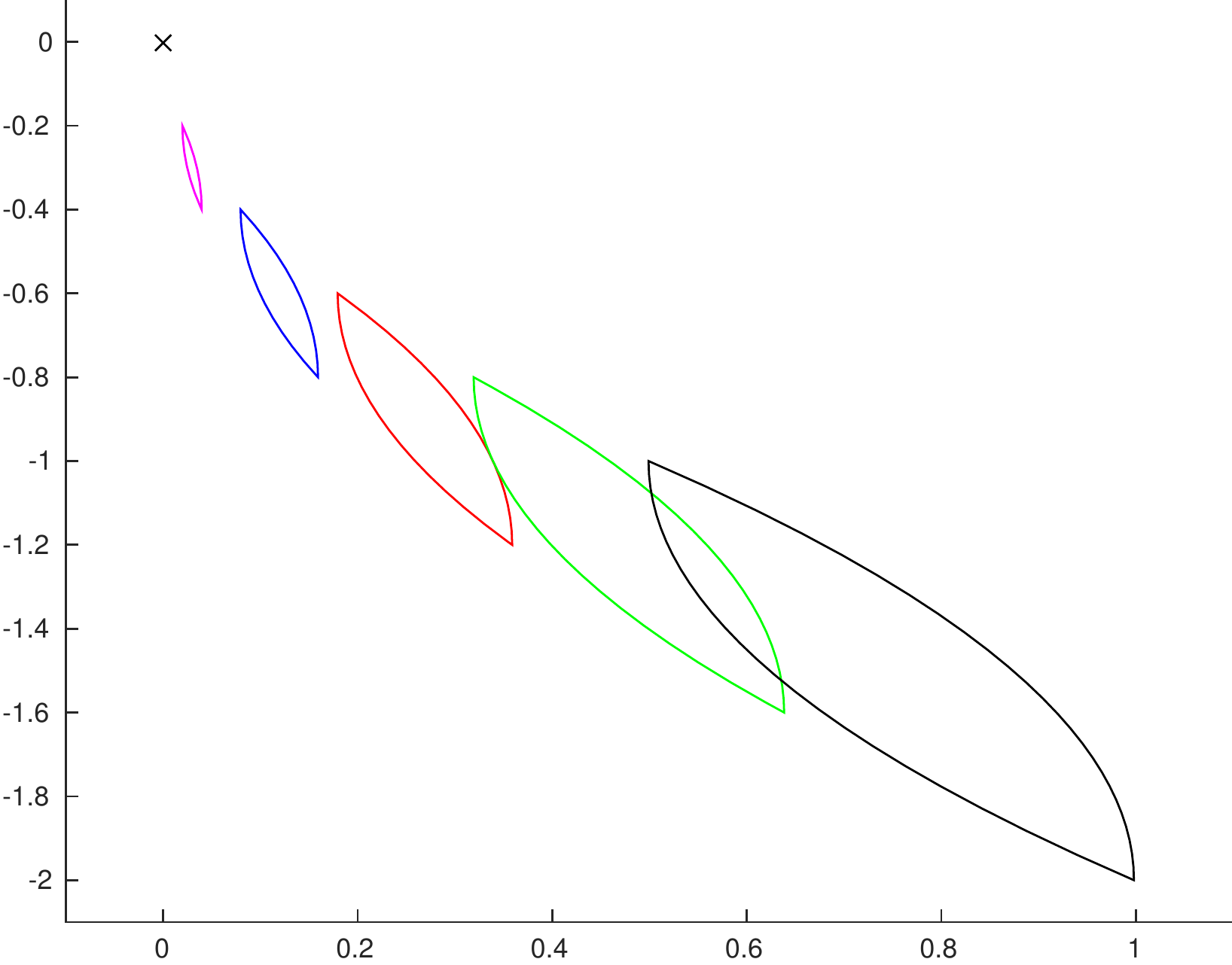}
			\caption{Reachable sets with various end times and different control sets for Example~\ref{ex:n5}} 
			\label{fig:exam_n5}
		\end{center}
	\end{figure}
	
\end{example}

\section{Conclusions}
\label{sec:concl}

Although the underlying set-valued method approximating reachable sets in linear control 
problems is very efficient,
the numerical implementation is a first realization only and can still be considerably improved.
Especially, step~3 in Algorithm \ref{algorithm}
can be computed more efficiently as in our
test implementation. Furthermore, higher order methods like the set-valued Simpson's rule
combined with the Runge-Kutta(4) method are an interesting option in examples where the underlying
reachable sets can be computed with higher order of convergence than 2, especially if 
the minimum time function is Lipschitz. But even if it is merely H\"older-continuous
with $\frac{1}{2}$, the higher order in the set-valued quadrature method
can balance the missing regularity of the minimum time function and improves the error
estimate.
We are currently working on extending this first approach for linear control problems
without the monotonicity assumption on reachable sets and for nonlinear control problems.
\section*{Acknowledgements}
The authors want to express their thanks to Giovanni Colombo, especially for pointing us to Attouch's theorem, and to Lars Gr\"une. Both of them supported us with helpful suggestions and motivating questions. They are also grateful to Matthias Gerdts about his comments to optimal control.


\medskip
Received xxxx 20xx; revised xxxx 20xx.
\medskip


\begin{thebibliography}{99}
\bibitem{ABGL-13}
\newblock W.~Alt, R.~Baier, M.~Gerdts and F.~Lempio,
\newblock Approximations of linear control problems with bang-bang solutions,
\newblock \emph{Optimization}, \textbf{62} (2013), 9--32.

\bibitem{Alth}
\newblock M.~Althoff,
\newblock \emph{Reachability Analysis and its Application to the Safety
  Assessment of Autonomous Cars},
\newblock PhD thesis, Fakult\"{a}t f\"{u}r Elektrotechnik und
  Informationstechnik, Technische Universit\"{a}t M\"{u}nchen, Munich, Germany,
  2010, 221 S.
 
\bibitem{ABS-P}
\newblock J.-P. Aubin, A.~M. Bayen and P.~Saint-Pierre,
\newblock \emph{Viability Theory. New Directions},
\newblock 2nd edition,
\newblock Springer, Heidelberg, 2011.

\bibitem{AC}
\newblock J.-P. Aubin and A.~Cellina,
\newblock \emph{Differential Inclusions}, 
\newblock vol. 264 of Grundlehren der mathematischen Wissenschaften, Springer-Verlag, Berlin--Heidelberg--New York--Tokyo, 1984.

\bibitem{Aum}
\newblock R.~J. Aumann,
\newblock Integrals of set-valued functions,
\newblock \emph{J.~Math.\ Anal.\ Appl.}, \textbf{12} (1965), 1--12.

\bibitem{BPhD}
\newblock R.~Baier,
\newblock \emph{Mengenwertige {I}ntegration und die diskrete {A}pproximation
  erreichbarer {M}engen},
\newblock PhD thesis, Universit\"at Bayreuth, 1994, xix+202 S.
 

\bibitem{B}
\newblock R.~Baier,
\newblock Selection strategies for set-valued {R}unge-{K}utta methods,
\newblock in \emph{Numerical Analysis and Its Applications, Third International
  Conference, NAA 2004, Rousse, Bulgaria, June 29 - July 3, 2004, Revised
  Selected Papers} (eds. Z.~Li, L.~G. Vulkov and J.~Wasniewski), vol. 3401 of
  Lecture Notes in Comput.\ Sci., Springer, Berlin--Heidelberg, 2005, 149--157.
 

\bibitem{BBCG}
\newblock R.~Baier, C.~B{\"u}skens, I.~A. Chahma and M.~Gerdts,
\newblock Approximation of reachable sets by direct solution methods of optimal
  control problems,
\newblock \emph{Optim.\ Methods Softw.}, \textbf{22} (2007), 433--452.

\bibitem{BLcham}
\newblock R.~Baier and F.~Lempio,
\newblock Approximating reachable sets by extrapolation methods,
\newblock in \emph{Curves and Surfaces in Geometric Design. Papers from the
  Second International Conference on Curves and Surfaces, held in
  Chamonix-Mont-Blanc, France, July 10--16, 1993} (eds. P.~J. Laurent, A.~L.
  M\'{e}haute\'{e} and L.~L. Schumaker), A K Peters, Wellesley, 1994, 9--18.
 
\bibitem{BLe}
\newblock R.~Baier and T.~T.~T. Le,
\newblock Construction of the minimum time function via reachable sets 
  of linear control systems. Part 1: error estimates, part 2: numerical computations, 
  preprint, Dec 2015.
\newblock \arXiv{1512.08630} and \arXiv{1512.08617}.
 

\bibitem{BL}
\newblock R.~Baier and F.~Lempio,
\newblock Computing {A}umann's integral,
\newblock in \emph{Modeling Techniques for Uncertain Systems, Proceedings of a
  Conference held in Sopron, Hungary, July 6--10, 1992} (eds. A.~B. Kurzhanski
  and V.~M. Veliov), vol.~18 of Progress in Systems and Control Theory, Birkh{\"a}user, Basel, 1994, 71--92.
 

\bibitem{BCD}
\newblock M.~Bardi and I.~Capuzzo-Dolcetta,
\newblock \emph{Optimal Control and Viscosity Solutions of
  {H}amilton-{J}acobi-{B}ellman Equations},
\newblock Systems \& Control: Foundations \& Applications, Birkh\"auser Boston
  Inc., Boston, MA, 1997, with appendices by Maurizio Falcone and Pierpaolo Soravia.
 

\bibitem{BF1}
\newblock M.~Bardi and M.~Falcone,
\newblock An approximation scheme for the minimum time function,
\newblock \emph{SIAM J.~Control Optim.}, \textbf{28} (1990), 950--965.

\bibitem{BF2}
\newblock M.~Bardi and M.~Falcone,
\newblock Discrete approximation of the minimal time function for systems with
  regular optimal trajectories,
\newblock in \emph{Analysis and Optimization of Systems. Proceedings of the 9th
  International Conference Antibes, June 12--15, 1990} (eds. A.~Bensoussan and
  J.~L. Lions), vol. 144 of Lecture Notes in Control and Inform.\ Sci., Springer, Berlin--Heidelberg, 1990, 103--112.
 

\bibitem{BFS}
\newblock M.~Bardi, M.~Falcone and P.~Soravia,
\newblock Numerical methods for pursuit-evasion games via viscosity solutions,
\newblock in \emph{Stochastic and Differential Games}, vol.~4 of Ann.\
  Internat.\ Soc.\ Dynam.\ Games, Birkh\"auser Boston, Boston, MA, 1999, 105--175.
 

\bibitem{BBZ}
\newblock O.~Bokanowski, A.~Briani and H.~Zidani,
\newblock Minimum time control problems for non-autonomous differential
  equations,
\newblock \emph{Systems Control Lett.}, \textbf{58} (2009), 742--746.

\bibitem{BFZ}
\newblock O.~Bokanowski, N.~Forcadel and H.~Zidani,
\newblock Reachability and minimal times for state constrained nonlinear
  problems without any controllability assumption,
\newblock \emph{SIAM J.~Control Optim.}, \textbf{48} (2010), 4292--4316.

\bibitem{CLSW}
\newblock F.~H. Clarke, Y.~S. Ledyaev, R.~J. Stern and P.~R. Wolenski,
\newblock \emph{Nonsmooth Analysis and Control Theory},
\newblock Springer-Verlag, New York, 1998.

\bibitem{CL}
\newblock G.~Colombo and T.~T.~T. Le,
\newblock Higher order discrete controllability and the approximation of the
  minimum time function,
\newblock \emph{Discrete Contin.\ Dyn.\ Syst.}, \textbf{35} (2015), 4293--4322.

\bibitem{CMW}
\newblock G.~Colombo, A.~Marigonda and P.~R. Wolenski,
\newblock Some new regularity properties for the minimal time function,
\newblock \emph{SIAM J.~Control Optim.}, \textbf{44} (2006), 2285--2299.

\bibitem{CNN}
\newblock G.~Colombo, K.~T. Nguyen and L.~V. Nguyen,
\newblock Non-{L}ipschitz points and the {$SBV$} regularity of the minimum time
  function,
\newblock \emph{Calc.\ Var.\ Partial Differential Equations}, \textbf{51}
  (2014), 439--463.

\bibitem{DV}
\newblock B.~D. Doitchinov and V.~M. Veliov,
\newblock Parametrizations of integrals of set-valued mappings and
  applications,
\newblock \emph{J.~Math.\ Anal.\ Appl.}, \textbf{179} (1993), 483--499.

\bibitem{D2F}
\newblock T.~D. Donchev and E.~M. Farkhi,
\newblock Moduli of smoothness of vector valued functions of a real variable
  and applications,
\newblock \emph{Numer.\ Funct.\ Anal.\ Optim.}, \textbf{11} (1990), 497--509.

\bibitem{DF}
\newblock A.~L. Dontchev and E.~M. Farkhi,
\newblock Error estimates for discretized differential inclusions,
\newblock \emph{Computing}, \textbf{41} (1989), 349--358.

\bibitem{F}
\newblock M.~Falcone,
\newblock Numerical {S}olution of {D}ynamic {P}rogramming {E}quations.
  {A}ppendix {A},
\newblock in \emph{Optimal Control and Viscosity Solutions of
  {H}amilton-{J}acobi-{B}ellman Equations} (eds. M.~Bardi and
  I.~Capuzzo-Dolcetta), Systems \& Control: Foundations \& Applications, Birkh\"auser Boston
    Inc., Boston, MA, 1997, 471--504.
 

\bibitem{Ge}
\newblock M.~Gerdts,
\newblock \emph{Optimal control of {ODE}s and {DAE}s},
\newblock de Gruyter Textbook, Walter de Gruyter \& Co., Berlin, 2012.

\bibitem{GlGM}
\newblock A.~Girard, C.~{Le Guernic} and O.~Maler,
\newblock Efficient computation of reachable sets of linear time-invariant
  systems with inputs,
\newblock in \emph{Hybrid Systems: Computation and Control}, vol. 3927 of
  Lecture Notes in Comput.\ Sci., Springer, Berlin, 2006, 257--271.
 

\bibitem{GL}
\newblock L.~Gr{\"u}ne and T.~T.~T. Le,
\newblock A double-sided dynamic programming approach to the minimum time
  problem and its numerical approximation,
\newblock \emph{Applied Numerical Mathematics}, \textbf{121} (2017), 68--81.

\bibitem{HL}
\newblock H.~Hermes and J.~LaSalle,
\newblock \emph{Functional Analysis and Time Optimal Control}, vol.~56 of
  Mathematics in science and engineering,
\newblock Academic Press, New York, 1969.

\bibitem{KV}
\newblock A.~B. Kurzhanski and P.~Varaiya,
\newblock \emph{Dynamics and Control of Trajectory Tubes. Theory and
  Computation},
\newblock Systems \& Control: Foundations \& Applications, Springer,
  Cham--Heidelberg--New York--Dordrecht--London, 2014.

\bibitem{Le}
\newblock T.~T.~T. Le,
\newblock \emph{Results on Controllability and Numerical Approximation of the
  Minimum Time Function},
\newblock PhD thesis, Dipartimento di Matematica, Padova, Italy, 2016.

\bibitem{lG}
\newblock C.~{Le Guernic},
\newblock \emph{Reachability {A}nalysis of {H}ybrid {S}ystems with {L}inear
  {C}ontinuous {D}ynamics},
\newblock PhD thesis, \'{E}cole Doctorale Math\'{e}matiques, Sciences et
  Technologies de l'Information, Informatique, Grenoble, France, 2009, 169 pages.
 

\bibitem{LM}
\newblock E.~Lee and L.~Markus,
\newblock \emph{Foundations of Optimal Control Theory}, SIAM Series in Applied Mathematics,
\newblock John Wiley \& Sons, Inc., New York--London--Sydney, 1967

\bibitem{AM}
\newblock A.~Marigonda,
\newblock Second order conditions for the controllability of nonlinear systems
  with drift,
\newblock \emph{Commun.\ Pure Appl.\ Anal.}, \textbf{5} (2006), 861--885.

\bibitem{PalUrb}
\newblock D.~Pallaschke and R.~Urba\'{n}ski,
\newblock \emph{Pairs of Compact Convex Sets}, vol. 548 of Mathematics and Its
  Applications,
\newblock Kluwer Academic Publishers, Dordrecht, 2002.

\bibitem{P}
\newblock N.~N. Petrov,
\newblock On the {B}ellman function for the time-optimality process problem,
\newblock \emph{Prikl.\ Mat.\ Meh.}, \textbf{34} (1970), 820--826.

\bibitem{Rockaf}
\newblock R.~T. Rockafellar and R.~J.-B. Wets,
\newblock \emph{Variational Analysis}, vol. 317 of Grundlehren der
  Mathematischen Wissenschaften [Fundamental Principles of Mathematical
  Sciences],
\newblock Springer Science \& Business Media, Berlin, 2009.

\bibitem{S-P}
\newblock P.~Saint-Pierre,
\newblock Approximation of the viability kernel,
\newblock \emph{Appl.\ Math.\ Optim.}, \textbf{29} (1994), 187--209.

\bibitem{Tol}
\newblock A.~Tolstonogov,
\newblock \emph{Differential Inclusions in a {B}anach Space}, vol. 524 of
  Mathematics and its Applications,
\newblock Kluwer Academic Publishers, Dordrecht, 2000, translated from the 1986 Russian original and revised by the author.
 

\bibitem{V_int}
\newblock V.~M. Veliov,
\newblock Discrete approximations of integrals of multivalued mappings,
\newblock \emph{C.~R.~Acad.\ Bulgare Sci.}, \textbf{42} (1989), 51--54.

\bibitem{V}
\newblock V.~M. Veliov,
\newblock Second order discrete approximation to linear differential
  inclusions,
\newblock \emph{SIAM J.~Numer.\ Anal.}, \textbf{29} (1992), 439--451.

\bibitem{Wil}
\newblock M.~D. Wills,
\newblock Hausdorff distance and convex sets,
\newblock \emph{J.~Convex Anal.}, \textbf{14} (2007), 109--117.

\bibitem{W}
\newblock P.~R. Wolenski,
\newblock The exponential formula for the reachable set of a {L}ipschitz
  differential inclusion,
\newblock \emph{SIAM J.~Control Optim.}, \textbf{28} (1990), 1148--1161.
\end{thebibliography}
\end{document}